\documentclass[12pt]{article}
\usepackage{graphicx}
\usepackage{amsmath}
\usepackage{amsthm}
\usepackage{graphics}
\usepackage{epsfig}
\usepackage{amssymb}
\usepackage{enumitem}
\usepackage{hyperref}
\usepackage{dsfont}
\usepackage{float}
\usepackage{authblk}
\usepackage[toc,page]{appendix}
\usepackage{xcolor}
\usepackage{setspace}
\usepackage{extarrows}
\usepackage[sort&compress]{natbib}
\usepackage{subfig}
\usepackage{lmodern}
\usepackage{pdflscape}
\usepackage{subfig}
\usepackage{enumitem}
\allowdisplaybreaks

\newcommand{\lb}{\left(}
\newcommand{\rb}{\right)}
\newcommand{\lcb}{\left\{}
\newcommand{\rcb}{\right\}}
\newcommand{\lsb}{\left[}
\newcommand{\rsb}{\right]}
\newcommand\lv{\left\vert}
\newcommand\rv{\right\vert}
\newcommand\rnorm{\right\Vert}
\newcommand\lnorm{\left\Vert}

\let\originalleft\left
\let\originalright\right
\renewcommand{\left}{\mathopen{}\mathclose\bgroup\originalleft}
\renewcommand{\right}{\aftergroup\egroup\originalright}

\newcommand{\Var}{\operatorname{Var}}

\newcommand\as{\operatorname{a. s.}}
\newcommand{\tr}{\operatorname{tr}}

\newcommand{\cond}{\stackrel{\mathcal{D}}{\to}}
\newcommand{\conp}{\stackrel{\mathbb{P}}{\to}}
\newcommand{\conas}{\stackrel{\as}{\to}}

\newcommand{\diag}{\operatorname{diag}}
\newcommand{\spi}{\operatorname{spi}}
\newcommand{\Laplace}{\operatorname{Laplace}}

\newcommand{\bfA}{\mathbf{A}}
\newcommand{\bfB}{\mathbf{B}}

\newcommand{\bfF}{\mathbf{F}}
\newcommand{\bfG}{\mathbf{G}}
\newcommand{\bfH}{\mathbf{H}}
\newcommand{\bfI}{\mathbf{I}}

\newcommand{\bfP}{\mathbf{P}}

\newcommand{\bfR}{\mathbf{R}}

\newcommand{\bfS}{\mathbf{S}}
\newcommand{\bfU}{\mathbf{U}}

\newcommand{\bfZ}{\mathbf{Z}}

\newcommand{\bfx}{\mathbf{x}}
\newcommand{\bfy}{\mathbf{y}}
\newcommand{\bfz}{\mathbf{z}}
\newcommand\CB{\mathcal{B}}
\newcommand\CC{\mathcal{C}}
\newcommand\CF{\mathcal{F}}
\newcommand{\E}{\mathbb{E}}
\newcommand{\N}{\mathbb{N}}
\newcommand{\R}{\mathbb{R}}
\newcommand{\PR}{\mathbb{P}}

\newcommand{\bfSigma}{\mathbf{\Sigma}}
\newcommand{\bfLambda}{\mathbf{\Lambda}}

\newtheorem{theorem}{Theorem}[section]
\newtheorem{lemma}{Lemma}[section]

\newtheorem{corollary}{Corollary}[section]

\newtheorem{remark}{Remark}[section]

\numberwithin{equation}{section}

\allowdisplaybreaks

\begin{document}
\title{A New Two-Sample Test for Covariance Matrices in High Dimensions: U-Statistics Meet Leading Eigenvalues}
\date{\today}
\author[1]{Thomas Lam}
\author[2]{Nina Dörnemann}
\author[1]{Holger Dette}
\affil[1]{Department of Mathematics, Ruhr University Bochum, Germany}
\affil[2]{Department of Mathematics, Aarhus University, Denmark}

\maketitle
 \begin{abstract}
    We propose a two-sample test for covariance matrices in the high-dimensional regime, where the dimension diverges proportionally to the sample size. Our hybrid test combines a Frobenius-norm-based statistic as considered in \cite{li_and_chen_2012} with the leading eigenvalue approach proposed in  \cite{zhang2022asymptotic}, making it sensitive to both dense and sparse alternatives. The two statistics are combined via Fisher's method, leveraging our key theoretical result: a joint central limit theorem showing the asymptotic independence of the leading eigenvalues of the sample covariance matrix and an estimator of the Frobenius norm of the difference of the two population covariance matrices, under suitable signal conditions. 
    The level of the test can be controlled asymptotically, and we show consistency against certain types of both sparse and dense alternatives. A comprehensive numerical study confirms the favorable performance of our method compared to existing approaches.
  \end{abstract}

  AMS subject classification:  15A18, 60F05, 62H15 

Keywords and phrases: Central limit theorems, covariance matrices, high-dimensional statistics, hypothesis testing.

\section{Introduction}

Testing the equality of covariance matrices for two given samples is a fundamental problem in multivariate statistics as the assumption of equal covariance matrices appears in a range of problems, including two-sample tests for mean vectors \citep{anderson2003, bai1996effect}. 
Traditional methods, however, do not meet the challenges of modern settings where the dimension $p$ is comparable to the sample size $n$ \citep{jiangyang2013}. The wide availability of data sets with a huge amount of parameters has spurred the development of new testing procedures that are specifically tailored to the high-dimensional framework. For instance, many authors have investigated adaptations of likelihood ratio tests to the high dimensional setup \citep{dettedoernemann2020, doernemann2022, jiangyang2013, qi_et_al_2019}. 
While likelihood ratio-type tests are known to be very powerful, they are not applicable if the dimension exceeds the sample size, as they rely on the log-determinant of the sample covariance matrix. Several alternative approaches have been proposed that remain valid in such cases, see, e.g., \cite{li_and_chen_2012} on a test based on the Frobenius distance, \cite{ding2024two} on a random-matrix theory approach for ultra-high dimensional data, \cite{SRIVASTAVA2014289} on a two-sample test for normally distributed data. Furthermore, tests for sparse alternatives, where the number of nonzero elements of the difference of two covariance matrices is small, have been proposed by \cite{CaiLiuXia2013, YangPan2017, ZhengLinGuoYin2020}. 

Moreover, \cite{zhang2022asymptotic} have developed a test based on linear spectral statistics and the individual leading eigenvalues of the sample covariance matrices.  
As it is well-known in random matrix theory, if the dimension is of the same order as the sample size, a signal in the population eigenvalue is only carried by its sample version if the signal is above a critical threshold \citep{bbp}.
Thus, to distinguish noise from signal in the high-dimensional setup, the population eigenvalues are assumed to be above the critical phase transition threshold. In this case, their sample counterpart satisfy a central limit theorem \citep{paul2007asymptotics, bai2008central, li_et_al_2020} which facilitates hypothesis testing \citep{zhang2022asymptotic}. 
\medskip

\noindent\textbf{Problem formulation.}
A comprehensive list of definitions and assumptions is provided in Section \ref{sec_2}, but the main ideas can already be discussed here with minimal notation.
Let $\bfx_1^{(i)}, \ldots, \bfx_n^{(i)}$ be two independent samples consisting of i.i.d. $p$-dimensional random vectors  with covariance matrix $\bfSigma_n^{(i)} \in \R^{p\times p}$, $i\in \{1,2\}.$ 
We are interested in testing 
\begin{align} \label{eq_hypothesis}
    \mathsf{H}_0: \bfSigma_n^{(1)} = \bfSigma_n^{(2)} 
    \textnormal{ vs. }
    \mathsf{H}_1: \bfSigma_n^{(1)} \neq \bfSigma_n^{(2)}
\end{align}
in the setting where $p$ is asymptotically proportional to $n.$ 
Unlike in the traditional low-dimensional setting, a key characteristic of high-dimensional data is the fundamental distinction between dense and sparse signals. Many existing methods are optimized for one of these regimes and may fail to detect signals from the other. In practice, however, the structure of the signal is unknown, and  typically tests that are sensitive to diverse types of deviations from the null hypothesis are desirable. 
In this work, we address this challenge and concentrate on the case where the dimension is asymptotically proportional to the sample size.  

An initial step in this direction has already been taken by \cite{zhang2022asymptotic}, who have proposed a test statistic that combines differences in linear spectral statistics of the two sample covariance matrices with the differences in their leading eigenvalues.
More precisely, letting $\bfS_n^{(i)}$ denote the sample covariance matrix of the $i$th sample and $\lambda_1(\bfS_n^{(i)})$ the leading sample eigenvalue, their test statistic takes on the form
\begin{align*}
 \frac{1}{\hat\zeta_2^2} \lb  \tr (\bfS_n^{(1)}) +  \tr (\bfS_n^{(1)^2}) - \tr (\bfS_n^{(2)}) - \tr (\bfS_n^{(2)^2})  \rb ^2 +  \frac{n}{\hat\zeta_1^2} \lb  \lambda_1(\bfS_n^{(1)}) - \lambda_1(\bfS_n^{(2)}) \rb^2 :=  W_{1,n} + W_{2,n}  ,
\end{align*}
where $\hat\zeta_1$ and $\hat\zeta_2$ are estimators specified in the aforementioned work. 
Their approach has aimed to capture two types of deviations: the first part $W_{1,n}$ is sensitive to dense changes that affect a large portion of the spectrum, while the second part $W_{2,n}$ targets changes in the leading eigenvalues. 
However, as the authors themselves note, $W_{2,n}$ is not a consistent estimator for a matrix norm $\| \bfSigma_n^{(1)} - \bfSigma_n^{(2)}\|$, and thus the test may fail to detect a broad class of dense alternatives. 
\medskip

\noindent\textbf{Main contributions.}
In this work, we propose a hybrid test which takes advantage from both the Frobenius-norm approach of \cite{li_and_chen_2012} and the leading eigenvalue approach in \cite{zhang2022asymptotic}, and is thus sensitive to both sparse and dense alternatives. We combine the two detectors via Fisher's combination test 
\citep{littell1971asymptotic, fisher1950statistical}. This rests on our main theoretical result saying that the two test statistics are asymptotically independent. 

Our main contributions towards developing a combined test for \eqref{eq_hypothesis} can be summarized as follows:

\begin{enumerate}
    \item We show that the leading sample eigenvalues and the $U-$statistics for the Frobenius norm satisfy a joint central limit theorem, supposing that the leading population eigenvalues are above the critical detection threshold. 
    In particular, these two components are asymptotically independent, an important question that has remained unresolved since \cite{zhang2022asymptotic}. 
    \item We estimate all unknown parameters in the above CLT and propose a Fisher-combined test. This test keeps its nominal  level asymptotically. Moreover, as it rests on the Frobenius norm and the leading eigenvalues, the proposed test is consistent for a broad class of both dense and sparse alternatives.     
    \item We generalize the test to include multiple leading eigenvalues, and provide the corresponding statistical guarantees under $\mathsf{H}_0$ and $\mathsf{H}_1.$ 
    \item The theoretical guarantees are illustrated by comprehensive numerical experiments, covering different data-generating distributions and alternative scenarios. Our tests demonstrate favorable performance compared to existing methods.
\end{enumerate}

\noindent \textbf{Outline.} The remaining part of this paper is structured as follows. 
The new tests and their theoretical guarantees are presented in Section \ref{sec_2}. In Section \ref{sec_sim}, we investigate the finite-sample performance of our tests and compare them to existing methods. The proofs of the main results are placed in Section \ref{sec_proofs}, while several auxiliary results and corresponding proofs are stated in Appendix \ref{sec_appendix}. An overview of the main steps in the proofs can be found in Section \ref{sec_proof_outline}.
Additional numerical results can be found in Section \ref{sec_appendix_simulation}.

\newpage

\section{Testing the equality of spiked covariance matrices} \label{sec_2}

We define the empirical covariance matrix of the sample $\bfx_1^{(i)}, \ldots, \bfx_n^{(i)}$ by 
\begin{align*}
    \bfS_n^{(i)} &= \frac{1}{n} \sum_{j=1}^n \lb \bfx_j^{(i)} - \overline{\bfx}^{(i)} \rb \lb \bfx_j^{(i)} - \overline{\bfx}^{(i)}  \rb^\top ,
\end{align*}
where
    \begin{align*}
        \overline{\bfx}^{(i)} &= \frac{1}{n} \sum_{j=1}^n \bfx_j^{(i)}
    \end{align*}
denotes the sample mean. In this work, the eigenvalues of $\bfS_n^{(i)}$ and its population counterpart $\bfSigma_n^{(i)}$ will be of major interest. To this end,  we denote  by $\lambda_1(\bfA) \geq \ldots \geq \lambda_p (\bfA)$ the ordered eigenvalues of any symmetric $p\times p$ matrix $\bfA$.

\subsection{Methodology}
To detect a large class of alternatives for \eqref{eq_hypothesis} in the spiked covariance model, we propose a hybrid test statistic whose main components will be asymptotically independent. To simplify notation, we first propose a statistic that takes into account only the largest eigenvalue of $\bfS_n^{(1)}$ and $\bfS_n^{(2)}$, respectively. Later, we also present a generalization that incorporates several leading eigenvalues. 
\medskip

\textbf{Frobenius norm.} The first statistic measures the distance between $\bfSigma_n^{(1)}$ and $\bfSigma_{n}^{(2)}$ with respect to the Frobenius norm. In high dimensions, the estimation of 
\begin{align*}
    \tr \lb \bfSigma_n^{(1)} - \bfSigma_n^{(2)} \rb^2 
    = \tr \lb  \bfSigma_n^{(1)} \rb^2 + \tr \lb  \bfSigma_n^{(2)} \rb^2  - 2 \tr \lb  \bfSigma_n^{(1)}\bfSigma_n^{(2)} \rb
\end{align*}
cannot be realized through the corresponding sample version, but requires a sophisticated analysis. 
To this end, we follow \cite{li_and_chen_2012} who have proposed individual estimators based on $U-$statistics for each summand, which are defined as
    \begin{align*}
        B_n ^{(i)} & =  \frac{1}{n\lb n-1\rb}\sum_{j\neq k}^{n} \lb \bfx_j^{(i)^\top} \bfx_k^{(i)}\rb^2-\frac{2}{n(n-1)(n-2)}\sum_{j,k,\ell}^{\star}\bfx_j^{(i)^\top}\bfx_k^{(i)}\bfx_k^{(i)^\top}\bfx_{\ell}^{(i)}
        \\
        &\phantom{{}={}}+\frac{1}{n(n-1)(n-2)(n-3)}\sum_{j, k , \ell, m}^{\star}\bfx_j^{(i)^\top}\bfx_k^{(i)}\bfx_{\ell}^{(i)^\top}\bfx_m^{(i)}
        \\ 
        C_n & = \frac{1}{n^2}\sum_{j=1}^{n}\sum_{k=1}^{n} \lb \bfx_j^{(1)^\top}\bfx_k^{(2)}\rb^2-\frac{1}{n^2(n-1)}\sum_{j, k}^{\star}\sum_{\ell=1}^{n} \lb\bfx_j^{(1)^\top}\bfx_k^{(2)}\bfx_k^{(2)^\top}\bfx_{\ell}^{(1)}+\bfx_j^{(2)^\top}\bfx_k^{(1)}\bfx_k^{(1)^\top}\bfx_{\ell}^{(2)}\rb
        \\
        &\phantom{{}={}}+\frac{1}{n^2(n-1)(n-2)}\sum_{j, k}^{\star}\sum_{\ell, m}^{\star} \bfx_j^{(1)^\top}\bfx_k^{(2)}\bfx_{\ell}^{(1)^\top}\bfx_{m}^{(2)},
    \end{align*}
    where $\sum\limits^{\star}$ denotes the summation over mutually distinct indices.
Then, we define the statistic
\begin{align*}
    T_{n,1} & = \frac{1}{\hat\sigma_{n,1}} \lb  B_n ^{(1)} + B_n ^{(2)} - 2 C_n \rb  ,
\end{align*}
where $\hat{\sigma}_{n,1}$ is the estimator of the variance $\Var(B_n^{(1)}+ B_n^{(2)}-2C_n)$ under the null hypothesis. A detailed description of this estimator is deferred to Section \ref{sec_proof_consistent_estimators}.

\smallskip

\textbf{Leading eigenvalues.} The second part of the statistic is designed to detect differences in the leading population eigenvalues. 
More precisely, our statistic is build on the difference  $\lb \lambda_1(\bfS_n^{(1)}) - \lambda_1(\bfS_n^{(2)}) \rb$ between the largest sample eigenvalues corresponding to the two samples and thus requires an estimate $\hat\sigma_{n,2}^2$ of the variance $\Var(\sqrt{n}  (\lambda_1(\bfS_n^{(1)}) - \lambda_1(\bfS_n^{(2)})) )$. A detailed description of the estimator $\hat\sigma_{n,2}^2$ is deferred to Section \ref{sec_proof_consistent_estimators}. 
Then, the second part of our statistic is given by 
\begin{align*}
    T_{n,2} & = \frac{\sqrt{n}}{\hat\sigma_{n,2}} \lb \lambda_1(\bfS_n^{(1)}) - \lambda_1(\bfS_n^{(2)}) \rb.
\end{align*}

\smallskip 

\noindent\textbf{Combined test statistic.}
We will show later (see Theorem \ref{thm_asympt_ind}) that the statistics $T_{n,1}$ and $T_{n,2}$ are asymptotically standard normal distributed and independent. Thus, we may use Fisher's combination test \citep{littell1971asymptotic, fisher1950statistical} to aggregate the asymptotically independent $p$-values
\begin{align}\label{eq_def_p_value}
    p_{n,1}= 1- \Phi(T_{n,1}), \quad p_{n,2} 
    =2(1-\Phi\lb \lv T_{n,2} \rv \rb).
\end{align}
Here, $\Phi$ denotes the c.d.f. of the standard normal distribution. 
Then, our Fisher-combined test statistic is given by 
\begin{align*}
    T_{n, FC}=-2\log(p_{n,1}) -2\log(p_{n,2}) .
\end{align*}
 The null hypothesis $\mathsf{H}_{0}$ is rejected whenever 
\begin{align}
    \label{eq_rejection_region}
    T_{n, FC}>q_{1-\alpha},
\end{align}
where  $q_{1-\alpha}$ denotes the $(1-\alpha)$-quantile of the $\chi_4^2$-distribution.

\subsection{Theoretical results}
To formulate our statistical guarantees, we need the following definitions. 
\medskip

\noindent\textbf{Spectral distribution.}
For any symmetric $p\times p$ matrix $\mathbf{A}$, we denote the empirical spectral distribution of $\bfA$ by 
 \begin{align*}
        F^{\bfA}=\frac{1}{p}\sum_{j=1}^{p}\delta_{\lambda_{j}(\bfA)},
    \end{align*}
    where $\delta_{\cdot}$ denotes the Dirac measure. We call the weak limit of $F^{\bfA}$ as $p\to\infty$, if existent, limiting spectral distribution of $\bfA$. 
\medskip

\noindent\textbf{Supercritical eigenvalues.}
As explained in the introduction, differences between the leading eigenvalues of $\bfSigma_n^{(1)}$ and $\bfSigma_n^{(2)}$ are only carried by their sample versions if the largest population eigenvalues are well-separated from the bulk eigenvalues. 
The magnitude of an individual population eigenvalue, compared to the bulk, is generally captured by the derivative of the function
\begin{align} \label{eq_def_psi}
    \psi^{(i)} (\alpha ) & = \alpha + y \alpha \int \frac{t}{\alpha - t} \, dH^{(i)}(t)
    , \quad i=1,2,
\end{align}
where $H^{(i)}$ denotes the limiting spectral distribution of the population covariance matrix $\bfSigma_n^{(i)}$ as $p,n \to \infty$ such that $p/n \to y >0$, and $\alpha$ lies outside the support of $H^{(i)}$.
If, for some $1 \leq j \leq p$, the derivative $\psi^{(i)^\prime} (\lambda_j(\bfSigma_n^{(i)}))>0$, then $\lambda_j(\bfSigma_n^{(i)})$ is called a supercritical eigenvalue, where $\psi^{(i)^\prime}$ is given by
\begin{align}\label{eq_def_derivative_psi}
\psi^{\lb i\rb^\prime} \lb \alpha\rb = 1- y\int\frac{t^2}{\lb \alpha-t\rb^2}\, dH^{\lb i\rb}\lb t\rb.    
\end{align}
In this case, the corresponding sample eigenvalue is detached from the bulk eigenvalues and concentrates around $\psi^{(i)}(\lambda_j(\bfSigma_n^{(i)})) $ (see, for example, \cite{zhang2022asymptotic}). In the context of our two-sample problem, this implies that the signal can be detected.
\\ ~ \medskip \\
\noindent\textbf{Assumptions.}

\begin{enumerate}[label=(A\arabic*)]
    \item \label{ass_p_n}
    As $n \to \infty$, we have $p=p\lb n\rb \to \infty$ and $\frac{p}{n}=y_n \to y \in \lb 0, \infty\rb$.
    \item \label{ass_random_vectors}For $i=1, 2$ and $j=1, \ldots, n$, we assume that the random variables $\bfx_j^{(i)}$ can be represented as 
    \begin{align*}
    \bfx_j^{(i)}=\bfSigma_n^{(i)^{1/2}}\bfz_j^{(i)}+\boldsymbol\mu^{(i)},
    \end{align*}
    where $\bfSigma_n^{(i)^{1/2}}$ is the nonnegative definite square root of the covariance matrix $\bfSigma_n^{(i)}$, $\boldsymbol\mu^{(i)}$ is $p$-dimensional nonrandom vector, and $\bfz_j^{(i)}=\lb z_{1,j}, \ldots, z_{p,j}\rb^\top$ is a $p$-dimensional random vector. Moreover, we assume that the double array
    \begin{align*}
    \lcb z_{j, k}^{(i)} : j=1, \ldots, p, k=1,\ldots ,n\rcb
    \end{align*}
    consists of i.i.d. random variables with 
    \begin{align*}
    \E \lsb z_{1, 1}^{(i)}\rsb = 0, \E \lsb\lb z_{1, 1}^{(i)}\rb^2 \rsb= 1 \quad\text{and}  \quad
    \E\lsb\lb z_{1, 1}^{(i)}\rb^8\rsb < \infty.
    \end{align*}
     \item\label{ass_population_lsd} For $i=1, 2$, $\bfSigma_n^{(i)}$ admits a nonrandom limiting spectral distribution $H^{(i)}$. Also, assume that 
     $\lambda_p(\bfSigma_n^{(i)})>C$ ($i=1,2$) for some constant $C>0$ and all sufficiently large $n.$
     \item \label{ass_leading_spiked_ev}Suppose that the largest eigenvalues $\lambda_1(\bfSigma_n^{(i)})=\alpha_{1}^{(i)}$, $i=1,2$, are supercritical and do not depend on $n$. Furthermore, we assume that $\inf_{n\in\N} \lb  \alpha_1^{(i)} - \lambda_{2}(\bfSigma_n^{(i)}) \rb > 0.$  
   \end{enumerate}

\begin{remark}
    {\rm 
    Assumptions \ref{ass_p_n} -  \ref{ass_leading_spiked_ev} are  standard in random matrix theory; see, for example, \cite{zhang2022asymptotic, li_et_al_2020, bai2004, liu2023clt}. In contrast to \cite{zhang2022asymptotic}, however, we require an $8$th moment assumption rather than a $4$th moment, because our test  statistic also involves estimates of the Frobenius norm of the difference of the two population covariance matrices \citep[see also][]{li_and_chen_2012}. Moreover, we impose the assumption that the smallest eigenvalues of the two covariance matrices are uniformly bounded away from $0$ in \ref{ass_population_lsd}. While this condition is purely technical and could potentially be relaxed, doing so would require additional technical developments beyond the scope of this work. 
    }
\end{remark}
    
\begin{theorem} \label{thm_asympt_ind}
  Suppose that assumptions \ref{ass_p_n} - \ref{ass_leading_spiked_ev} are satisfied and that the null hypothesis in \eqref{eq_hypothesis} holds true. Then, it follows that 
  \begin{align*}
      (T_{n,1}, T_{n,2})^\top  \cond \mathcal{N}_2 (\mathbf{0}_2, \bfI_2), \quad n\to\infty,
  \end{align*}
  where $\bfI_2$ denotes the identity matrix of dimension $2\times 2$.
\end{theorem}
\begin{remark}
    {\rm The asymptotic marginal distributions of $T_{n,1}$ and $T_{n,2}$ have been established in \cite{li_and_chen_2012} and \cite{zhang2022asymptotic}, respectively. 
    Theorem \ref{thm_asympt_ind} pioneers in establishing the joint distribution of the random vector $(T_{n,1}, T_{n,2})^\top $.
    }
\end{remark}
\begin{theorem}\label{thm_fisher_combi_asympt}
    Suppose that assumptions \ref{ass_p_n} - \ref{ass_leading_spiked_ev} are satisfied and that the null hypothesis in \eqref{eq_hypothesis}   holds true. Then, it follows that 
    \begin{align*}
        T_{n, FC} \cond \chi_4^2, \quad n\to \infty ~.
    \end{align*}
\end{theorem}
Next, we formulate statistical guarantees for the proposed test under the null hypothesis and the alternative. Theorem \ref{thm_fisher_combi_asympt} immediately implies that the level of the test \eqref{eq_rejection_region} can be controlled asymptotically at a prescribed level $\alpha \in (0,1).$
Moreover, under the alternative $\mathsf{H}_1$, the test in \eqref{eq_rejection_region} is consistent, if either the difference between $\bfSigma_{n}^{(1)}$ and $\bfSigma_{n}^{(2)}$ with respect to the squared Frobenius norm or the difference between the leading population eigenvalue is large. To quantify the latter, we define the functions $\psi_{n}^{(i)}$ as a finite-sample analogue of $\psi^{(i)}$ in \eqref{eq_def_psi} with $(H^{(i)},y)$ replaced by $(F^{\bfLambda_{P,n}^{(i)}},y_n)$, where the diagonal matrix 
\begin{align*}
    \bfLambda_{P,n}^{(i)}=\diag \lb\lambda_{2}(\bfSigma_{n}^{(i)}),\ldots,  \lambda_P(\bfSigma_{n}^{(i)})\rb, \quad i=1,2,
\end{align*}
consists of the remaining eigenvalues of the matrix  $\bfSigma_n^{(i)}$.
More precisely, we define
\begin{align*} 
    \psi_n^{(i)} (\alpha ) & = \alpha + y_n \alpha \int \frac{\lambda}{\alpha - \lambda} d F^{\bfLambda_{P,n}^{(i)}}(\lambda)
    , \quad i=1,2,
\end{align*}
where $\alpha$ lies outside the support of $F^{\bfLambda_{P,n}^{(i)}}$.

In the following corollary, we summarize the performance of the test in \eqref{eq_rejection_region} the under null and alternative hypothesis. 
\begin{corollary}\label{cor_fisher_combi_1_norm_performance}
    Suppose that assumptions \ref{ass_p_n} - \ref{ass_leading_spiked_ev} are satisfied. 
    
    \begin{enumerate}
        \item    Under $\mathsf{H}_0,$ it holds 
         \begin{align*}
    \PR\lb T_{n, FC} > q_{1-\alpha}\rb \to \alpha, \quad n\to \infty ~.
\end{align*}

  \item  Suppose that 
   \begin{align*}
        \tr\lb \lcb \bfSigma_n^{(1)}-\bfSigma_n^{(2)}\rcb^2 \rb \to \infty, \quad n \to \infty,
    \end{align*}
    or that
\begin{align}
        \sqrt{n}\lv \psi_{n}^{(1)}\lb \alpha_1^{(1)}\rb- \psi_n^{(2)}\lb \alpha_1^{(2)}\rb \rv\to \infty, \quad n\to \infty ~.\nonumber
    \end{align}
    Then, under the alternative $\mathsf{H}_1$ we have 
    \begin{align*}
        \PR \lb T_{n, FC} > q_{1-\alpha}\rb \to 1, \quad n\to \infty ~.
    \end{align*}
    \end{enumerate}
\end{corollary}
In the following, we demonstrate how to adjust the proposed test to account for multiple supercritical eigenvalues. This requires a refinement of our assumptions. 
The following assumption says that the $K$ leading eigenvalues of $\bfSigma_n^{(i)}$ are supercritical, where the number  $K\in \N$ does not depend on the sample size  $n$. To distinguish the outlier eigenvalues from the bulk eigenvalues, we introduce the notation
    \begin{align*}
    \alpha_k^{(i)} =\lambda_k(\bfSigma_n^{(i)}), \quad 1 \leq k \leq K, ~ i= 1,2,
    \end{align*} 
    and make the following additional assumption.
    
\begin{enumerate}[label=(A\arabic*), resume]
    \item \label{ass_supercritical_ev_generalized}
    Suppose that for $i=1,2$, the eigenvalues $\alpha_1^{(i)}, \ldots, \alpha_{K}^{(i)} $ are supercritical, do not depend on  $n$, and satisfy 
    \begin{align*}
    \alpha_1^{(i)} > \ldots > \alpha_{K}^{(i)}.
    \end{align*}
    Furthermore, we assume that $\inf_{n\in \N} \lb \alpha_K^{(i)} - \lambda_{K+1}(\bfSigma_n^{(i)})\rb > 0.$
\end{enumerate}
Finally, the last assumption ensures the existence of the asymptotic variances and covariances of the supercritical sample eigenvalues. For this purpose, we consider the singular value decomposition
    \begin{align*}
        \bfSigma_{n}^{(i)} = \bfU_n^{(i)} \diag (\bfSigma_n^{(i)}) \lb  \bfU_n^{(i)} \rb ^\top , \quad i=1,2,
    \end{align*}
    where $\bfU_n^{(i)}$ denotes a $\R^{p\times p}$ orthogonal matrix. In addition,  we denote by  $\mathbf{u}_{k,n}^{(i)}=(u_{1,k,n}^{(i)}, \ldots, u_{p,k,n}^{(i)})^\top$  the $k$th column of the matrix $\bfU_{ n}^{(i)}$ ($1 \leq k \leq p$).
\begin{enumerate}[label=(A\arabic*), resume]
    \item \label{ass_asympt_var_cov} For distinct $k, \ell=1, \ldots K$, the following limits exist
    \begin{align*}
        \lim_{n\to \infty} \sum_{j=1}^{p} u_{j,k,n}^{(i)^4}, \quad \text{and} \quad \lim_{n\to \infty} \sum_{j=1}^{p} u_{j,k,n}^{(i)^2}u_{j, \ell, n}^{(i)^2}.
    \end{align*}

  \end{enumerate}
  The asymptotic variances $\sigma_{\spi, k}^{(i)^2}$ and covariances $ \sigma_{\spi, k,\ell}^{(i)}$ of the supercritical sample eigenvalues (after a $\sqrt{n}$-scaling) are given by 
  \begin{align}
  \label{det11a}
        \sigma_{\spi, k}^{(i)^2}&=  \lim_{p\to \infty} \sigma_{\spi, k ,n}^{(i)^2} = \lim_{p\to \infty} \lb \gamma_4^{(i)}-3\rb \alpha_k^{(i)^2}\lcb\psi^{(i)^{\prime}}\lb\alpha_k^{(i)}\rb\rcb^2\sum_{j=1}^{p}u_{j, k, n}^{(i)^4}+2\alpha_k^{(i)^2}\psi^{(i)^\prime}\lb \alpha_k^{(i)}\rb,
        \\
          \label{det11b}
        \sigma_{\spi, k,\ell}^{(i)}&=\lim_{p\to \infty} \sigma_{\spi, k, \ell ,n}^{(i)^2} =  \lim_{p\to \infty}\lb \gamma_4^{(i)}-3\rb\alpha_k^{(i)}\alpha_{\ell}^{(i)}\psi^{(i)^\prime}\lb \alpha_k^{(i)}\rb\psi^{(i)^\prime}\lb \alpha_{\ell}^{(i)}\rb\sum_{j=1}^{p} u_{j, k, n}^{(i)^2}u_{j,\ell,n}^{(i)^2}.
    \end{align}
    Note that assumption \ref{ass_asympt_var_cov} ensures that the limits in \eqref{det11a} and \eqref{det11b} exist. Moreover, we briefly note that $\sigma_{\spi,k,n}^{(i)^2}$ is positive.
More details can be found in Lemma \ref{lem_lower_bounded_variances} in the Appendix.
 To detect differences in the $m\in \lcb 1, \ldots, K\rcb$ leading population eigenvalues, the second part of the statistic is build on the differences between the $m$ largest sample eigenvalues corresponding to the two samples. To be more precise, the second part of our statistic is given by
  \begin{align*}
  {T}_{n,2,m}= \sqrt{n} \sum_{j=1}^{m} \lv \lambda_{j}(\bfS_n^{(1)})-\lambda_{j}(\bfS_n^{(2)})\rv.
  \end{align*}
 In the following, we present a generalization of Theorem \ref{thm_asympt_ind} involving the $K$ leading eigenvalues of $\bfS_n^{(1)}$ and $\bfS_n^{(2)}.$
 To this end, for $1\leq m \leq K$, we define 
  \begin{align*}
      \tilde{\mathbf{T}}_{n,2,m} = \sqrt{n}\lb \lambda_{1}(\bfS_n^{(1)})-\lambda_{1}(\bfS_n^{(2)}), \ldots, \lambda_{m}(\bfS_n^{(1)})-\lambda_{m}(\bfS_n^{(2)})\rb.
  \end{align*}
  \begin{theorem}\label{thm_generalized_asympt_ind}
   Suppose that assumptions \ref{ass_p_n} - \ref{ass_population_lsd}, \ref{ass_supercritical_ev_generalized} and \ref{ass_asympt_var_cov} are satisfied and that the null hypothesis in \eqref{eq_hypothesis} holds true. Then  it follows for $m\in \lcb 1, \ldots, K\rcb$ that 
      \begin{align*}
          (T_{n,1}, \tilde{\mathbf{T}}_{n,2,m})^\top \cond \mathcal{N}_{m+1}(\mathbf{0}_{m+1}, \bfSigma_m), \quad n\to \infty,
      \end{align*}
     where the matrix   $\bfSigma_m  \in \R^{(m+1)\times (m+1)}$ is  defined by 
      \begin{align*}
          \bfSigma_m = \begin{pmatrix}
        1 & 0
        \\
        0 & \bfSigma_{E,m}
        \end{pmatrix},
      \end{align*}
     and the elements $(\bfSigma_{E,m})_{k,\ell}$ of the matrix  
     $\bfSigma_{E,m} \in \R^{m \times m}$ are given by \begin{align*}
		(\bfSigma_{E,m})_{k,\ell}=\begin{cases}
			\sigma_{\spi,k}^{(1)^2} + \sigma_{\spi,k}^{(2)^2}& \text{if } k=\ell
			\\
			\sigma_{\spi,k,\ell}^{(1)} + \sigma_{\spi,k,\ell}^{(2)} & \text{if }k\neq \ell
		\end{cases} ~.
	\end{align*}
  \end{theorem}
  Similar to Theorem \ref{thm_asympt_ind},  Theorem \ref{thm_generalized_asympt_ind} shows that the test statistics $T_{n,1}$ and ${T}_{n,2,m}$ are asymptotically independent. 
  Thus, we may use Fisher's combination test and define the test statistics
  \begin{align*}
      T_{n, FC,m} = -2\log(p_{n,1}) - 2\log(p_{n,2,m}),
  \end{align*}
  where $p_{n,1}$ is defined in \eqref{eq_def_p_value} and
  \begin{align}\label{eq_def_generalized_p_value_spi}
      p_{n,2,m}=  \PR( \lnorm \mathbf{W} \rnorm_1 \geq T_{n,2,m} \mid \mathsf{H}_{0}),
  \end{align}
  where $\lnorm . \rnorm_1$ refers to the $L^1$-norm and $\mathbf{W}\sim \mathcal{N}_m (\mathbf{0}_m, \bfSigma_{E,m})$, denotes the $p$-value with respect to the test statistic $T_{n,2,m}$.
  \begin{remark}
  {\rm 
      We emphasize that in practice the covariance matrix $\bfSigma_{E,m}$ is unknown. Thus, to compute the $p$-value $p_{n,2,m}$, it is necessary to give a consistent estimator for $\bfSigma_{E,m}$. For this purpose, we estimate each element $(\bfSigma_{E,m})_{k,\ell}$ consistently by $(\hat{\bfSigma}_{E,m,n})_{k,\ell}$, where
      \begin{align}\label{eq_def_estimator_covariance_matrix_entries}
          (\hat{\bfSigma}_{E,m,n})_{k,\ell}=\begin{cases}
			\hat{\sigma}_{\spi,k,n}^{(1)^2} + \hat{\sigma}_{\spi,k,n}^{(2)^2}& \text{if } k=\ell
			\\
			\hat{\sigma}_{\spi,k,\ell,n}^{(1)} + \hat{\sigma}_{\spi,k,\ell,n}^{(2)} & \text{if }k\neq \ell
		\end{cases} ~.
      \end{align}
      A detailed description of the estimators $\hat{\sigma}_{\spi,k,n}^{(i)^2}$ and $\hat{\sigma}_{\spi,k,\ell,n}^{(i)}$ is deferred to Section \ref{sec_proof_consistent_estimators}.
      }
  \end{remark}
  \begin{theorem}\label{thm_generalized_fisher_combi_asympt}
      Suppose that assumptions \ref{ass_p_n} - \ref{ass_population_lsd}, \ref{ass_supercritical_ev_generalized} and \ref{ass_asympt_var_cov} are satisfied and that the null hypothesis in \eqref{eq_hypothesis} holds true. Then it follows for  $m\in \lcb 1,\ldots, K\rcb$, that
      \begin{align*}
      T_{n, FC,m} \cond \chi_4^2, \quad n\to \infty~.
      \end{align*}
  \end{theorem}
  Based on Theorem \ref{thm_generalized_fisher_combi_asympt}, we propose to reject the null hypothesis $\mathsf{H}_{0}$ if
  \begin{align}\label{eq_generalized_rejection_region}
  T_{n, FC,m} > q_{1-\alpha},
  \end{align}
   where $q_{1-\alpha}$ denotes the $(1-\alpha)$-quantile of the $\chi_4^{2}$-distribution. 
  Similar to our first proposed test \eqref{eq_rejection_region}, the level of the test \eqref{eq_generalized_rejection_region} can be controlled asymptotically. Moreover, the test \eqref{eq_generalized_rejection_region} is consistent under the alternative $\mathsf{H}_1$. Since the first component of $T_{n, FC}$ and $T_{n, FC,m}$ are identical, it remains that the test \eqref{eq_generalized_rejection_region} powerful, if the difference between $\bfSigma_n^{(1)}$ and $\bfSigma_n^{(2)}$ with respect to the squared Frobenius norm is large. In particular, due to the second part of $T_{n, FC,m}$ the test in $\eqref{eq_generalized_rejection_region}$ is consistent as well, if the difference between one of the $m$ leading population eigenvalues is large. To quantify this power-consistency, we define the generalized functions $\psi_{n,K}^{(i)}$ of $\psi_{n}^{(i)}$.
  Let \begin{align*}
      \bfLambda_{P,n,K}^{(i)} = \diag\lb \lambda_{K+1}(\bfSigma_n^{(i)}), \ldots, \lambda_{p}(\bfSigma_n^{(i)})\rb, \quad i=1,2, 
  \end{align*} 
  be the matrix consisting the remaining eigenvalues of $\bfSigma_n^{(i)}$. Then, we define \begin{align}\label{eq_def_generalized_psi_n}
      \psi_{n,K}^{(i)}(\alpha) = \alpha + y_n \alpha \int \frac{\lambda}{\alpha -\lambda} dF^{\bfLambda_{P,n,K}^{(i)}}(\lambda), \quad i=1,2,
  \end{align}
  where $\alpha$ lies outside the support of $F^{\bfLambda_{P,n,K}^{(i)}}$.
  Finally, we conclude in the following corollary the performance of the test in \eqref{eq_generalized_rejection_region} under null and alternative hypothesis. 
  \begin{corollary}\label{cor_generalized_fisher_combi_1_norm_performance}
      Suppose that assumptions \ref{ass_p_n} - \ref{ass_population_lsd}, \ref{ass_supercritical_ev_generalized} and \ref{ass_asympt_var_cov} are satisfied and let $m\in \lcb 1, \ldots, K \rcb$.
      \begin{enumerate}
        \item    Under the null hypothesis $\mathsf{H}_0,$ it holds 
         \begin{align*}
    \PR\lb T_{n, FC,m} > q_{1-\alpha}\rb \to \alpha, \quad n\to \infty ~.
\end{align*}

  \item  Suppose that 
   \begin{align}\label{eq_frobenius_alternative}
        \tr\lb \lcb \bfSigma_n^{(1)}-\bfSigma_n^{(2)}\rcb^2 \rb \to \infty, \quad n \to \infty,
    \end{align}
    or for at least one $1\leq j \leq m$ that
\begin{align}\label{eq_generalized_eigenvalues_alternative}
        \sqrt{n}\lv \psi_{n,K}^{(1)}\lb \alpha_j^{(1)}\rb- \psi_{n,K}^{(2)}\lb \alpha_j^{(2)}\rb \rv\to \infty, \quad n\to \infty ~.
    \end{align}
    Then, under the alternative $\mathsf{H}_1$ we have 
    \begin{align*}
        \PR \lb T_{n, FC,m} > q_{1-\alpha}\rb \to 1, \quad n\to \infty ~.
    \end{align*}
    \end{enumerate}
  \end{corollary}

\section{Finite-sample properties} \label{sec_sim}
In this section, we study the finite sample properties of the tests \eqref{eq_rejection_region} and \eqref{eq_generalized_rejection_region}. We consider the model
\begin{align*}
\bfx_j^{(i)} = \bfSigma_n^{(i)^{1/2}}\bfz_j^{(i)}, \quad i=1,2,
\end{align*}
where $\bfz_1^{(i)}, \ldots, \bfz_n^{(i)}$ are $p$-dimensional random vectors with i.i.d. entries $z_{1,1}^{(i)}$, where we consider three case for their distribution:
\begin{itemize}
    \item[(a)]

 a standard normal distribution $(z_{1,1}^{(i)}\sim \mathcal{N}(0,1))$
 \item[(b)]
 a t-distribution with $7$ degrees of freedom standardized such that its variance is $1$ $(z_{1,1}^{(i)}\sim t_7/\sqrt{7/5})$ \item[(c)] a Laplace distribution standardized such that its variance is $1$ $(z_{1,1}^{(i)}\sim \Laplace(0, 1/\sqrt{2}))$
 \end{itemize}
Note that the $t$-distribution  does not satisfy assumption \ref{ass_random_vectors}. Nevertheless, we will demonstrate that the proposed tests \eqref{eq_rejection_region} and \eqref{eq_generalized_rejection_region} have a reasonable performance. All reported results are based on $500$ simulations runs, and the nominal level is $\alpha=0.05$. The structure of the covariance matrices $\bfSigma_n^{(i)}$ is specified in the following subsections. 

\paragraph*{Numerical experiments for one leading sample eigenvalue}
    \begin{figure}[p]
		\begin{center}	            
        \includegraphics[scale=0.36]{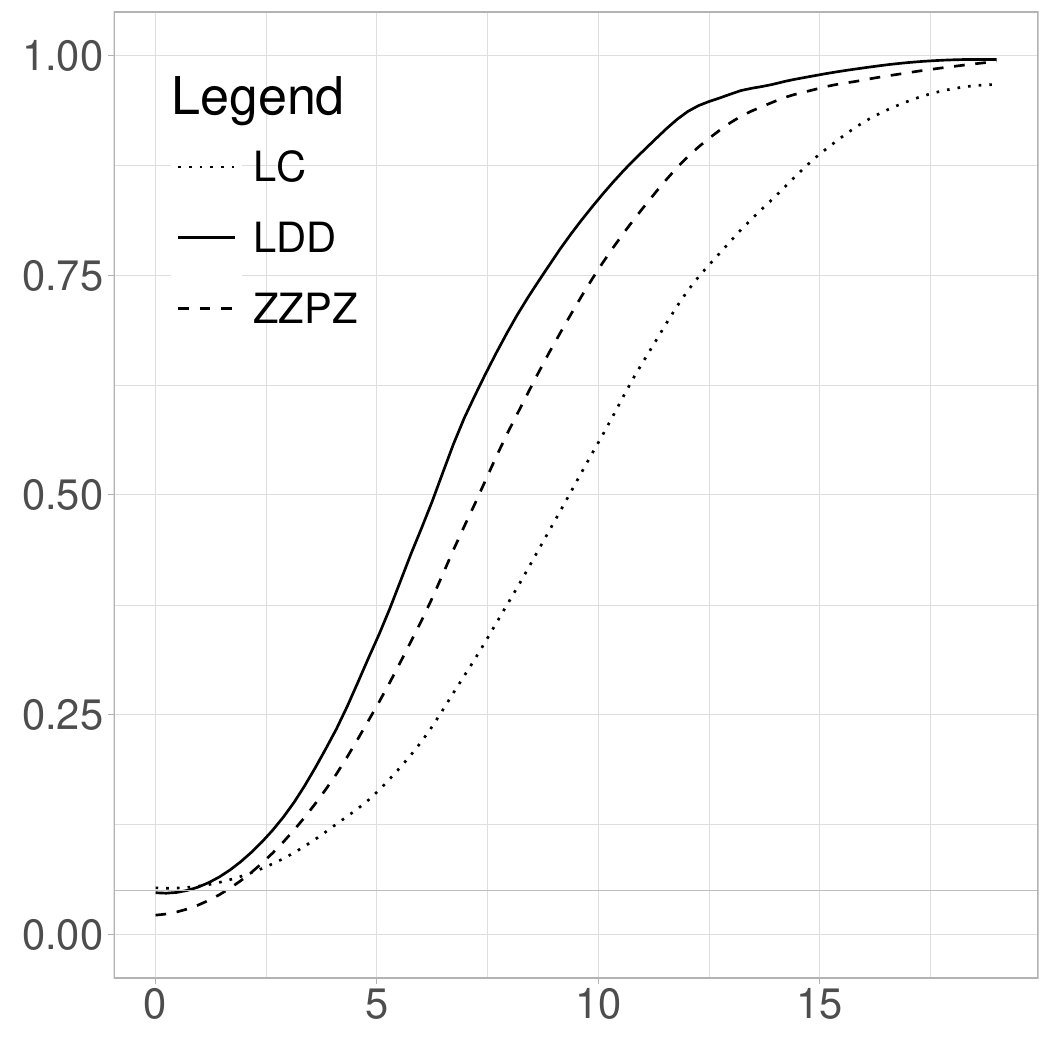}
        ~~~~ ~~~~ 
        \includegraphics[scale=0.36]{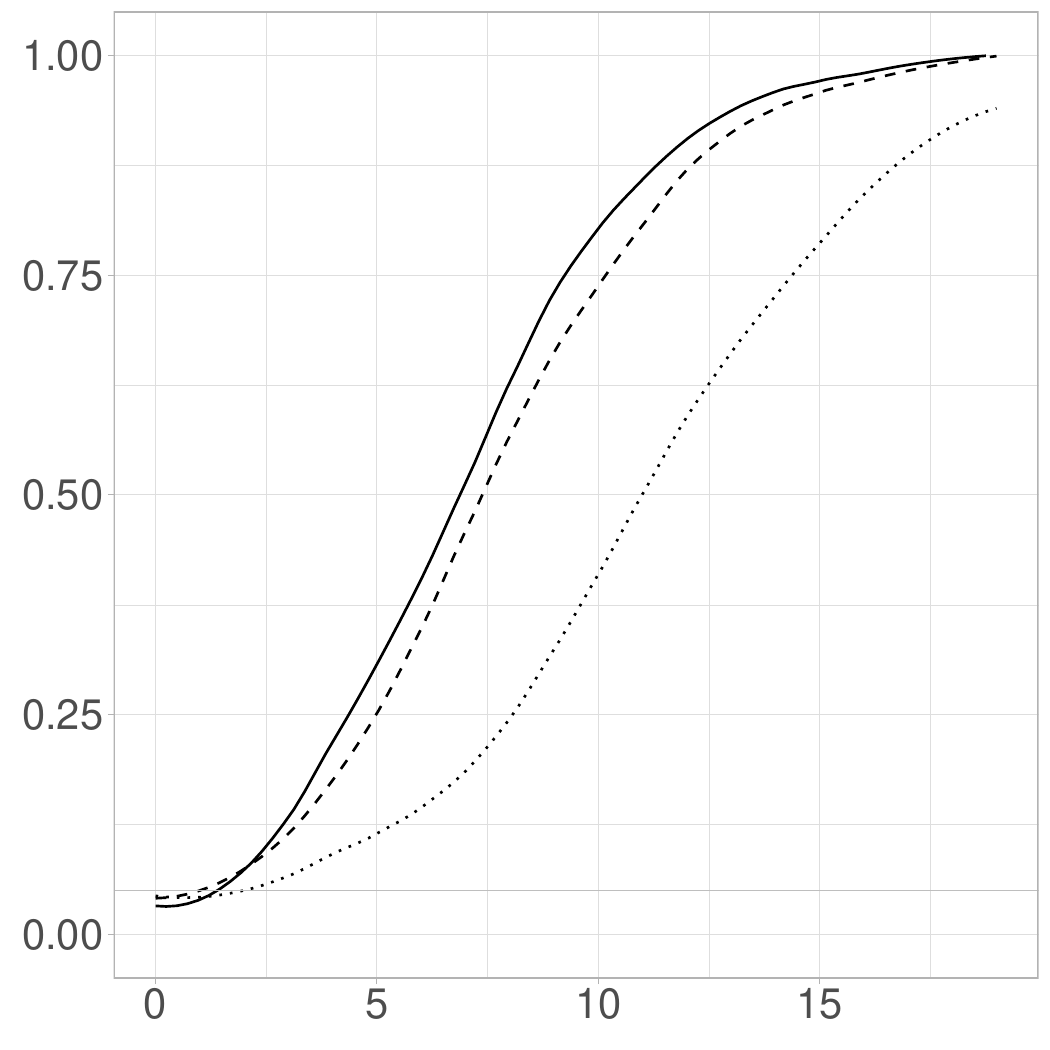}
        \includegraphics[scale=0.36]{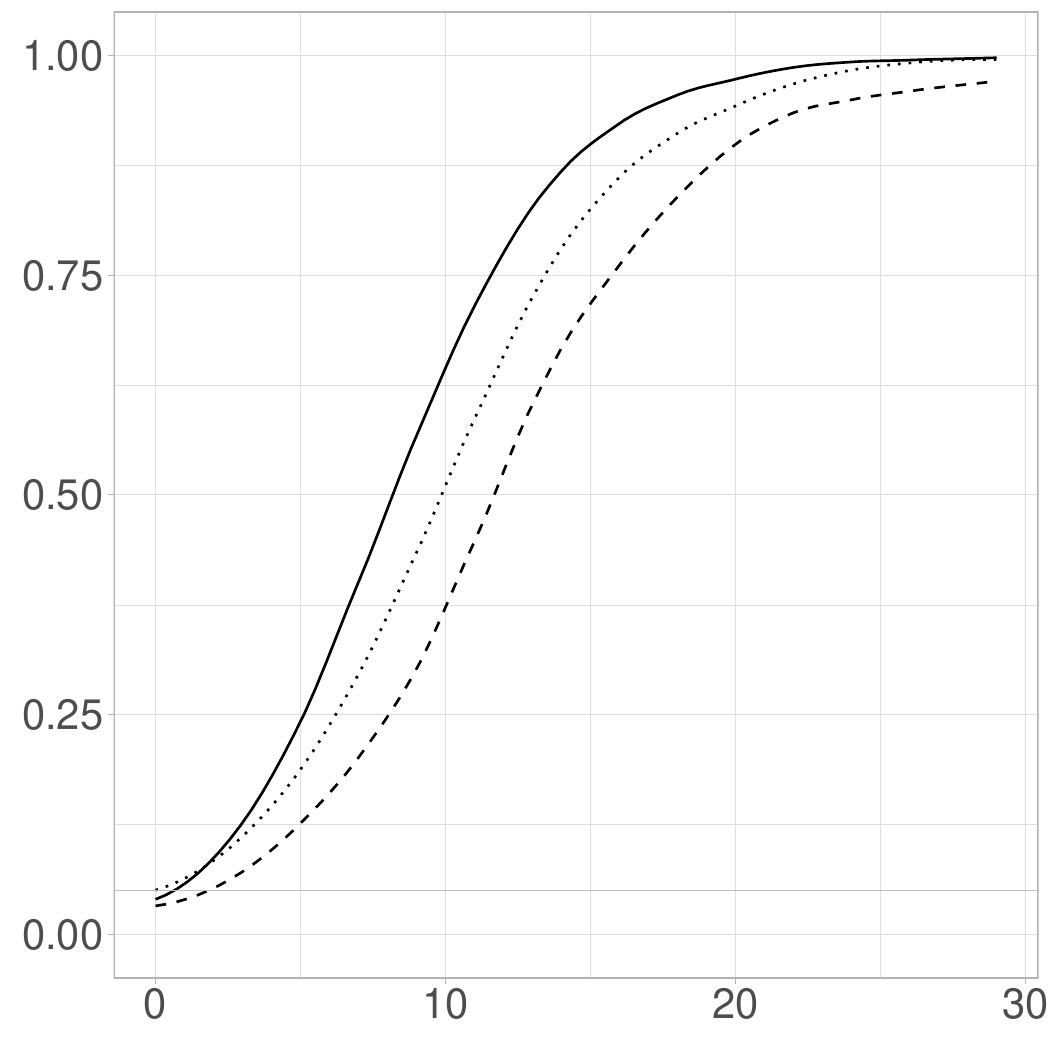}
        ~~~~ ~~~~     
        \includegraphics[scale=0.36]{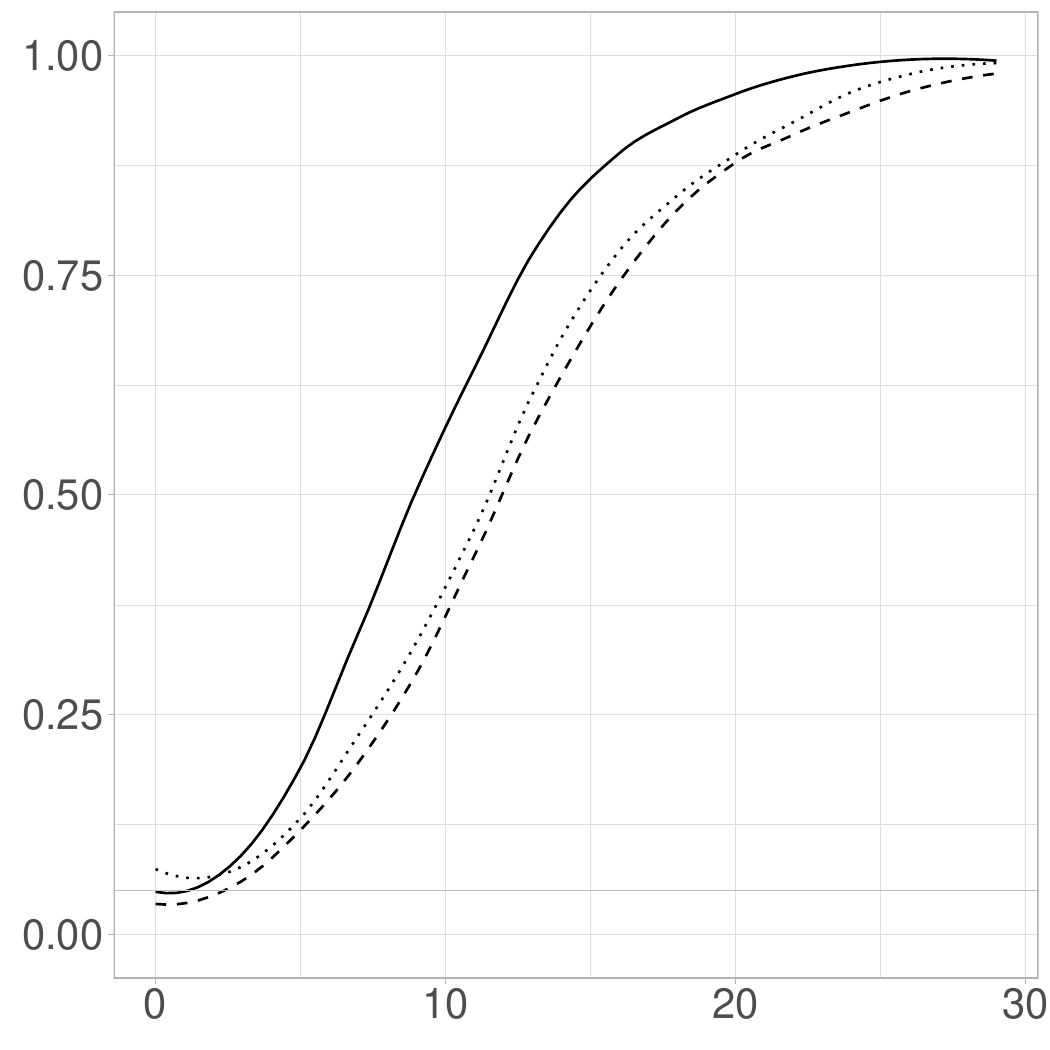}
        \includegraphics[scale=0.36]{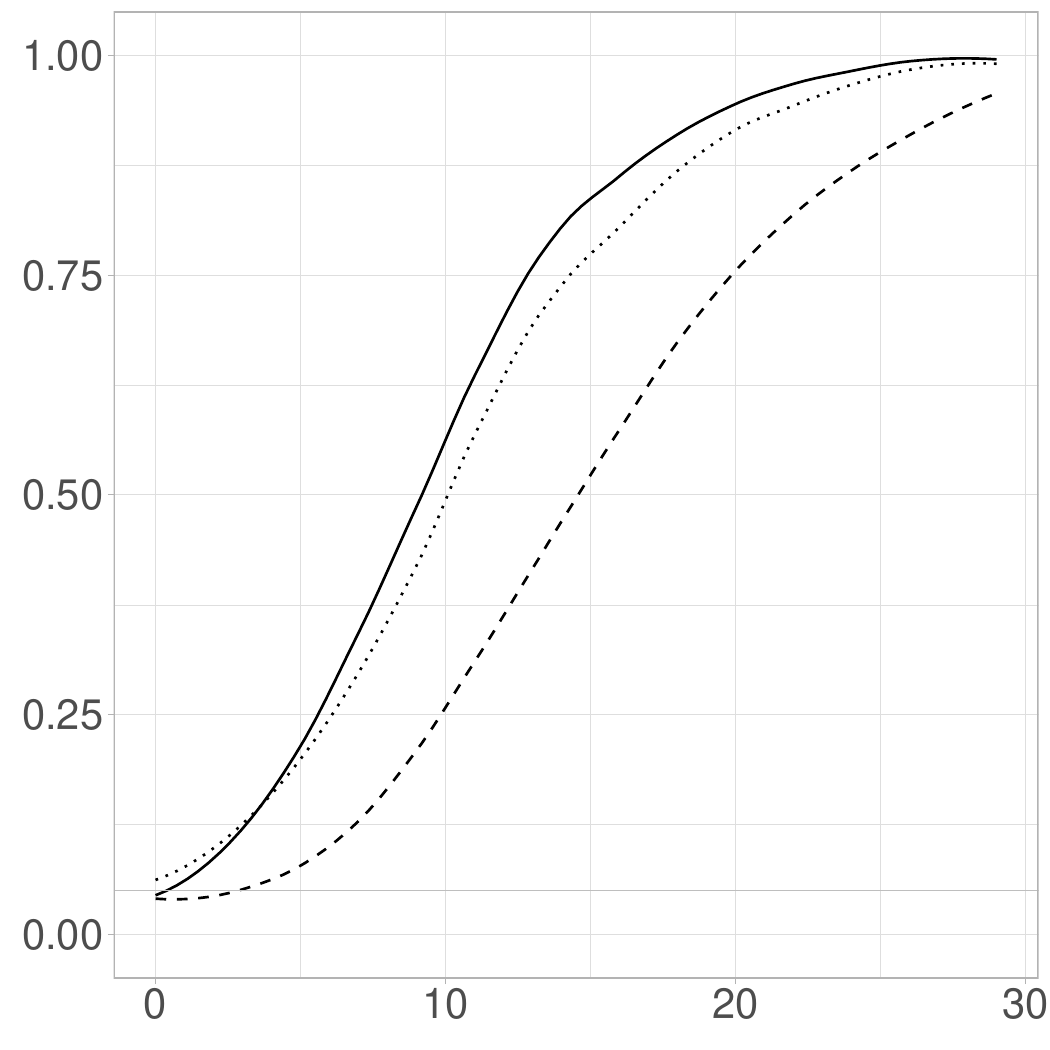}
        ~~~~ ~~~~     
        \includegraphics[scale=0.36]{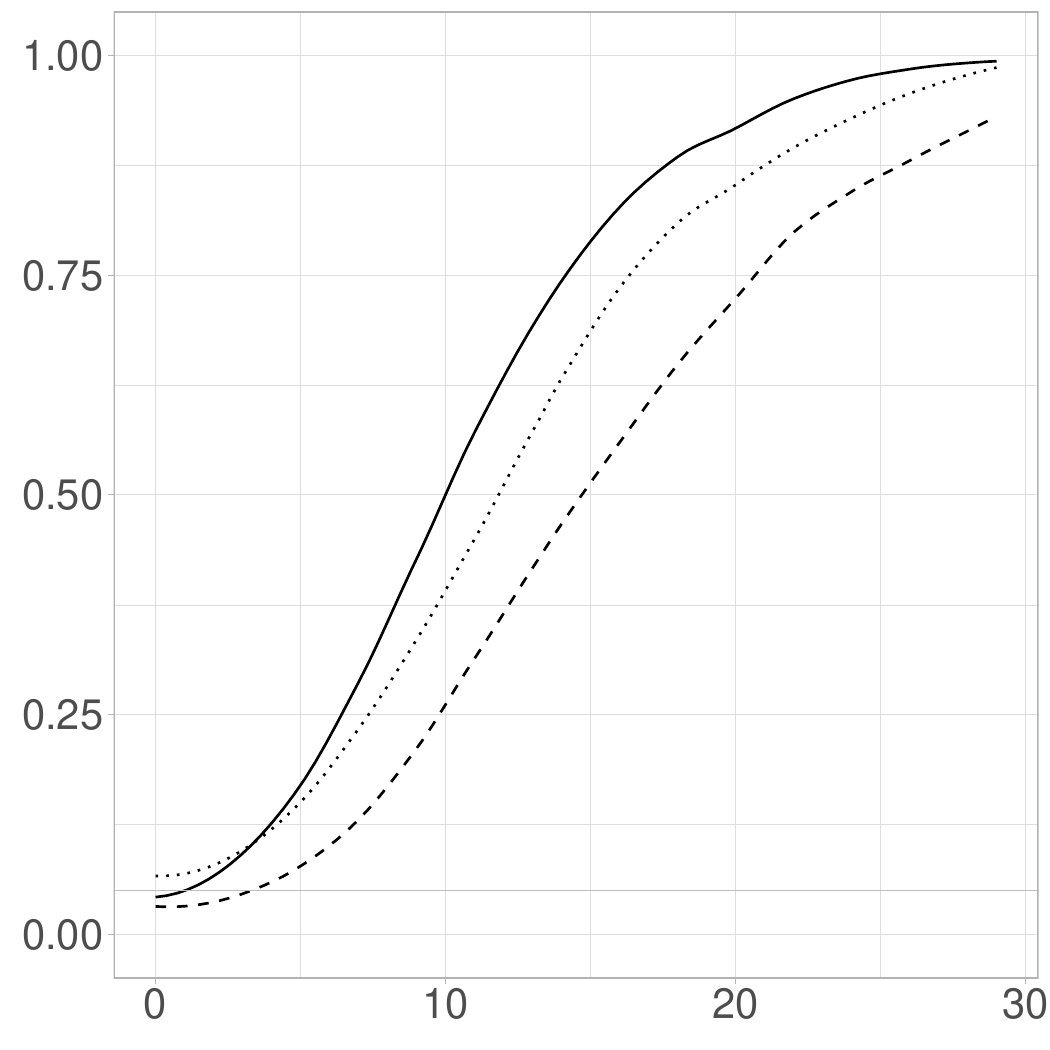}
		\end{center}
		\caption{\it Simulated rejection probabilities of the test \eqref{eq_rejection_region} (solid line, LDD), the test of \cite{zhang2022asymptotic} (dashed line, ZZPZ) and the test of \cite{li_and_chen_2012} (dotted line, LC)  in model \eqref{eq_model_1}. Left panels:   $(p,n)=(500, 100)$. Right panels: $(p,n)=(1000, 100)$. First row: $\bfz_{1,1}^{(i)} \sim \mathcal{N}(0,1)$, second row: $\bfz_{1,1}^{(i)}  \sim t_{7}/\sqrt{7/5}$, third row: $\bfz_{1,1}^{(i)}  \sim \Laplace(0, 1/\sqrt{2})$.}\label{fig_empirical_rej_model_1}
	\end{figure}
In the following, we provide numerical results on the performance of the new test \eqref{eq_rejection_region}.
First, we investigate population covariance matrices which differ in their leading  eigenvalue, that is 
\begin{align}\label{eq_model_1}
\bfSigma_n^{(1)} = \diag(10, \underbrace{7, \ldots, 7}_{10}, \underbrace{1, \ldots, 1}_{p-11}),\quad \bfSigma_n^{(2)}(\delta)= \diag(10+\delta, \underbrace{7, \ldots, 7}_{10}, \underbrace{1, \ldots, 1}_{p-11}),
\end{align}

where the leading eigenvalue changes as $\delta$ increases, and $\delta=0$ corresponds to the null hypothesis in \eqref{eq_hypothesis}. For these covariance matrices the limiting spectral distribution $H^{(1)}=H^{(2)}=\delta_{\lcb 1\rcb}$, where $\delta_{\lcb 1 \rcb}$ is the Dirac measure at point $1$. Thus, eigenvalues $\alpha_{j}$ satisfying $\alpha_j > 1+ \sqrt{y}$ are supercritical eigenvalues. We will set $p/n \leq 10$. Therefore, the eigenvalues $\alpha_1=10+\delta$ with $\delta\in \mathbb{N}_0$ and $\alpha_2=\ldots =\alpha_{11}=7$ are supercritical eigenvalues satisfying assumption \ref{ass_supercritical_ev_generalized}. The empirical rejection probabilities of the test \eqref{eq_rejection_region}
are displayed in Figure \ref{fig_empirical_rej_model_1} for $(p,n)=(500,100)$ (left panel) and $(p,n)=(1000,100)$ (right panel), where the entries $z_{1,1}^{(i)}$ are standard normal distributed (first row), $t_{7}/\sqrt{7/5}$ (second row) and $\Laplace(0, 1/\sqrt{2})$ (third row) and various values of $\delta$. For the sake of comparison, we also display the empirical rejection probabilities of the proposed test in \cite{li_and_chen_2012} as well as in \cite{zhang2022asymptotic}. 
We observe that the three tests keep the nominal level $\alpha=0.05$ in all cases under consideration. Furthermore, our new test \eqref{eq_rejection_region} consistently outperforms the proposed test in \cite{zhang2022asymptotic}. In particular, the new test \eqref{eq_rejection_region}  is the most powerful test among the three competitors in all six scenarios. Moreover, in the case that the data are not (standard) normally distributed, we can see from Figure \ref{fig_empirical_rej_model_1} that the proposed test in \cite{zhang2022asymptotic} (dashed line) suffers a considerable loss in power, such that the test proposed in \cite{li_and_chen_2012} is more powerful. 

\medskip

Next, we examine the power of the  test \eqref{eq_rejection_region} for covariance matrices, which share the same supercritical eigenvalues but are distinguished in their bulk spectrum. For this purpose, we consider
\begin{align}\label{eq_model_2}
\begin{split}
\bfSigma_n^{(1)} & =\diag(11, 7,7,7, \underbrace{1, \ldots, 1}_{p-4}), \\ 
\bfSigma_n^{(2)}(\delta) &= \diag(11, 7,7,7, \underbrace{1+\frac{\delta}{20}, \ldots, 1+\frac{\delta}{20}}_{\frac{p-4}{2}}, \underbrace{1-\frac{\delta}{20}, \ldots, 1-\frac{\delta}{20}}_{\frac{p-4}{2}})
\end{split}
\end{align}
as population covariance matrices, where the bulk changes with $\delta$, which varies from $0$ to $19$. The case $\delta=0$ corresponds to the null hypothesis in \eqref{eq_hypothesis}.

    \begin{figure}[p]
		\begin{center}	            
        \includegraphics[scale=0.36]{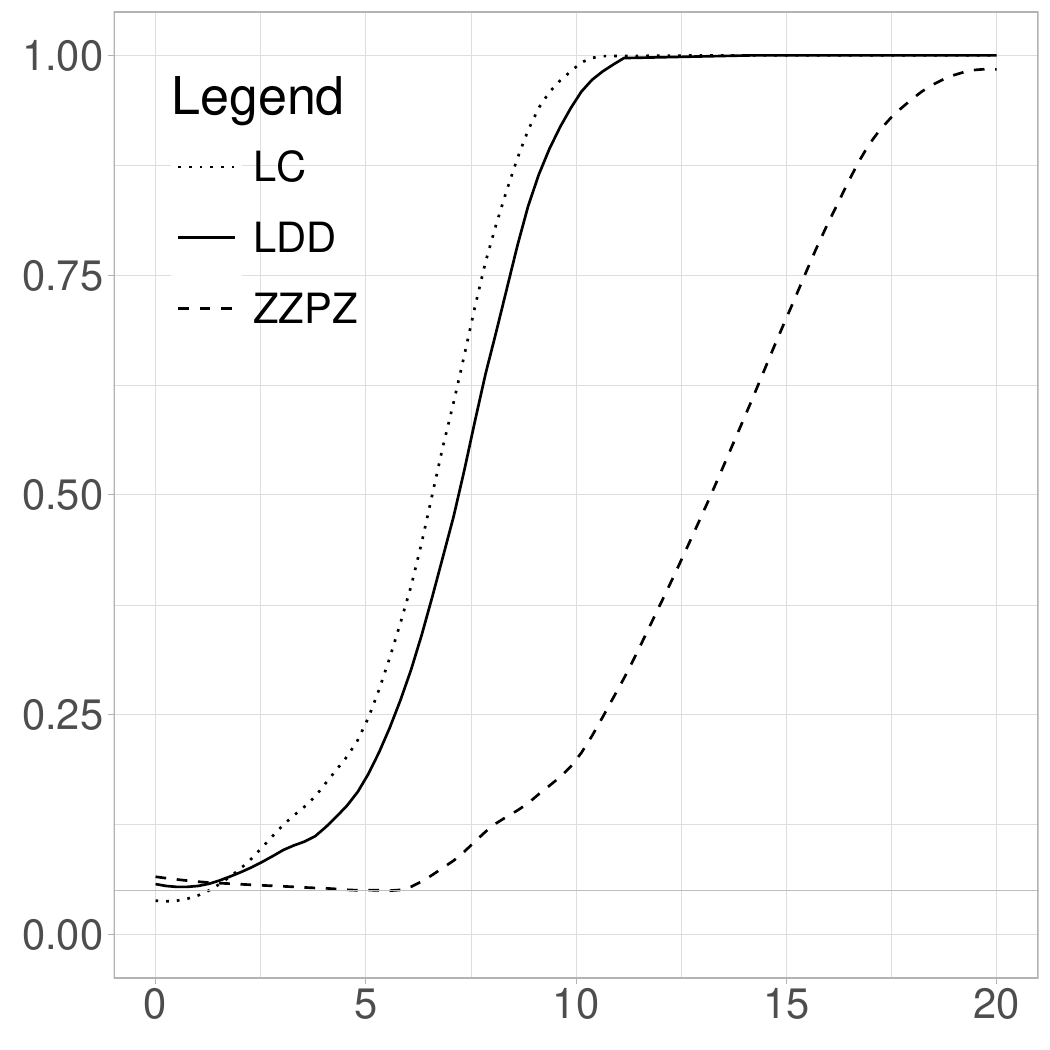}
         ~~~~ ~~~~
        \includegraphics[scale=0.36]{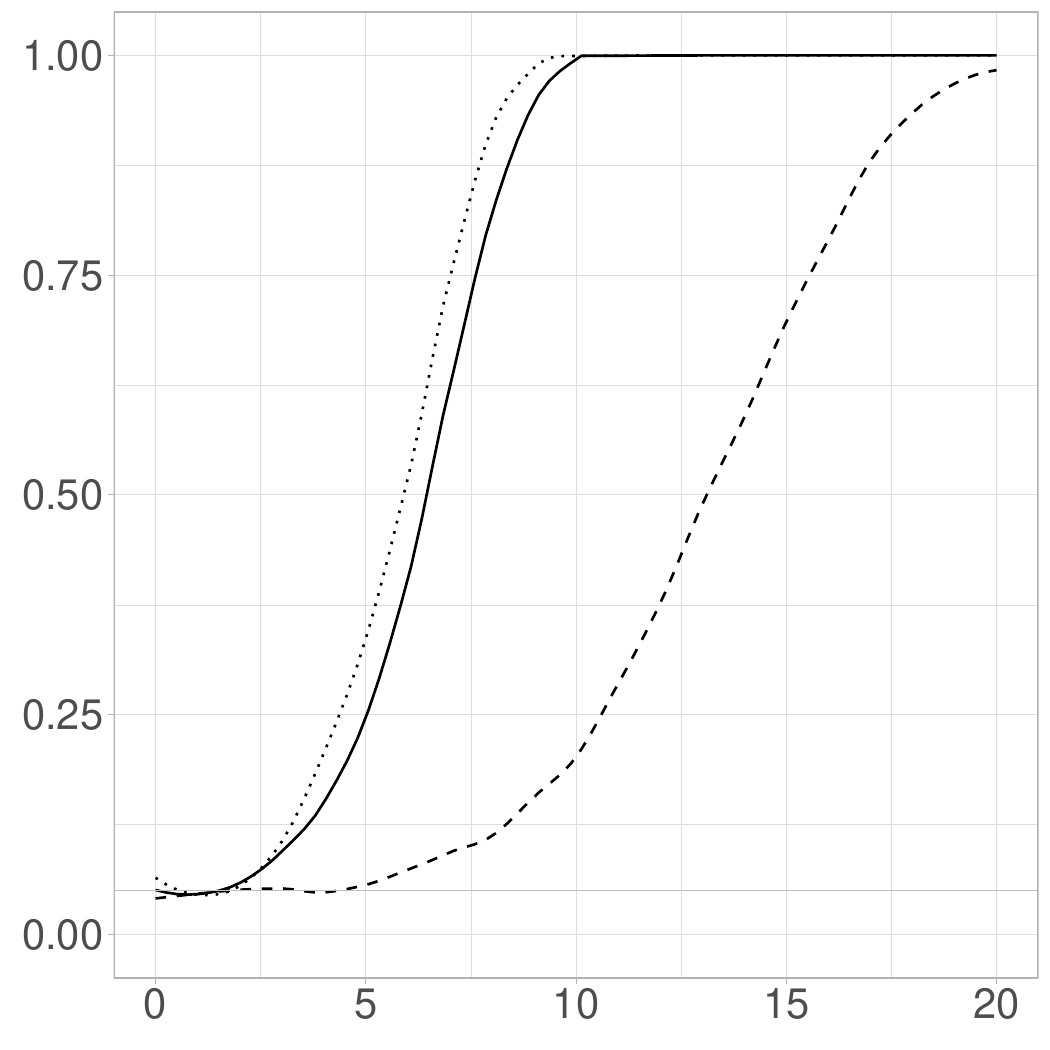}
        \includegraphics[scale=0.36]{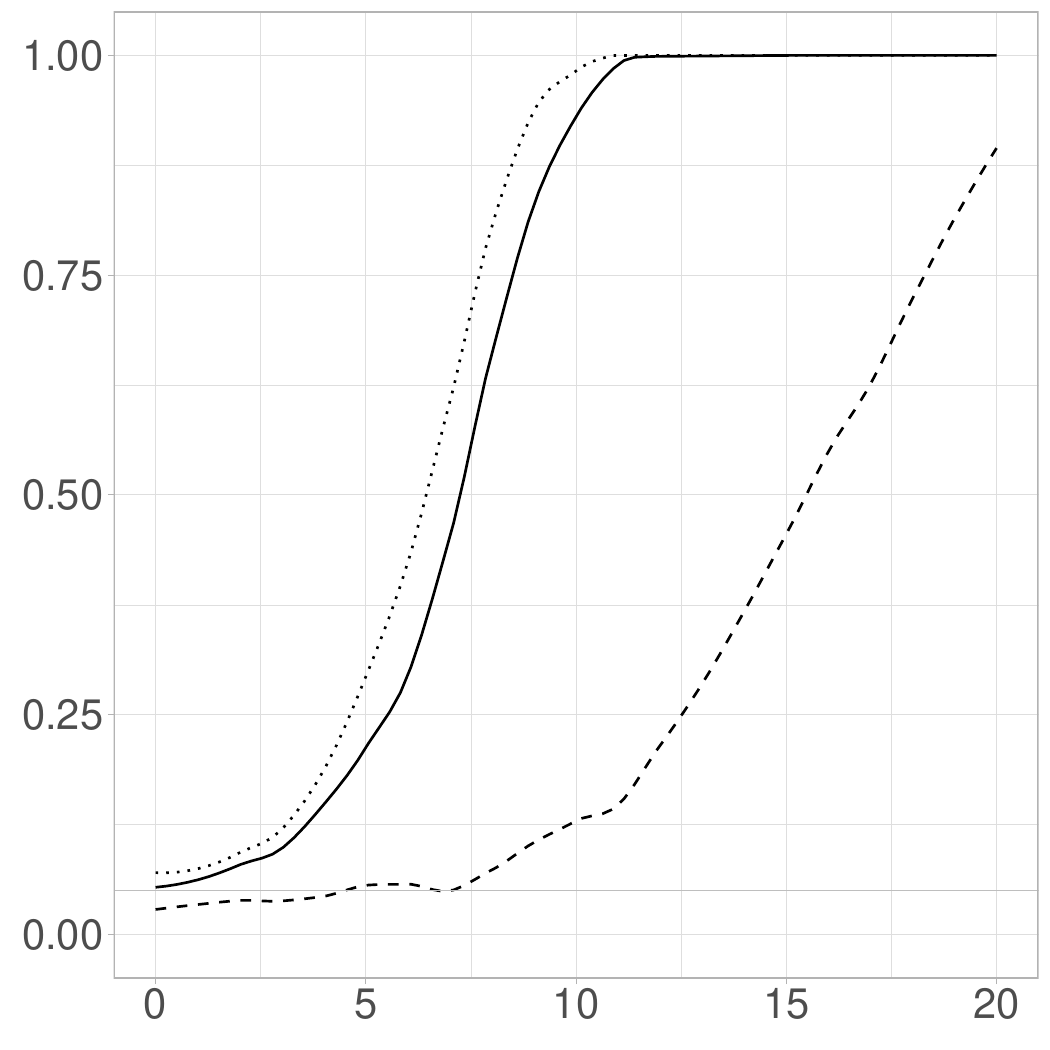}
         ~~~~ ~~~~
        \includegraphics[scale=0.36]{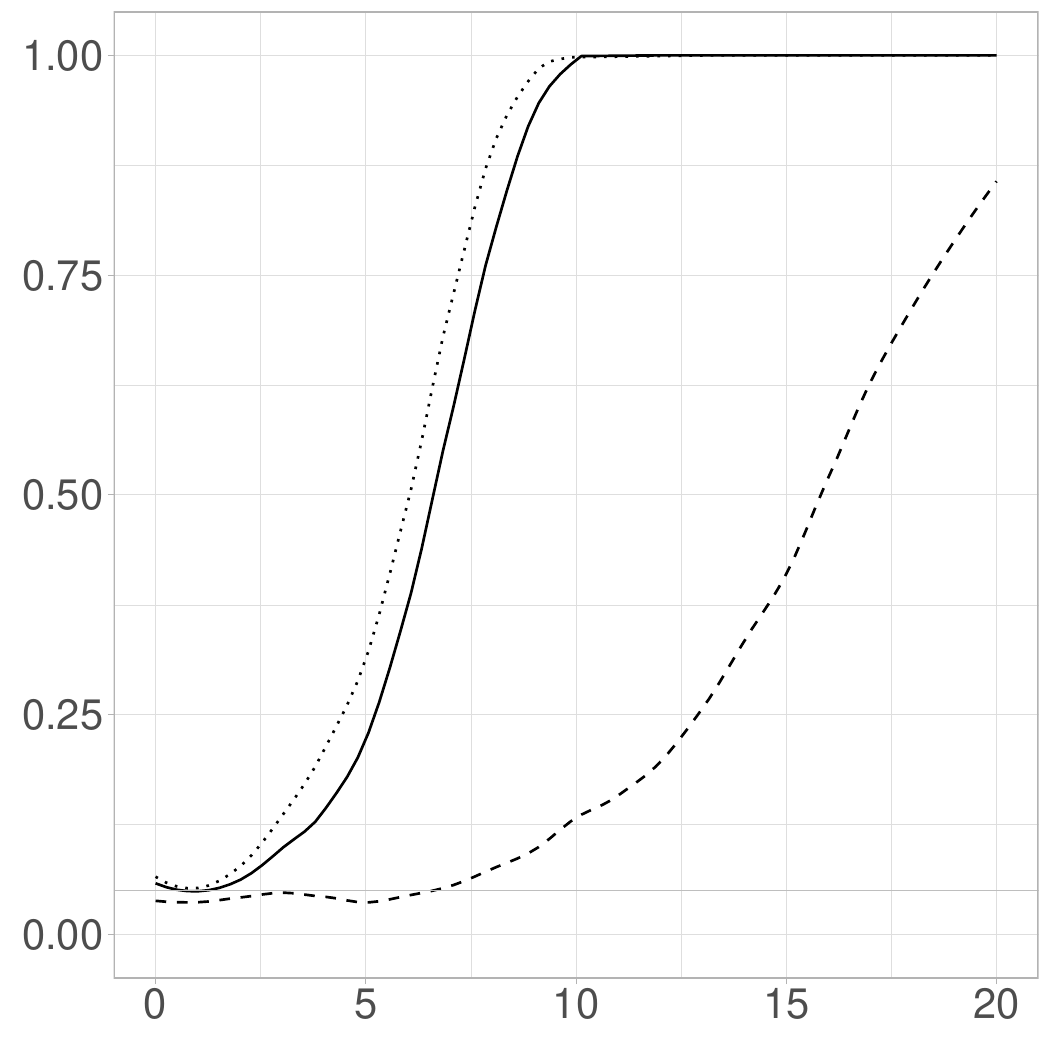}
        \includegraphics[scale=0.36]{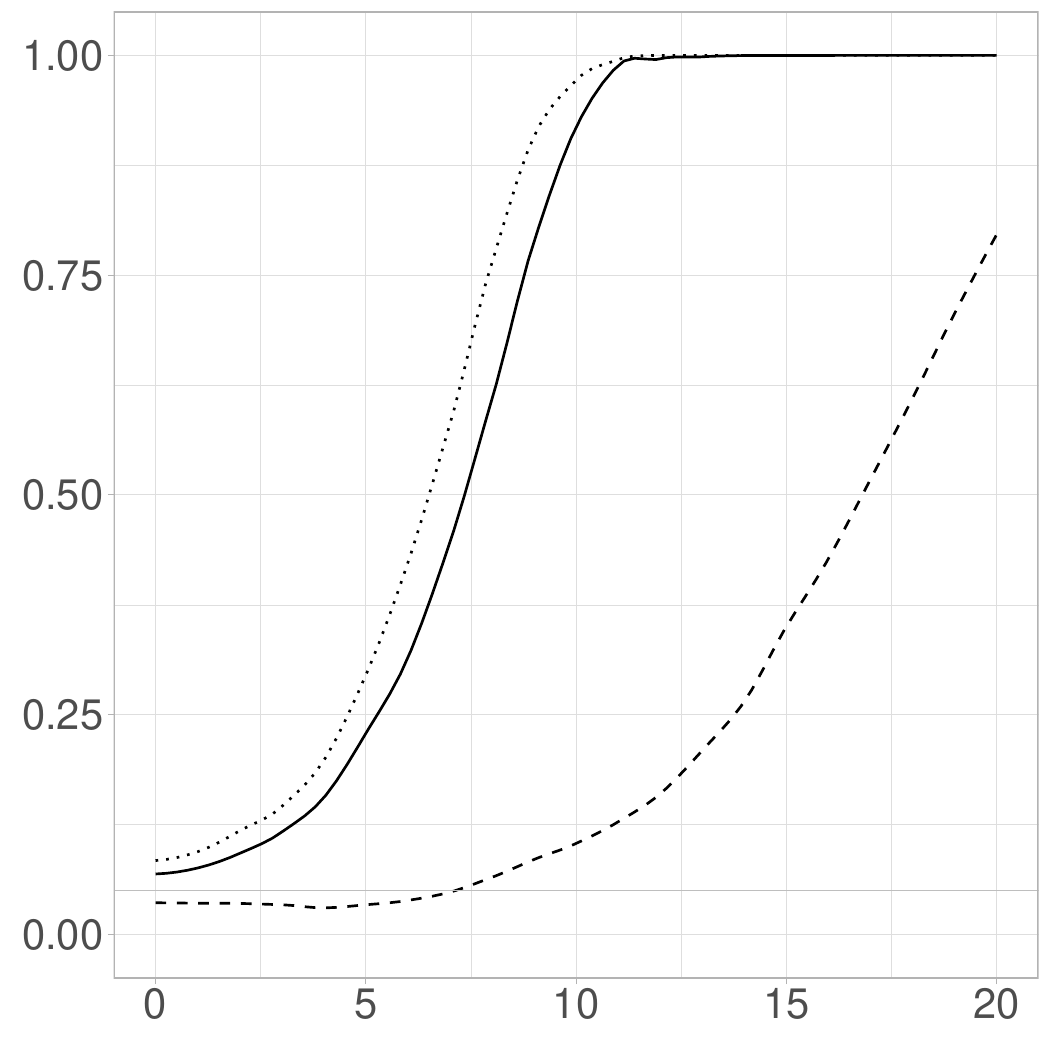}
         ~~~~ ~~~~
        \includegraphics[scale=0.36]{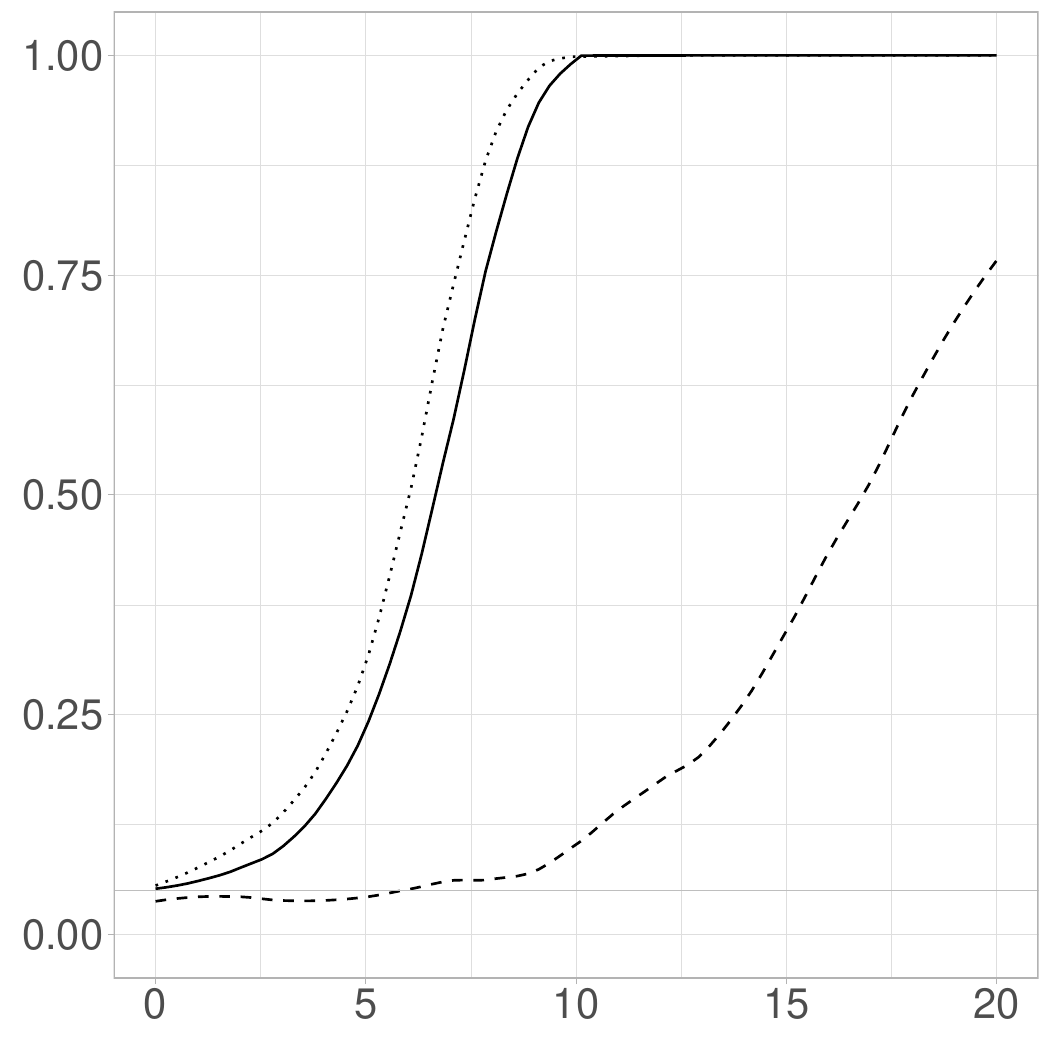}
		\end{center}
		\caption{\it Simulated rejection probabilities of the test \eqref{eq_rejection_region} (solid line, LDD), the test of \cite{zhang2022asymptotic} (dashed line, ZZPZ) and the test of \cite{li_and_chen_2012} (dotted line, LC) in model \eqref{eq_model_2}. Left panels:   $(p,n)=(500, 100)$. Right panels: $(p,n)=(1000, 100)$. First row: $\bfz_{1,1}^{(i)} \sim \mathcal{N}(0,1)$, second row: $\bfz_{1,1}^{(i)}  \sim t_{7}/\sqrt{7/5}$, third row: $\bfz_{1,1}^{(i)}  \sim \Laplace(0, 1/\sqrt{2})$.}\label{fig_empirical_rej_model_2}
	\end{figure}
The limiting spectral distribution in the first sample is given by  $H^{(1)}=\delta_{\lcb 1\rcb}$, and the eigenvalues $\alpha_1=11$ and $\alpha_2=\ldots =\alpha_4=7$ are supercritical eigenvalues. The limiting distribution in the second sample is given by $H^{(2)}$ is defined as $\frac{1}{2} \big (  \delta_{\{ 1+\frac{\delta}{20}\} }+ \delta_{\{ 1-\frac{\delta}{20}\} } \big )$. Consequently, eigenvalues $\alpha_j$ satisfying the inequality  $2 > \frac{y}{400} \big ( {(\alpha_j-\frac{1}{20})^{-2} } + {(\alpha_j+\frac{1}{20})^{-2}}\big )  $ are supercritical eigenvalues. Because $p/n \leq 10$, the eigenvalues $\alpha_1=11$ and $\alpha_2=\ldots =\alpha_4=7$ are supercritical for the covariance matrix $\bfSigma_n^{(2)}$ as well. The
empirical rejection probabilities  of the three tests  are  displayed in Figure \ref{fig_empirical_rej_model_2}. In all cases under consideration, the  nominal level $\alpha=0.05$ is reasonably approximated by the three methods. With increasing $\delta$, the  test \eqref{eq_rejection_region} and the test proposed in  \cite{li_and_chen_2012} detect the alternative with reasonable power. In particular, for $\delta\geq 10$, both tests reject the null hypothesis with probability close to $1$.
The test proposed in  \cite{li_and_chen_2012} has slightly greater power than the test \eqref{eq_rejection_region}, because it is tailored to dense alternatives as specified in \eqref{eq_model_2}.  In contrast, the test proposed  in \cite{zhang2022asymptotic} has much smaller power. If  $\delta \leq 4 $, the rejection probability of their test remains at the nominal level $\alpha=0.05$. For  $\delta \geq 5$, the power of their test increases, but it does not reach the same level as for the other proposed tests.

\paragraph*{Numerical experiments for three leading sample eigenvalue}
    In this section, we present some numerical results illustrating  the performance of the test \eqref{eq_generalized_rejection_region} that can detect differences across multiple population eigenvalues. Specifically, we examine the test statistic $T_{n, FC,m}$ for $m=3$. Because the $p$-value of the distribution $\lnorm \mathbf{W} \rnorm_1$ with $\mathbf{W}\sim\mathcal{N}_3(\mathbf{0}_3, \bfSigma_{E,3})$ is unknown, the $p$-value will be determined in each run with the help of resampling.  For this procedure, let $\hat{\bfSigma}_{E,3,n}$ be the consistent estimator of $\bfSigma_{E,3}$ from Lemma \ref{lem_consistent_estimators}. Then, by resampling the random variable $\lnorm \mathbf{W}_r^\star \rnorm_1$ with $\mathbf{W}_r^\star \sim \mathcal{N}_3(\mathbf{0}_3, \hat{\bfSigma}_{E,3,n})$ 
    $r=10000$ times, we determine the corresponding  $p$-value. 
  \medskip
  
    In the following, we consider two different models. First, we consider
     the population covariance matrices \begin{align}
     \label{eq_model_3}
     \begin{split}
        \bfSigma_n^{(1)} &= \diag(10, 8, 7, \underbrace{6, \ldots, 6}_{8}, \underbrace{1, \ldots, 1}_{p-11}),  \\
         \bfSigma_n^{(2)}(\delta) &= \diag(10+\delta, 8+\delta, 7+\delta, \underbrace{6, \ldots, 6}_{8}, \underbrace{1, \ldots, 1}_{p-11}).
             \end{split}
    \end{align}
 Here, the limiting distributions of the bulk spectrum are given by $H^{(1)} = H^{(2)}=\delta_{\lcb 1\rcb}$, and   
  it follows from the discussion for model \eqref{eq_model_1} that the eigenvalues $\alpha_1=10+\delta, \alpha_2=8+\delta, \alpha_3=7+\delta$ with $\delta\in \mathbb{N}_0$ and $\alpha_4=\ldots=\alpha_{11}=6$ are supercritical eigenvalues.

Second, we consider
    the population covariance matrices 
    \begin{align}
    \label{eq_model_4}
    \begin{split}
    \bfSigma_n^{(1)} &= \diag(20, 15, 13, 12, d_5, \ldots, d_p), \quad d_i=3-2.5/(p-5)\cdot(i-5), 
    \\
    \bfSigma_n^{(2)}(\delta) &= \diag(20+\delta, 15+\delta, 13+\delta, 12, d_5, \ldots, d_p), \quad d_i=3-2.5/(p-5)\cdot(i-5).
        \end{split}
    \end{align}
  Here, the limiting distributions of the bulk spectrum are given by $H^{(1)} = H^{(2)}=\mathcal{U}_{\lsb 0.5, 3\rsb}$. Therefore, we have
    \begin{align*}
        \psi^{(i)} (\alpha) = \alpha(1-y) - y\alpha^2 \frac{1}{2.5} \log \lb \frac{\alpha-3}{\alpha-0.5}\rb,
    \end{align*}
    which leads to
    \begin{align*}
        \psi^{(i)^\prime}(\alpha) = 1-y\lb \frac{4}{5} \alpha \log\lb \frac{\alpha-3}{\alpha-0.5}\rb + \alpha^2 \frac{1}{\lb \alpha-3\rb \lb \alpha-0.5\rb}+1\rb.
    \end{align*}
We will investigate $p/n\leq 10$, and  one can verify that the inequality 
    \begin{align*}
        \psi^{(i)^\prime}(\alpha) > 0 \Leftrightarrow \alpha > 8.31816
    \end{align*}
    holds. Therefore, the eigenvalues $\alpha_1 = 20+\delta, \alpha_2=15+\delta, \alpha_3=13+\delta$ and $\delta_4=12$ for $\delta\in\mathbb{N}_0$ are supercritical eigenvalues satisfying assumption \ref{ass_supercritical_ev_generalized}. 
    
  \medskip

  For both models \eqref{eq_model_3} and \eqref{eq_model_4}, the three largest  eigenvalues of the population matrix $ \bfSigma_n^{(2)}$ change as $\delta$ grows, where  $\delta=0$ corresponds to the null hypothesis in \eqref{eq_hypothesis}.  
The empirical rejection probabilities of the test \eqref{eq_generalized_rejection_region} for models \eqref{eq_model_3} and \eqref{eq_model_4} are displayed in Figure \ref{fig_empirical_rej_model_3} and \ref{fig_empirical_rej_model_4}, respectively. In these figures, we also display the empirical rejection probabilities of the test \eqref{eq_rejection_region} (which uses only one leading eigenvalue) and of the tests proposed in    \cite{li_and_chen_2012} and \cite{zhang2022asymptotic}. All tests provide a reasonable approximation of the nominal level $\alpha=0.05$ under the null hypothesis. As expected, the test  \eqref{eq_generalized_rejection_region} (dashed line) has more power than the test  \eqref{eq_rejection_region} (solid line). In particular, the test \eqref{eq_generalized_rejection_region} is in all cases the most powerful test among the four competitors. 
Furthermore, if the data are (standard) normal distributed, the proposed test in \cite{zhang2022asymptotic} (dotted line) has more power than the new test \eqref{eq_rejection_region}. However, when the data are not (standard) normal distributed, we can observe from Figure \ref{fig_empirical_rej_model_3} and \ref{fig_empirical_rej_model_4} that the proposed test in \cite{zhang2022asymptotic} suffers a considerable loss in power, and in this case the new test \eqref{eq_rejection_region} is more powerful.

   \medskip
   
    Further simulation results are presented in Section \ref{sec_appendix_simulation} of the Appendix. There, we consider covariance matrices, which have supercritical eigenvalues with multiplicity greater than $1$, such that assumption \ref{ass_supercritical_ev_generalized} is not satisfied. Similar to models \eqref{eq_model_3} and \eqref{eq_model_4}, the three largest eigenvalues of the population matrix change as $\delta$ grows.
    Although assumption \ref{ass_supercritical_ev_generalized} is not satisfied, all tests provide in all cases under consideration a reasonable approximation of the nominal level $\alpha=0.05$ under the null hypothesis. In particular, the test $\eqref{eq_generalized_rejection_region}$ is in all cases the most powerful test among the four competitors.

    \begin{figure}[p]
		\begin{center}	            
        \includegraphics[scale=0.36]{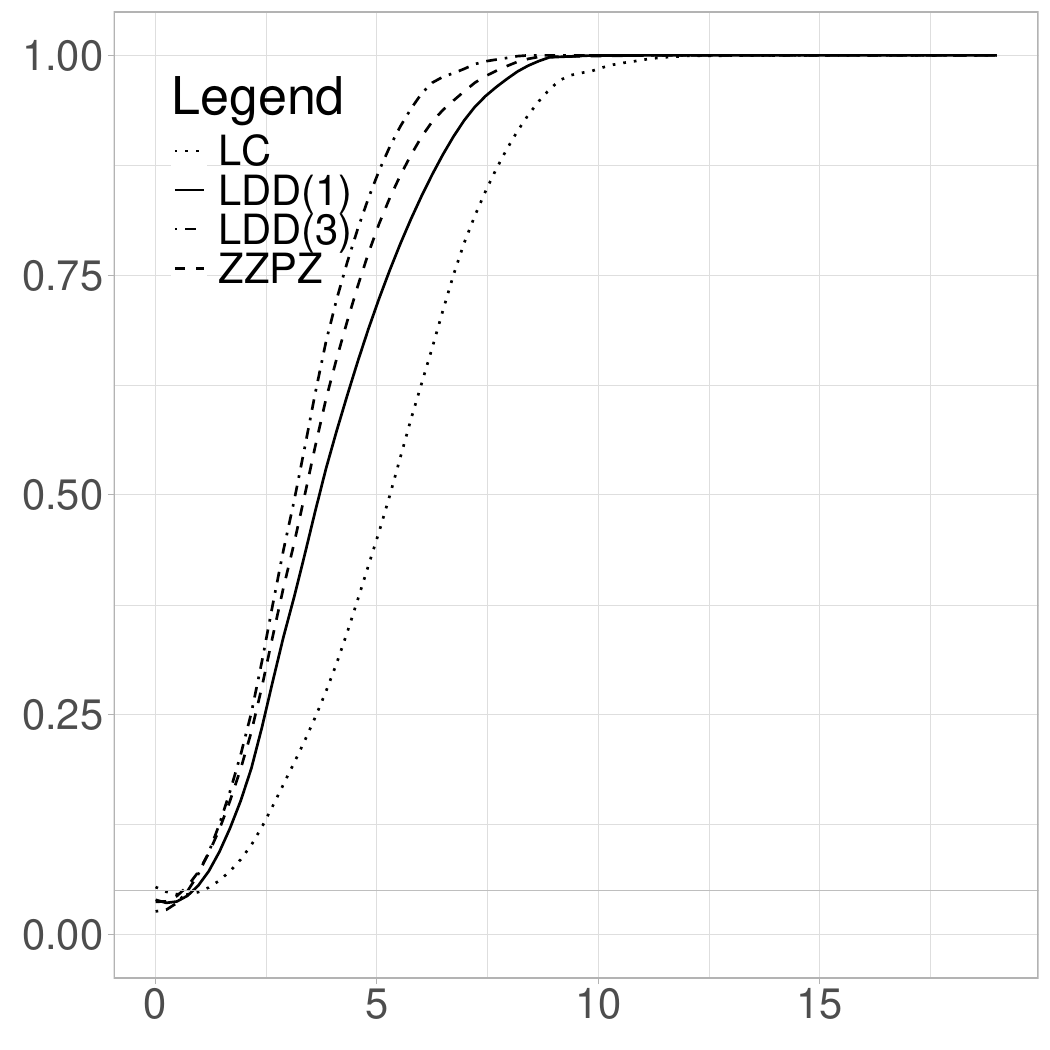}
         ~~~~ ~~~~
        \includegraphics[scale=0.36]{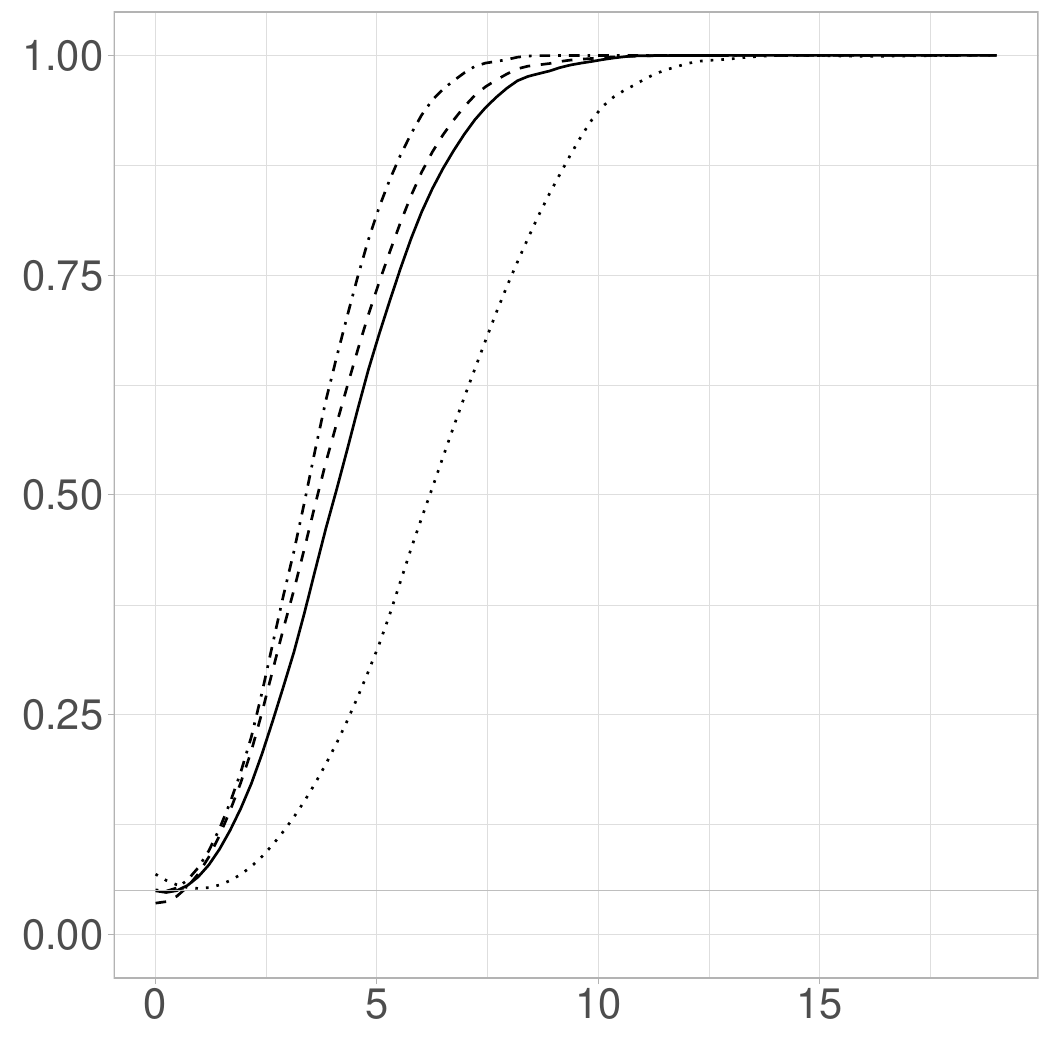}
        \includegraphics[scale=0.36]{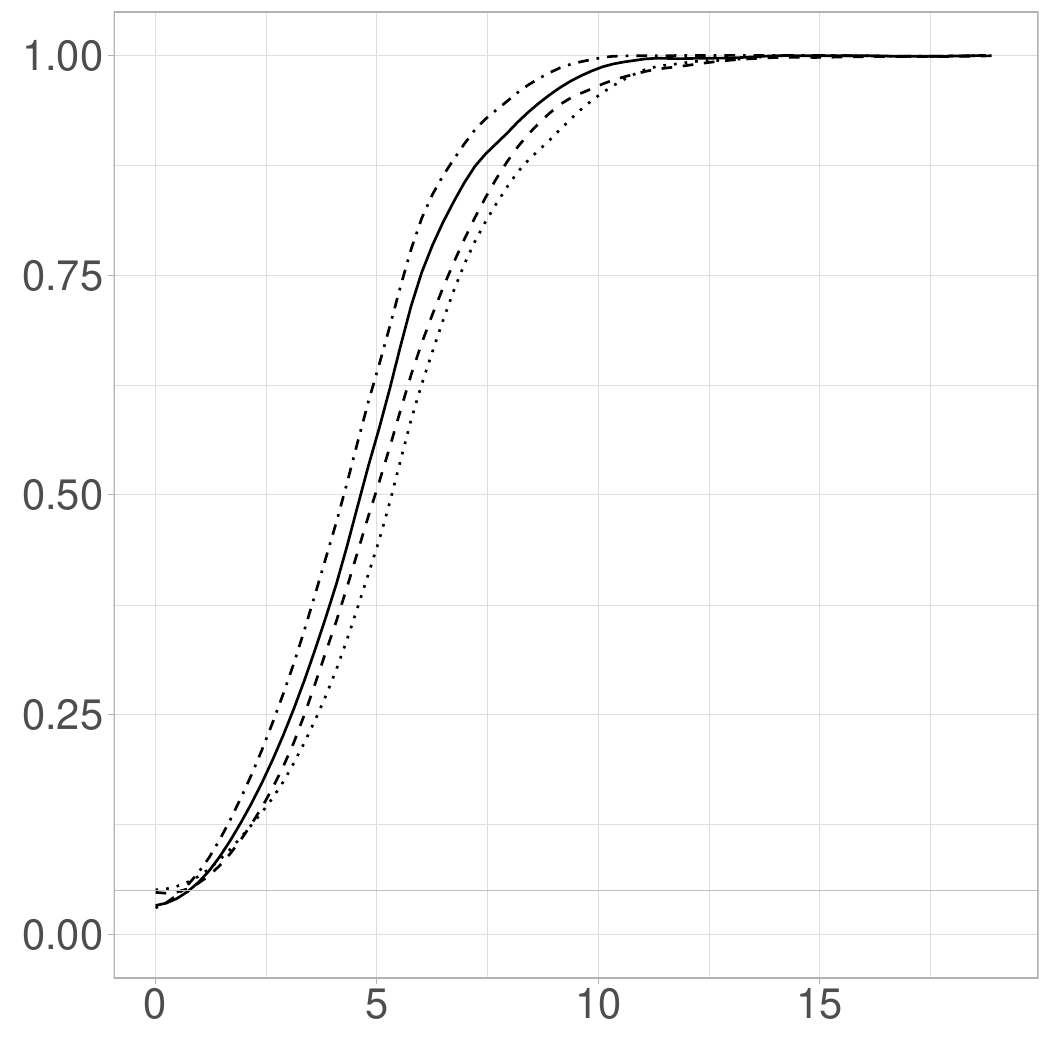}
         ~~~~ ~~~~
        \includegraphics[scale=0.36]{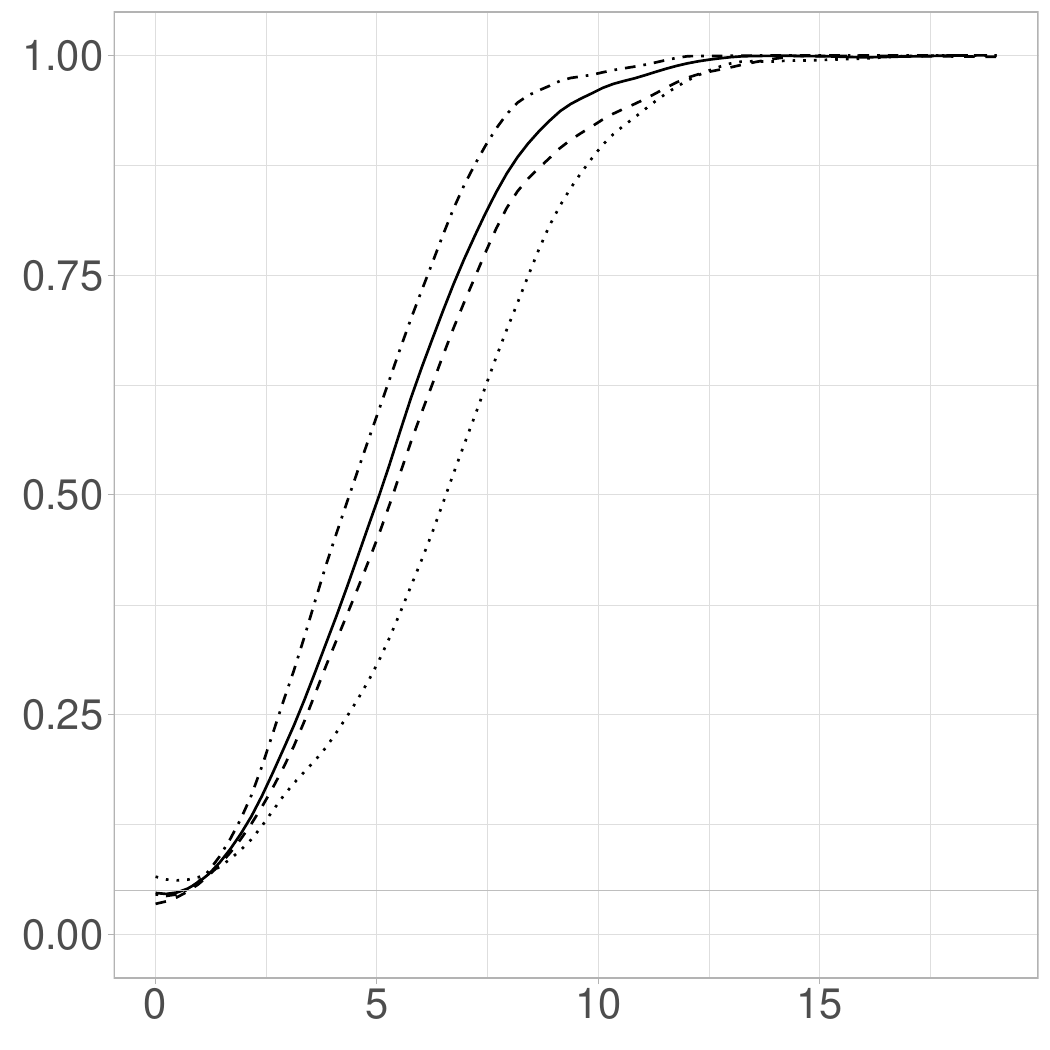}
        \includegraphics[scale=0.36]{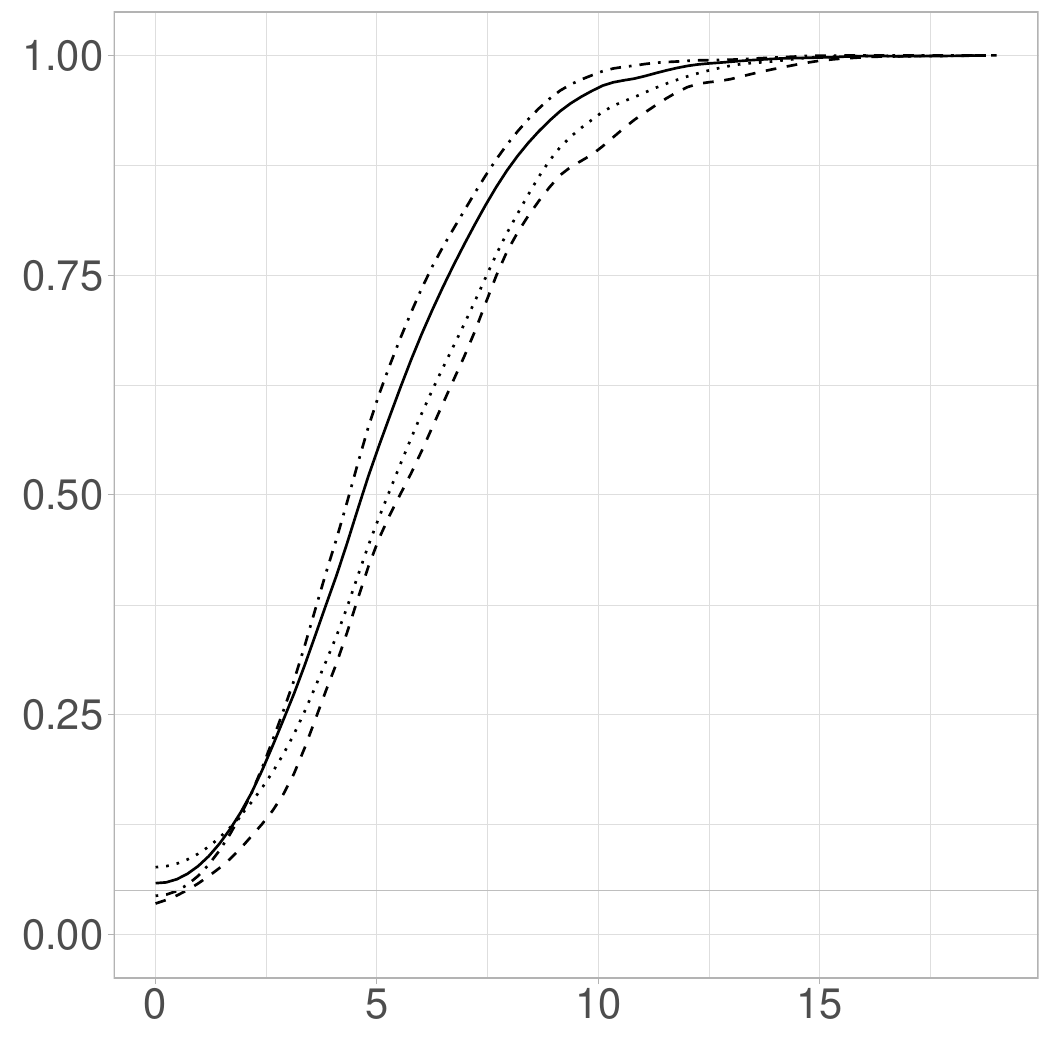}
         ~~~~ ~~~~
        \includegraphics[scale=0.36]{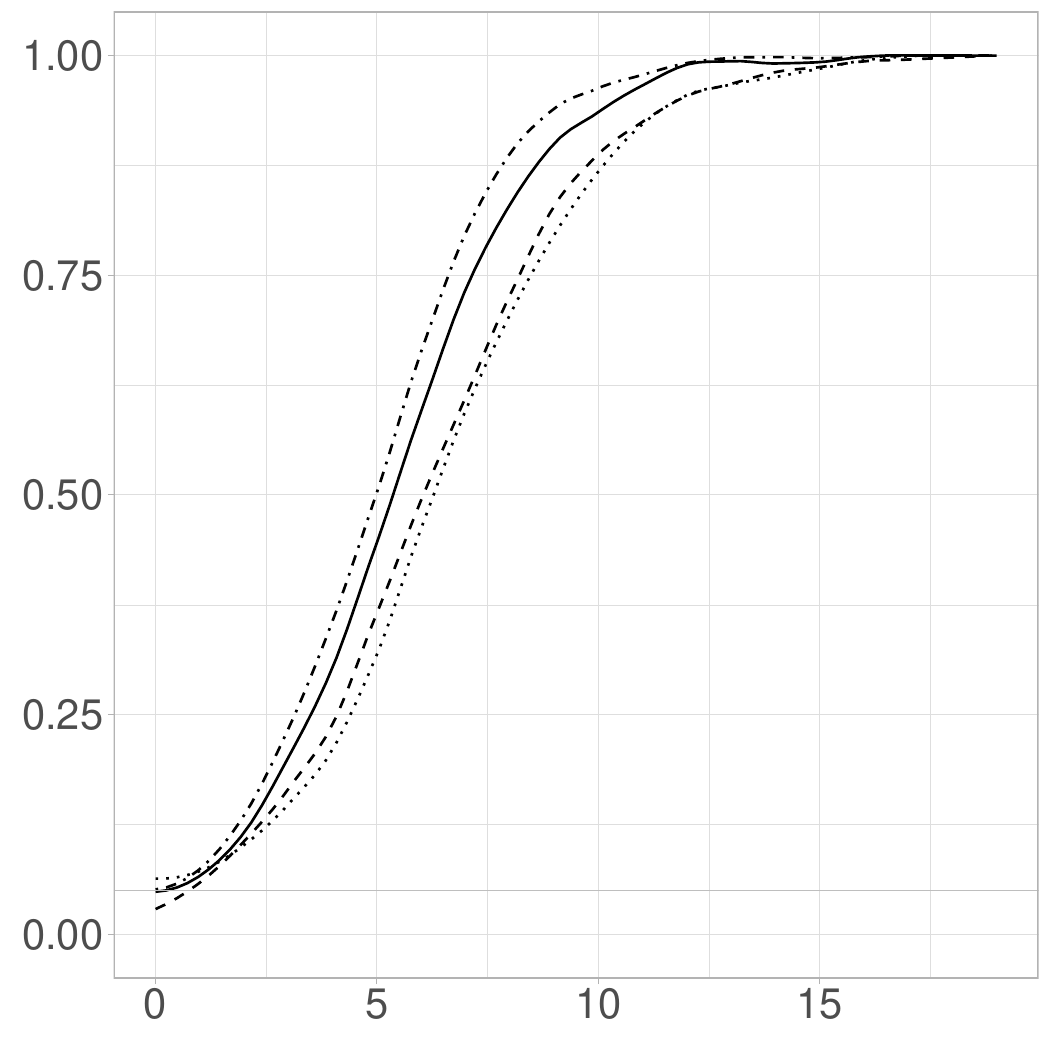}
		\end{center}
		\caption
        {\it  Simulated rejection probabilities of the test \eqref{eq_rejection_region} (solid line, LDD(1)), the test of \eqref{eq_generalized_rejection_region} (dot dashed line, LDD(3)), the test of \cite{zhang2022asymptotic} (dashed line, ZZPZ) and the test of \cite{li_and_chen_2012} (dotted line, LC)  in model \eqref{eq_model_3}. Left panels:   $(p,n)=(500, 100)$. Right panels: $(p,n)=(1000, 100)$. First row: $\bfz_{1,1}^{(i)} \sim \mathcal{N}(0,1)$, second row: $\bfz_{1,1}^{(i)}  \sim t_{7}/\sqrt{7/5}$, third row: $\bfz_{1,1}^{(i)}  \sim \Laplace(0, 1/\sqrt{2})$.}\label{fig_empirical_rej_model_3}
	\end{figure}

    \begin{figure}[p]
		\begin{center}	            
        \includegraphics[scale=0.36]{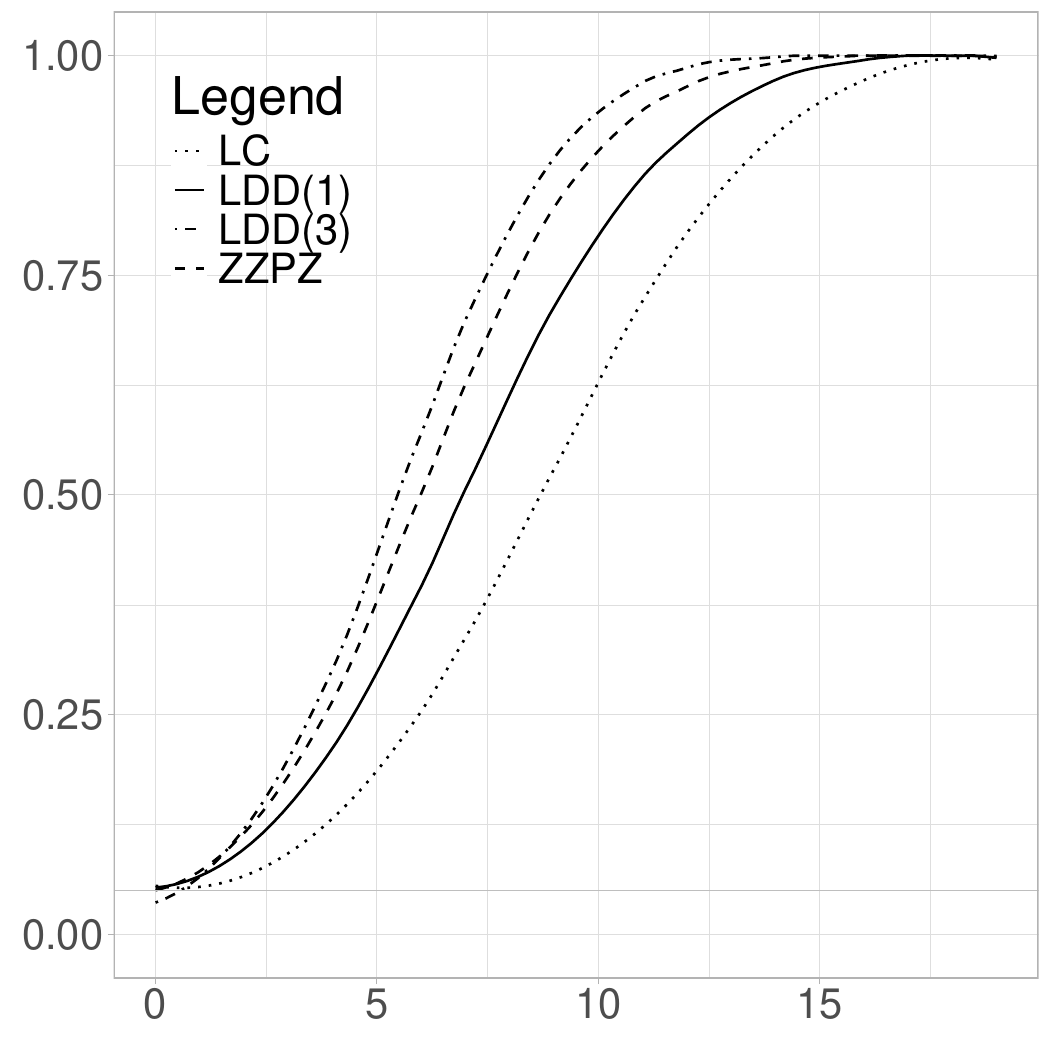}
        ~~~~ ~~~~
        \includegraphics[scale=0.36]{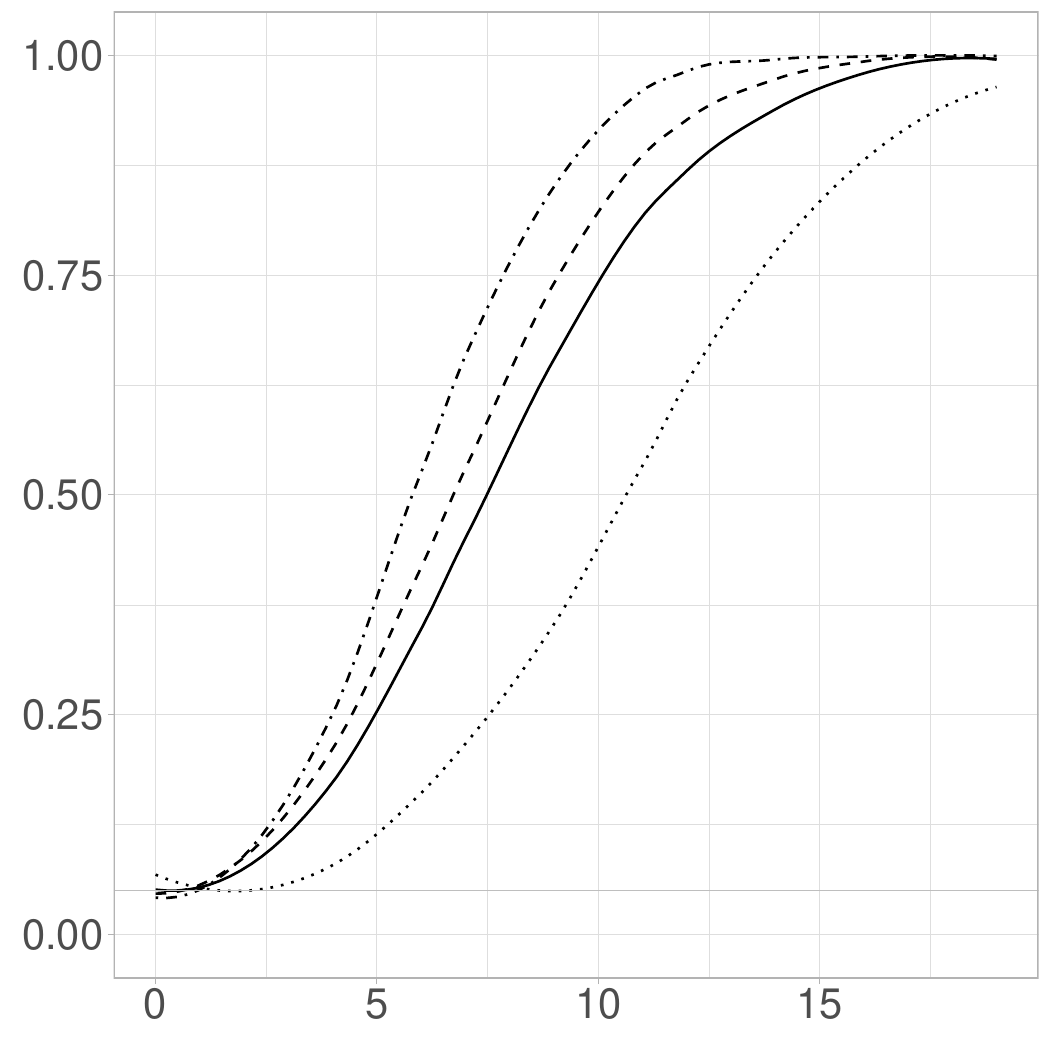}
        \includegraphics[scale=0.36]{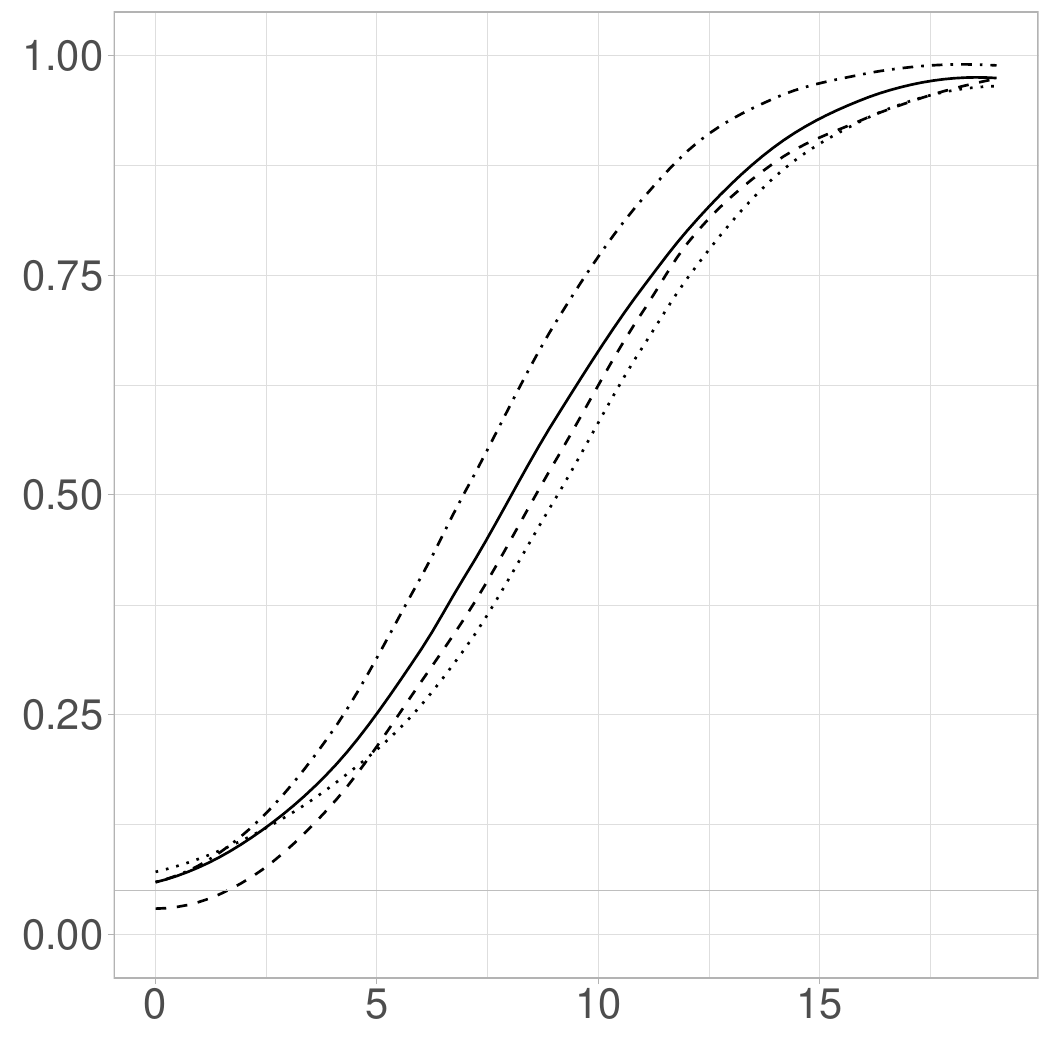}
        ~~~~ ~~~~
        \includegraphics[scale=0.36]{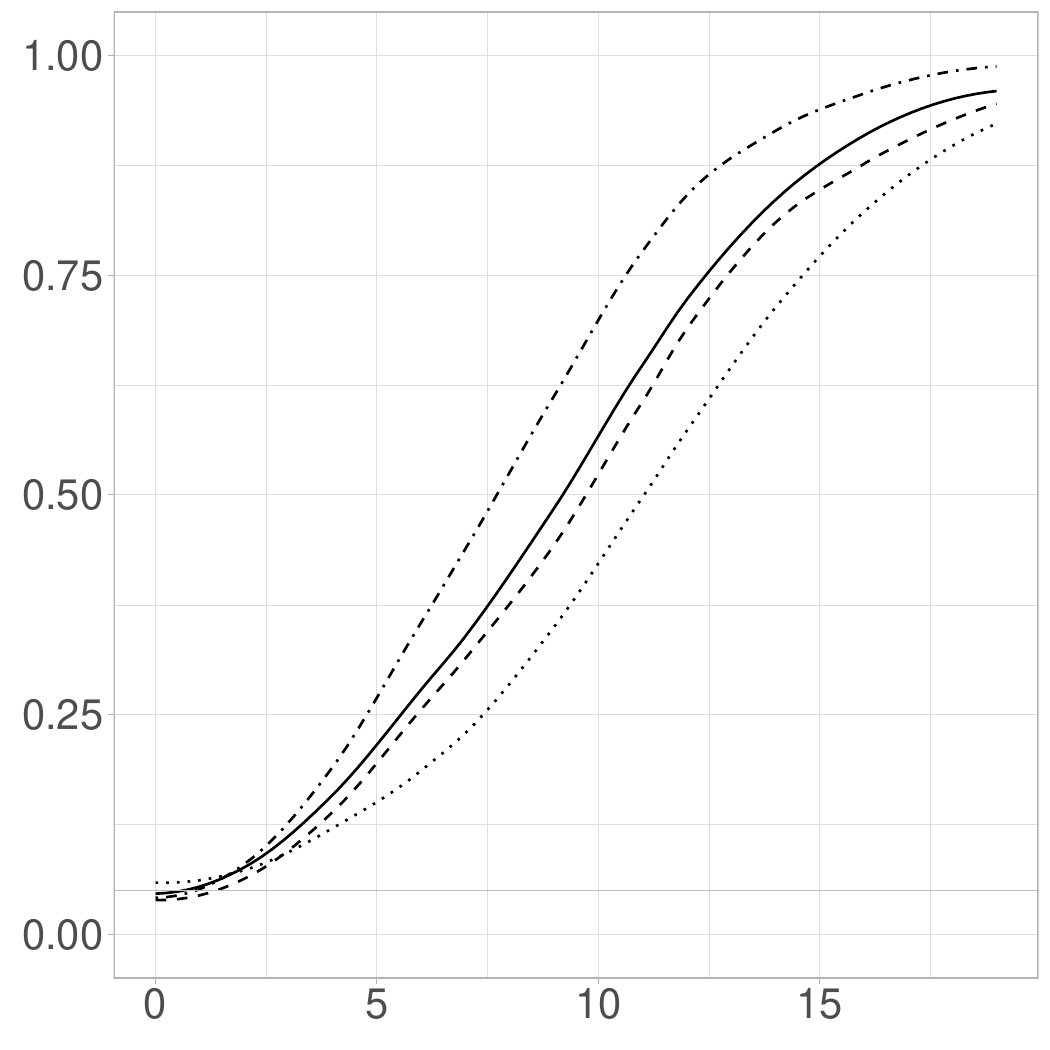}
        \includegraphics[scale=0.36]{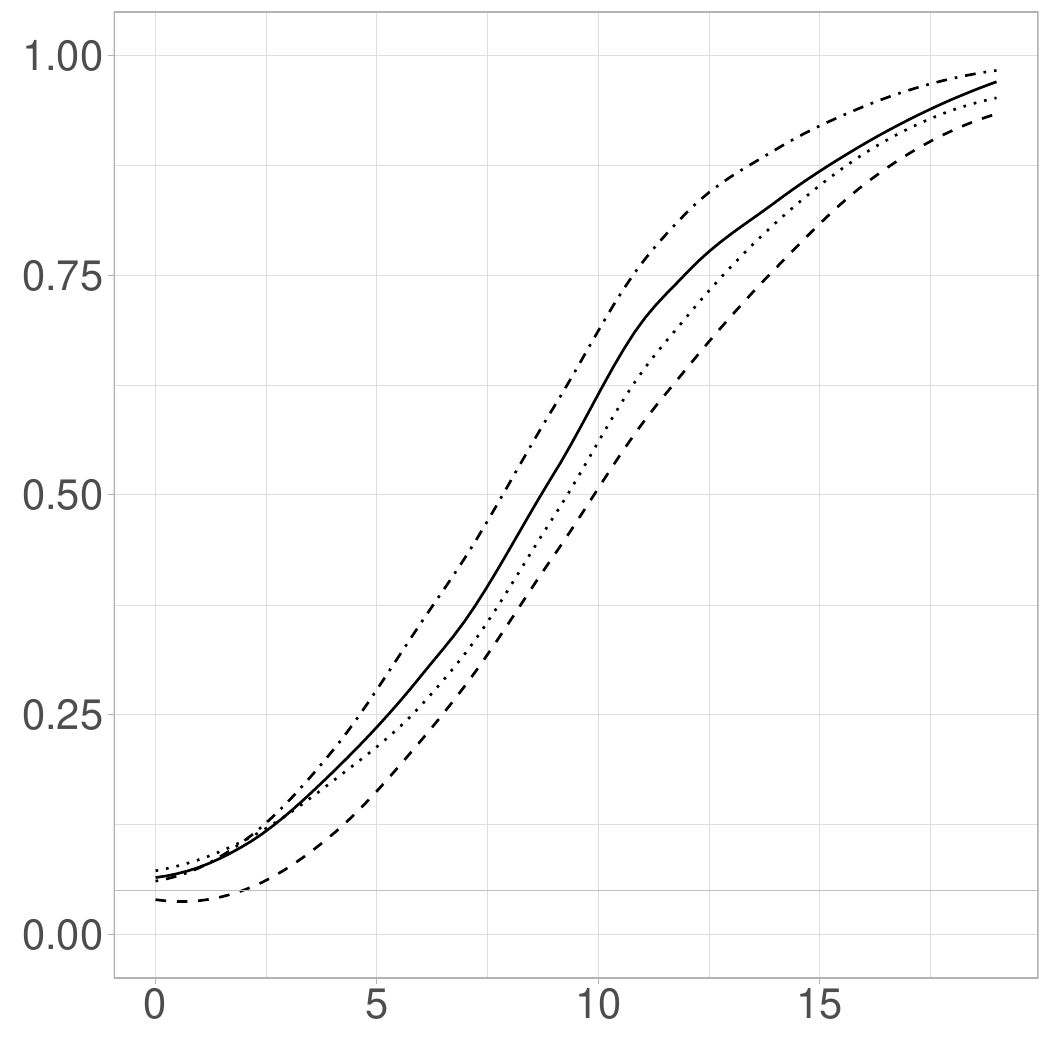}
        ~~~~ ~~~~
        \includegraphics[scale=0.36]{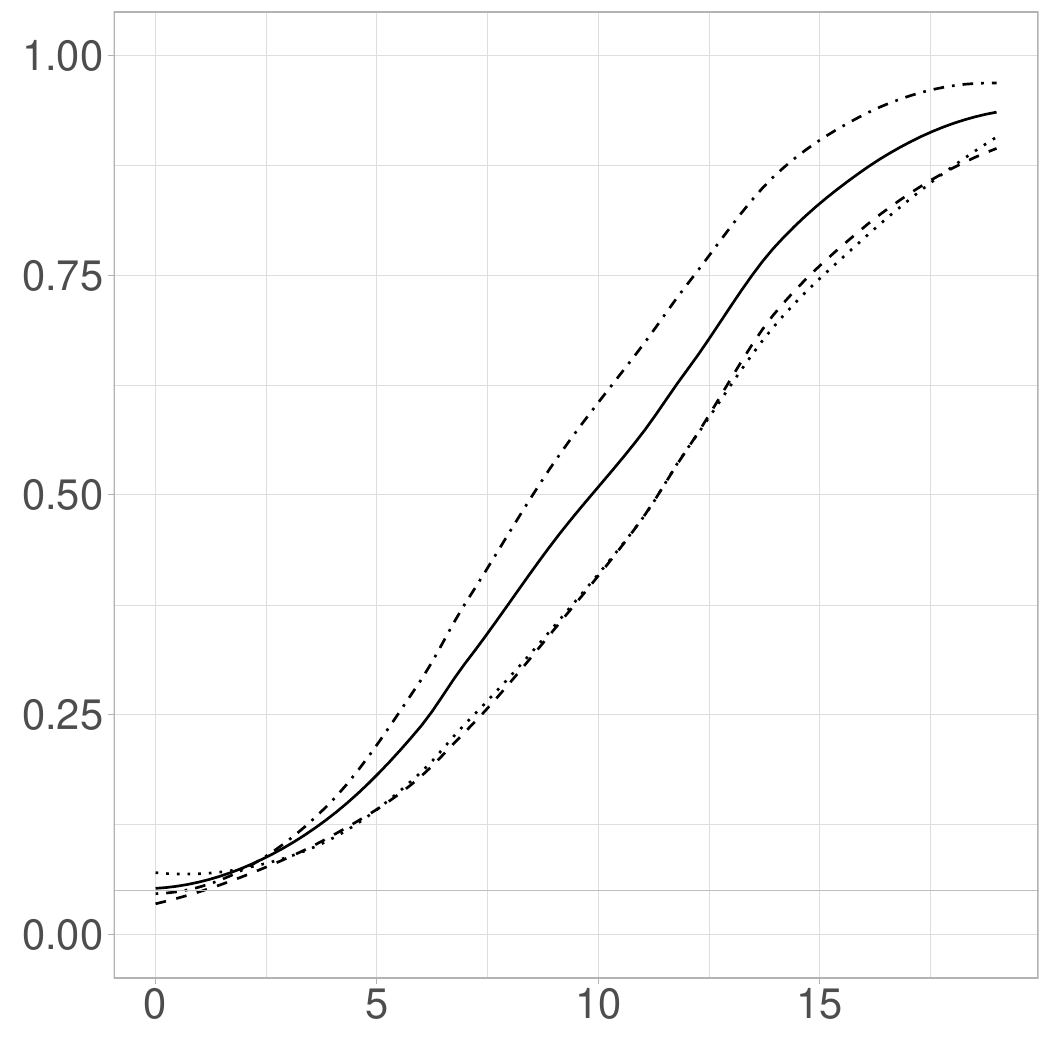}
		\end{center}
		\caption
        {Simulated rejection probabilities of the test \eqref{eq_rejection_region} (solid line, LDD(1)), the test of \eqref{eq_generalized_rejection_region} (dot dashed line, LDD(3)), the test of \cite{zhang2022asymptotic} (dashed line, ZZPZ) and the test of \cite{li_and_chen_2012} (dotted line, LC)  in model \eqref{eq_model_4}. Left panels:   $(p,n)=(500, 100)$. Right panels: $(p,n)=(1000, 100)$. First row: $\bfz_{1,1}^{(i)} \sim \mathcal{N}(0,1)$, second row: $\bfz_{1,1}^{(i)}  \sim t_{7}/\sqrt{7/5}$, third row: $\bfz_{1,1}^{(i)}  \sim \Laplace(0, 1/\sqrt{2})$.}\label{fig_empirical_rej_model_4}
	\end{figure}
\newpage 
\section{Proofs} \label{sec_proofs}
\subsection{Outline for the proofs of Theorem \ref{thm_asympt_ind} and Theorem \ref{thm_generalized_asympt_ind}} \label{sec_proof_outline}

The proofs of Theorems \ref{thm_asympt_ind} and \ref{thm_generalized_asympt_ind} are technically demanding and require substantial preparation, along with several reduction steps. To convey the main ideas without delving too much  into technical details, we provide an overview of the key steps in the proofs. Because the proof of Theorem \ref{thm_asympt_ind} closely follows that of Theorem \ref{thm_generalized_asympt_ind}, we confine our discussion to the main steps of proving the latter.

The proof of Theorem \ref{thm_generalized_asympt_ind} hinges on two auxiliary results. First, we need the consistency of the estimator $\hat{\sigma}_{n,1}$  of the variance $\Var(B_n^{(1)}+B_n^{(2)}-2C_n)$ under the null hypothesis, which is stated in  Lemma \ref{lem_consistent_estimators} below. This consistency allows us to replace the estimator $\hat{\sigma}_{n,1}$ in the definition of $T_{n,1}$ with its population counterpart $\sigma_{n,1}$. We denote the corresponding (infeasible) approximate of $T_{n,1}$ by $\breve{T}_{n,1}.$ As a result, we can concentrate on the limiting distribution of $(\breve{T}_{n,1}, \tilde{\mathbf{T}}_{n,2,m})^\top$ instead of that of $(T_{n,1}, \tilde{\mathbf{T}}_{n,2,m})^\top$.
Second, we need the asymptotic independence of the random variables $\breve{T}_{n,1}$ and $ \tilde{\mathbf{T}}_{n,2,m}$, which is stated in Lemma \ref{lem_ind_quad_form_frobenius_non_cen}.
 
The technical details for proving Theorem \ref{thm_generalized_asympt_ind} based on Lemma \ref{lem_consistent_estimators} and Lemma \ref{lem_ind_quad_form_frobenius_non_cen} can be found in Section \ref{sec_proof_thm_generalized_asympt_ind}.

A key contribution of our work lies in verifying the asymptotic independence stated in Lemma \ref{lem_ind_quad_form_frobenius_non_cen} and thus, the proof of Theorem \ref{thm_generalized_asympt_ind} is organized around proving this main lemma.  
In the following, we outline the main reduction steps in this proof, which can be found in Section \ref{sec_proof_asymptotical_independence_ind_quad_non_cen}.

As in previous related works \cite{zhang2022asymptotic, bai2008limit}, the supercritical sample eigenvalues are associated with a random quadratic form, meaning that the statement in Lemma \ref{lem_ind_quad_form_frobenius_non_cen} can be reduced to a corresponding statement for these quadratic forms (see Lemma \ref{lem_quad_form_frobenius_non_cen} below). This result, in turn, can be further reduced to the case of centered data $\boldsymbol{\mu}^{(i)}=\mathbf{0}_{p}$ (Lemma \ref{lem_quad_form_frobenius_cen}).
Having established these reductions, Lemma \ref{lem_ind_quad_form_frobenius_non_cen} is eventually implied by Lemma \ref{lem_quad_form_frobenius_cen}. Its proof is based on a martingale representation for $\breve{T}_{n,1}$ and the quadratic forms associated with the leading sample eigenvalues, which facilitates the application of the martingale CLT.
Here, to show the asymptotic independence, the crucial step lies in verifying that the terms contributing to the covariance are asymptotically negligible (see Section \ref{subsec_proof_cross_terms}).
 
\subsection{Proof of Theorem \ref{thm_generalized_asympt_ind}}\label{sec_proof_thm_generalized_asympt_ind}
First, we present the two auxiliary results. For this purpose, we need to introduce some notations and decompose the diagonal matrix $\bfLambda_n^{(i)}=\diag\lb\lambda_{1}(\bfSigma_n^{(i)}), \ldots ,\lambda_{p}(\bfSigma_n^{(i)})\rb$ as 
\begin{align*}
        \bfLambda_n^{(i)}=\begin{pmatrix}
        \bfLambda_S^{(i)} & 0
        \\
        0 & \bfLambda_{P, n, K}^{(i)}
        \end{pmatrix},
    \end{align*}
    where 
    \begin{align*}
        \bfLambda_S^{(i)} = \diag\lb\alpha_1^{(i)}, \ldots, \alpha_{K}^{(i)}\rb
    \end{align*}
   contains the leading $K$ supercritical eigenvalues, and
    \begin{align*}
        \bfLambda_{P,n, K}^{(i)} = \diag\lb\lambda_{K+1}(\bfSigma_n^{(i)}), \ldots, \lambda_{p}(\bfSigma_n^{(i)})\rb
    \end{align*}
   contains the remaining eigenvalues.
Next, we define the $2K$-dimensional random vector
\begin{align*}
        \mathbf{\Lambda}_{K, K}=\Bigg( &\sqrt{n}\lcb\lambda_1(\bfS_n^{(1)})-\theta_{1,n}^{(1)}\rcb, \ldots, \sqrt{n} \lcb\lambda_{K}(\bfS_n^{(1)})-\theta_{K,n}^{(1)}\rcb,
        \\
        &\sqrt{n}\lcb\lambda_1(\bfS_n^{(2)})-\theta_{1,n}^{(2)}\rcb, \ldots, \sqrt{n} \lcb\lambda_{K}(\bfS_n^{(2)})-\theta_{K,n}^{(2)}\rcb \Bigg)^\top,
\end{align*}
where 
$\theta_{j,n}^{(i)}= \psi_{n,K}^{(i)}(\alpha_{j}^{(i)})$ and $\psi_{n,K}^{(i)}$ is defined in \eqref{eq_def_generalized_psi_n}.
\newline
Furthermore, we define $\breve{T}_{n,1}$ as the population counterpart of $T_{n,1}$, where the empirical variance $\hat{\sigma}_{n,1}^2$ is replaced by the population variance $\sigma_{n,1}^2$. In addition, $\breve{T}_{n,1}$ is defined such that $T_{n,1}$ is unbiased under $\mathsf{H}_{0}$ as well as under $\mathsf{H}_{1}$. 
To be more precise, we set
 \begin{align}\label{def_population_t_frob_non_cen}
     \breve{T}_{n,1}=\frac{B_{n}^{(1)}+B_{n}^{(2)}-2C_n-\lnorm \bfSigma_n^{(1)} - \bfSigma_n^{(2)}\rnorm_F^2}{\sigma_{n,1}},
 \end{align}
where
     \begin{align*}
        \sigma_{n,1}^2 &= \sum_{i=1}^{2} \bigg[\frac{4}{n^2} \tr (\bfSigma_n^{(i)^2})+\frac{8}{n}\tr\lb\lcb \bfSigma_n^{(i)^2}-\bfSigma_{n}^{(1)}\bfSigma_{n}^{(2)}\rcb^2\rb
        \\
        &\phantom{{}=\sum_{i=1}^{2}\bigg[{}}+\frac{4(\gamma_4^{(i)}-3)}{n}\tr\lb\lcb\bfSigma_n^{(i)^{1/2}}\lb \bfSigma_n^{(1)}-\bfSigma_n^{(2)}\rb\bfSigma_n^{(i)^{1/2}}\rcb \circ \lcb\bfSigma_n^{(i)^{1/2}}\lb \bfSigma_n^{(1)}-\bfSigma_n^{(2)}\rb\bfSigma_n^{(i)^{1/2}}\rcb \rb\bigg]\nonumber
        \\
        &\phantom{{}={}}+\frac{8}{n^2}\tr^2\lb \bfSigma_n^{(1)}\bfSigma_n^{(2)}\rb.\nonumber
    \end{align*}
Here, $\textbf{A}\circ\textbf{B}$ denotes the Hadamard product for two matrices of the same dimension $\textbf{A}$ and $\textbf{B}$. 
Moreover, we note that the population variance $\sigma_{n,2}^2$ of $\sqrt{n}(\lambda_1(\bfS_n^{(1)})-\lambda_1(\bfS_n^{(2)}))$ is given by the sum 
    \begin{align*}
        \sigma_{n,2}^2 = \sigma_{\spi,1,n}^{(1)^2}+\sigma_{\spi,1,n}^{(2)^2},
    \end{align*}
    with 
    \begin{align} \label{def_var_spike}
        \sigma_{\spi,1,n}^{(i)^2}= (\gamma_4^{(i)}-3)\alpha_1^{(i)^2}(\psi^{(i)^\prime}(\alpha_{1}^{(i)}))^2\sum_{j=1}^{p}u_{j,1,n}^{(i)^4}+2\alpha_1^{(i)^2}\psi^{(i)^\prime}(\alpha_1^{(i)}), \quad  i=1,2,
    \end{align}
    where $\gamma_4^{(i)}:=\E\lsb z_{1,1}^{(i)}\rsb$ is defined as the kurtosis of the i.i.d. entries of the random vectors $\bfz_1^{(i)}, \ldots, \bfz_n^{(i)}$ defined in assumption \ref{ass_random_vectors} and the function $\psi^{(i)^\prime}$ is given in \eqref{eq_def_derivative_psi}. 
    We emphasize that $\sigma_{n,1}^2$ and $\sigma_{\spi,1,n}^{(i)^2}$ are positive and bounded from below. A detailed analysis can be found in Lemma \ref{lem_lower_bounded_variances}. 
\newline
Finally, we are now in the position to formulate the two auxiliary results needed in the proof of Theorem \ref{thm_generalized_asympt_ind}.
The proofs of the two lemmas is postponed to Section \ref{sec_proof_asymptotical_independence_ind_quad_non_cen} and \ref{sec_proof_consistent_estimators}, respectively.

\begin{lemma}\label{lem_consistent_estimators}
     Suppose that assumptions \ref{ass_p_n} - \ref{ass_leading_spiked_ev} are satisfied. Then, under the null hypothesis $\mathsf{H}_{0}$ we have
    \begin{align*}
        \frac{\sigma_{n,1}}{\hat{\sigma}_{n,1}}\conp 1, \quad n\to \infty,
    \end{align*}
    while under the alternative $\mathsf{H}_{1}$ we have
    \begin{align*}
        \frac{2}{n}\frac{\tr(\bfSigma_n^{(1)^2}+\bfSigma_n^{(2)^2})}{\hat{\sigma}_{n,1}}\conp 1, \quad n\to \infty ~.
    \end{align*}
    Furthermore, suppose that assumptions \ref{ass_p_n} - \ref{ass_population_lsd} and \ref{ass_supercritical_ev_generalized} are satisfied, then we have 
    \begin{align*}
        \frac{\sigma_{n,2}^2}{\hat{\sigma}_{n,2}^2}\conp 1, \quad n\to \infty ~.
    \end{align*}
    Finally, if \ref{ass_p_n} - \ref{ass_population_lsd}, \ref{ass_supercritical_ev_generalized} and \ref{ass_asympt_var_cov} are satisfied, then for $m\in \lcb 1, \ldots, K\rcb$ we have 
    \begin{align*}
    (\bfSigma_{E,m})_{k, \ell}-(\hat{\bfSigma}_{E,m,n})_{k,\ell} \conp 0, \quad n \to \infty ~.
    \end{align*}
\end{lemma}
\begin{lemma}
    \label{lem_ind_quad_form_frobenius_non_cen}
    Suppose that assumptions \ref{ass_p_n} - \ref{ass_population_lsd}, \ref{ass_supercritical_ev_generalized} and \ref{ass_asympt_var_cov} are satisfied. Then, 
    $\breve{T}_{n,1}$ and $\bfLambda_{K,K}$ are asymptotically independent, and it holds 
    \begin{align*}
        \breve{T}_{n,1}\cond \mathcal{N}(0,1), \quad 
        \bfLambda_{K,K} \cond \mathcal{N}_{2K}(\mathbf{0}_{2K}, \bfSigma_{\Lambda}), \quad n\to\infty,
    \end{align*}
  where
    \begin{align*}
        \bfSigma_{\Lambda}=\begin{pmatrix}
        \bfSigma_{\Lambda}^{(1)} & 0
        \\
        0 & \bfSigma_{\Lambda}^{(2)}
        \end{pmatrix},
    \end{align*}
 and for $i=1,2$,   $\bfSigma_{\Lambda}^{(i)}$ is a $\R^{K \times K}$ matrix and its elements $(\bfSigma_{\Lambda}^{(i)})_{j, k}$ are given by
    \begin{align*}
        (\bfSigma_{\Lambda}^{(i)})_{j, k}=
        \begin{cases}
            \sigma_{\spi,k}^{(i)^2} & j=k
            \\
            \sigma_{\spi,j, k}^{(i)} & j\neq k
        \end{cases}
    \end{align*}
and the elements  $\sigma_{\spi,k}^{(i)^2}$ and $\sigma_{\spi, j,k}^{(i)}$ are defined in assumption \ref{ass_asympt_var_cov}.
\end{lemma}
\begin{proof}[Proof of Theorem \ref{thm_generalized_asympt_ind}]
    Note, if $\mathsf{H}_{0}$ is satisfied then $\lnorm \bfSigma_n^{(1)} - \bfSigma_n^{(2)}\rnorm_F^2=0$ and $\theta_{j,n}^{(1)}=\theta_{j,n}^{(2)}$ for $j=1, \ldots, m$ holds. Thus, we can write
    \begin{align*}
    T_{n,1}= \frac{\sigma_{n,1}}{\hat{\sigma}_{n,1}}\breve{T}_{n,1}   \quad \text{and} \quad \tilde{\mathbf{T}}_{n,2,m}= \sqrt{n} \big( &\lambda_1(\bfS_n^{(1)})- \theta_{1,n}^{(1)}- (\lambda_1(\bfS_n^{(2)})-\theta_{1,n}^{(2)}), \ldots, 
    \\
    &\lambda_m(\bfS_n^{(1)})- \theta_{1,m}^{(1)}- (\lambda_m(\bfS_n^{(2)})-\theta_{1,m}^{(2)})\big).
    \end{align*}
    Combining Lemma \ref{lem_consistent_estimators} and \ref{lem_ind_quad_form_frobenius_non_cen} with Slutsky's Theorem yields 
    \begin{align*}
        T_{n,1} \cond \mathcal{N}(0,1) \quad \text{and} \quad \tilde{\mathbf{T}}_{n,2,m}^\top\cond \mathcal{N}_m(\mathbf{0}_m,\bfSigma_{E,m}), \quad n\to\infty,
    \end{align*}
    where we used the fact that the two samples are independent.
    Finally, by Lemma \ref{lem_ind_quad_form_frobenius_non_cen} it follows that $T_{n,1}$ and $\tilde{\mathbf{T}}_{n,2,m}^\top$ are asymptotically independent, which implies that
    \begin{align*}
        (T_{n,1}, \tilde{\mathbf{T}}_{n,2, m})^\top \cond \mathcal{N}_{m+1}(\mathbf{0}_{m+1}, \bfSigma_{m}), \quad n\to\infty.
    \end{align*}
\end{proof}

\subsection{Proof of Lemma \ref{lem_ind_quad_form_frobenius_non_cen}}\label{sec_proof_asymptotical_independence_ind_quad_non_cen}
    As we will see in this section,  the supercritical sample eigenvalues are closely associated with a certain random quadratic form. Thus, the statement of Lemma \ref{lem_ind_quad_form_frobenius_non_cen} can be traced back to the weak convergence of these  quadratic forms. 
    For their  definition, we introduce a partition for the orthogonal matrices $\bfU_n^{(i)}=(\bfU_{n,1}^{(i)}, \bfU_{n,2}^{(i)})$,
    where $\bfU_{n,1}^{(i)}$ is a $p\times K$ sub-matrix of $\bfU_n^{(i)}$ and $\bfU_{n,2}^{(i)}$ is a $p\times (p-K)$ sub-matrix of $\bfU_n^{(i)}$. With these notations, we define
    \begin{align*}
        \bfSigma_{S,n}^{(i)} = \bfU_{n,1}^{(i)}\bfLambda_S^{(i)}\bfU_{n,1}^{(i)^\top}, \bfSigma_{P, n}^{(i)}=\bfU_{n,2}^{(i)}\bfLambda_{P, n, K}^{(i)}\bfU_{n,2}^{(i)^\top} \in \mathbb{R}^{p\times p},\quad i=1,2.
    \end{align*}
    Recall, for $i=1,2$ and $j=1,\ldots, K$, $\mathbf{u}_{j,n}^{(i)}$ is the $j$-th column of the matrix $\mathbf{U}_{n}^{(i)}$, in particular of $\mathbf{U}_{n,1}^{(i)}$. We introduce for $i=1,2$ and $j=1, \ldots, K$ the quantities\begin{align}\label{eq_def_quadratic_form_definition_non_cen}
        \begin{split}
    \textbf{A}_{n}^{(i)}(\theta_{j, n}^{(i)}) &= \textbf{I}_n -\frac{1}{\theta_{j,n}^{(i)}} \frac{1}{n}\mathbf{\Delta}_n\bfZ_n^{(i)^\top} \bfSigma_{P,n}^{(i)}\bfZ_n^{(i)},
    \\
    Q_{j, n}^{(i)} = Q_{n}^{(i)}(\alpha_j^{(i)}) &= \sqrt{n}\lb \frac{1}{n} \mathbf{u}_{j,n}^{(i)^\top}\bfZ_n^{(i)}\lcb \textbf{A}_{n}^{(i)}(\theta_{j, n}^{(i)})\rcb^{-1}\mathbf{\Delta}_n\bfZ_n^{(i)^\top}\mathbf{u}_{j,n}^{(i)}-\frac{\theta_{j,n}^{(i)}}{\alpha_j^{(i)}}\rb,    \end{split}
    \end{align}
    where  $\bfZ_n^{(i)}= (\bfz_1^{(i)}, \ldots, \bfz_n^{(i)})$ and $\mathbf{\Delta}_n=\textbf{I}_n -\frac{1}{n}\mathbf{1}_n\mathbf{1}_n^\top$,  while $\mathbf{1}_n$ denotes the $n$-dimensional vector with all entries are equal $1$. Furthermore, we define the $2K$-dimensional random vector $\mathbf{L}_{K,K}$ as the counterpart to $\mathbf{\Lambda}_{K,K}$, where we replace the sample eigenvalues by the corresponding quadratic form, that is 
    \begin{align*}
        \mathbf{L}_{K,K}=\lb Q_{1,n}^{(1)}, \ldots, Q_{K,n}^{(1)}, Q_{1,n}^{(2)}, \ldots, Q_{K,n}^{(2)}\rb^\top.
    \end{align*}
    Finally, we define the matrix
    \begin{align*}
        \bfSigma_{L}=\begin{pmatrix}
        \bfSigma_{L}^{(1)} & 0
        \\
        0 & \bfSigma_{L}^{(2)}
        \end{pmatrix},
    \end{align*}
     where $\bfSigma_{L}^{(i)}$ is a $K \times K$ matrix and the elements $(\bfSigma_{L}^{(i)})_{j, k}$ are given by 
    \begin{align*}
        (\bfSigma_{L}^{(i)})_{j, k}&=
        \begin{cases}
            \lim\limits_{n\to \infty} \frac{\psi^{(i)^2}\lb \alpha_k^{(i)}\rb}{\alpha_{k}^{(i)^2}}\lcb \lb \gamma_4^{(i)}-3\rb\sum\limits_{\ell=1}^{p} u_{\ell,k,n}^{(i)^4}+\frac{2}{\psi^{(i)^\prime}\lb \alpha_k^{(i)}\rb}\rcb & \text{if } j=k
            \\
            \lim\limits_{n\to \infty}\frac{\psi^{(i)}\lb\alpha_j^{(i)}\rb\psi^{(i)}\lb \alpha_k^{(i)}\rb}{\alpha_{j}^{(i)}\alpha_k^{(i)}}\lb \gamma_4^{(i)}-3\rb\sum\limits_{\ell=1}^{p}u_{\ell, j,n}^{(i)^2}u_{\ell, k,n}^{(i)^2} & \text{if }j\neq k
        \end{cases}
        \\
        &=
        \begin{cases}
            \frac{\psi^{(i)^2}\lb \alpha_k^{(i)}\rb }{\alpha_{k}^{(i)^4}\lcb\psi^{(i)^\prime}\lb \alpha_k^{(i)}\rb\rcb^2}\sigma_{\spi,k}^{2} & \text{if } j=k
            \\
            \frac{\psi^{(i)}\lb \alpha_j^{(i)}\rb\psi^{(i)}\lb \alpha_k^{(i)}\rb}{\alpha_j^{(i)}\alpha_k^{(i)}\psi^{(i)^\prime}\lb \alpha_{j}^{(i)}\rb\psi^{(i)^\prime}\lb \alpha_{k}^{(i)}\rb}\sigma_{\spi,j,k} & \text{if } j\neq k
        \end{cases}
    \end{align*}
   and  $\sigma_{\spi, k}^2$ and $\sigma_{\spi,j,k}$ are defined in assumption \ref{ass_asympt_var_cov}.
    Then, the proof of Lemma \ref{lem_ind_quad_form_frobenius_non_cen} is facilitated by the following result, which yields  the asymptotic independence of the random variables $\breve T_{n,1}$ defined in \eqref{def_population_t_frob_non_cen} and $\mathbf{L}_{K, K}.$
    \begin{lemma}
    \label{lem_quad_form_frobenius_non_cen}
    Suppose that assumptions \ref{ass_p_n} - \ref{ass_population_lsd}, \ref{ass_supercritical_ev_generalized}, and \ref{ass_asympt_var_cov} are satisfied. Then, the random variables 
    $\breve{T}_{n,1}$ and $\mathbf{L}_{K,K}$ are asymptotically independent, and it holds
    \begin{align*}
        \breve{T}_{n,1} \cond \mathcal{N}(0,1), \quad \mathbf{L}_{K,K} \cond \mathcal{N}_{2K}(\mathbf{0}_{2K}, \bfSigma_{L}), \quad n\to\infty.
    \end{align*}
\end{lemma}
The proof of Lemma \ref{lem_quad_form_frobenius_non_cen} is postponed to Section \ref{sec_proof_asymptotic_independence_quad_form_frobenius_non_cen}, and  we continue with the proof of Lemma \ref{lem_ind_quad_form_frobenius_non_cen}. 
\begin{proof}[Proof of Lemma \ref{lem_ind_quad_form_frobenius_non_cen}]
    Similar to the identity $(\text{S}.3.9)$ in the supplement to \cite{zhang2022asymptotic}, we obtain for $i=1,2$ and $j=1, \ldots, K$ that
    \begin{align*}
    \sqrt{n}(\lambda_j(\bfS_n^{(i)})-\theta_{j,n}^{(i)}) = \frac{Q_{j,n}^{(i)}}{\theta_{j,n}^{(i)}R_{j,n}^{(i)}}
    \end{align*}
    holds, where
    \begin{align*}
        R_{j,n}^{(i)}= \frac{1}{n}\frac{1}{\alpha_{j}^{(i)}\theta_{j,n}^{(i)}}\mathbf{u}_{j,n}^{(i)^\top}  \bfZ_n^{(i)}\lcb \bfA_n^{(i)}\lb \alpha_{j}^{(i)}\rb\rcb^{-1}\lcb \bfA_n^{(i)}\lb \theta_{j,n}^{(i)}\rb\rcb^{-1}\mathbf{\Phi}_n\bfZ_n^{(i)^\top}\mathbf{u}_{j,n}^{(i)}.
    \end{align*}
    Thus, $\mathbf{\Lambda}_{K,K}$ can be represented as
    \begin{align*}
    \mathbf{\Lambda}_{K,K}=\lb \frac{Q_{1,n}^{(1)}}{\theta_{1,n}^{(1)}R_{1,n}^{(1)}}, \ldots, \frac{Q_{K,n}^{(1)}}{\theta_{K,n}^{(1)}R_{K,n}^{(1)}}, \frac{Q_{1,n}^{(2)}}{\theta_{1,n}^{(2)}R_{1,n}^{(2)}}, \ldots, \frac{Q_{K,n}^{(2)}}{\theta_{K,n}^{(2)}R_{K,n}^{(2)}}\rb^\top.
    \end{align*}
    Using similar arguments to the equation $(\text{S}.4.4)$ in the supplement to \cite{zhang2022asymptotic}, it follows for $i=1,2$ and $j=1, \ldots, K$ that
    \begin{align*}
        R_{j,n}^{(i)} - \tilde{R}_{j,n}^{(i)} \conp 0, \quad n\to\infty,
    \end{align*}
    where
    \begin{align*}
        \tilde{R}_{j,n}^{(i)}= \frac{1}{n}\frac{1}{\alpha_j^{(i)}\theta_{j,n}^{(i)}}\mathbf{u}_{j,n}^{(i)^\top}  \bfZ_n^{(i)}\lcb \tilde{\bfA}_n^{(i)}\lb \alpha_j^{(i)}\rb\rcb^{-1}\lcb \tilde{\bfA}_n^{(i)}\lb \theta_{j,n}^{(i)}\rb\rcb^{-1}\bfZ_n^{(i)^\top}\mathbf{u}_{j,n}^{(i)}
    \end{align*}
    and  
    \begin{align}\label{eq_def_matrix_A_centred}
        \tilde{\bfA}_{n}^{(i)}(\theta_{j,n}^{(i)})=\bfI_{n}-\frac{1}{\theta_{j,n}^{(i)}}\frac{1}{n} \bfZ_n^{(i)^\top}\bfSigma_{P,n}^{(i)}\bfZ_n^{(i)}.
    \end{align}
    Furthermore, for $i=1,2$ and $j=1,\ldots, K$ we have due to equation $(\text{S}.3.17)$ in the supplement to \cite{zhang2022asymptotic} that
    \begin{align*}
    \tilde{R}_{j,n}^{(i)} \conp \frac{1}{\alpha_{j}^{(i)^2}\psi^{(i)^\prime}\lb\alpha_{j}^{(i)}\rb}, \quad n\to\infty.
    \end{align*}
    Since, we have 
    \begin{align*}
        \theta_{j,n}^{(i)} \conas \psi^{(i)}\lb \alpha_j^{(i)}\rb, \quad n\to\infty,
    \end{align*}
    for $i=1,2$ and $j=1, \ldots, K$, it follows from Lemma \ref{lem_quad_form_frobenius_non_cen} that $\breve{T}_{n,1}$ and $\mathbf{\Lambda}_{K,K}$ are asymptotically independent. In particular, as a consequence of Slutsky's theorem, it holds for the marginals
    \begin{align*}
        \breve{T}_{n,1}\cond \mathcal{N}(0,1), \quad 
        \bfLambda_{K,K} \cond \mathcal{N}_{2K}(\mathbf{0}_{2K}, \bfSigma_{\Lambda}), \quad n\to\infty.
    \end{align*}
\end{proof}

\subsection{Proof of Lemma \ref{lem_quad_form_frobenius_non_cen}}\label{sec_proof_asymptotic_independence_quad_form_frobenius_non_cen}
The proof of Lemma \ref{lem_quad_form_frobenius_non_cen} is facilitated by the corresponding statement in the case that $\E\big [ \bfz_j^{(i)}\big ] =\boldsymbol{\mu}^{(i)}=\mathbf{0}_p$.
Thus, we consider $\tilde{T}_{n,1}$ instead of $\breve{T}_{n,1}$, where for $i=1,2$ and $j=1, \ldots, n$ we replace each $\bfx_j^{(i)}$ in $\breve{T}_{n,1}$ by  $\tilde{\bfx}_j^{(i)}=\bfx_j^{(i)}-\boldsymbol{\mu}^{(i)}$. To be more precise, we define 
\begin{align*}
    \tilde{T}_{n,1}= \frac{\tilde{B}_{n}^{(1)}+\tilde{B}_n^{(2)}-2\tilde{C}_n-\lnorm \bfSigma_n^{(1)} - \bfSigma_n^{(2)}\rnorm_F^2}{\sigma_{n,1}},
\end{align*}
where $\tilde{B}_n^{(i)}$ and $\tilde{C}_{n}$ are given by substituting $\bfx_j^{(i)}$ with $\tilde{\bfx}_j^{(i)}$ in the definitions of $B_{n}^{(i)}$ and $C_n$. Similar to the definitions of the quantities in \eqref{eq_def_quadratic_form_definition_non_cen}, we introduce 
    \begin{align*}
    \tilde{Q}_{j,n}^{(i)}=\tilde{Q}_n^{(i)}(\alpha_{j}^{(i)})&= \sqrt{n}\lb \frac{1}{n} \mathbf{u}_{j,n}^{(i)^\top}\bfZ_n^{(i)}\lcb\tilde{ \textbf{A}}_{n}^{(i)}(\theta_{j, n}^{(i)})\rcb^{-1}\bfZ_n^{(i)^\top}\mathbf{u}_{j,n}^{(i)}-\frac{\theta_{j,n}^{(i)}}{\alpha_{j}^{(i)}}\rb, \\    \widetilde{\mathbf{L}}_{K,K} & =\lb \tilde{Q}_{1,n}^{(1)}, \ldots, \tilde{Q}_{K,n}^{(1)}, \tilde{Q}_{1,n}^{(2)}, \ldots, \tilde{Q}_{K,n}^{(2)}\rb^\top,
    \end{align*}
    where the matrix $\tilde{\textbf{A}}_n^{(i)}(\theta_{j,n}^{(i)})$ is defined in \eqref{eq_def_matrix_A_centred}.
\begin{lemma}\label{lem_quad_form_frobenius_cen}
    Suppose that assumptions \ref{ass_p_n} - \ref{ass_population_lsd}, \ref{ass_supercritical_ev_generalized} and \ref{ass_asympt_var_cov} are satisfied. Then, $\tilde{T}_{n,1}$ and $\widetilde{\mathbf{L}}_{K,K}$ are asymptotically independent, and it holds 
    \begin{align*}
        \tilde{T}_{n,1} \cond \mathcal{N}(0,1), \quad 
        \widetilde{\mathbf{L}}_{K,K} \cond \mathcal{N}_{2K}(\mathbf{0}_{2K}, \bfSigma_{L}), 
    \end{align*}
    where the matrix $\bfSigma_{L}$ is defined in Lemma \ref{lem_quad_form_frobenius_non_cen}.
\end{lemma}

\begin{proof}[Proof of Lemma \ref{lem_quad_form_frobenius_non_cen}]
We note that $\breve{T}_{n,1}$ defined in \eqref{def_population_t_frob_non_cen} and $\tilde{T}_{n,1}$ are identical distributed, because $\breve{T}_{n,1}$ is invariant under location shifts  \citep[see][]{li_and_chen_2012}.
Consequently, we obtain by Lemma \ref{lem_quad_form_frobenius_cen} that
\begin{align*}
    \breve{T}_{n,1}\cond \mathcal{N}(0,1), \quad n \to \infty ~.
\end{align*}
From  equation $(\text{S}.4.4)$ in the supplement to \cite{zhang2022asymptotic}, we obtain for  the quadratic forms $Q_{j,n}^{(i)}$ and $\tilde{Q}_{j,n}^{(i)}$ that $Q_{j,n}^{(i)} -\tilde{Q}_{j,n}^{(i)}=o_{\PR}(1)$ ($i=1,2$ and $j=1,\ldots, K$). Therefore, it follows from Lemma \ref{lem_quad_form_frobenius_cen} that 
\begin{align*}
    \mathbf{L}_{K,K} \cond \mathcal{N}_{2K}(\mathbf{0}_{2K}, \bfSigma_{L}), \quad n \to \infty ~.
\end{align*}
Finally, Lemma \ref{lem_quad_form_frobenius_cen} yields that $\breve{T}_{n,1}$ and $\mathbf{L}_{K,K}$ are asymptotically independent.
\end{proof}
\subsection{Proof of Lemma \ref{lem_quad_form_frobenius_cen}}\label{sec_proof_asymptotical_independence_ind_quad_cen}
   Throughout the rest of this section, for $i=1,2$ and $1\leq k \leq n$, we will make use of the abbreviations
    \begin{align}\label{eq_definition_list}
    \bfZ_{n,k}^{(i)} &= \bfZ_{n}^{(i)}-\bfz_{k}^{(i)}\textbf{e}_k^{\top}, \quad \bfR_{n,k}^{(i)}(\theta_{j,n}^{(i)}) = \textbf{I}_p-\frac{1}{\theta_{j,n}^{(i)}}\frac{1}{n}\bfSigma_{P,n}^{(i)}\bfZ_{n,k}^{(i)}\bfZ_{n,k}^{(i)^\top},
    \\
    \bfP_{n,k}^{(i)}(\theta_{j,n}^{(i)}) &=\lb \bfR_{n,k}^{(i)}(\theta_{j,n}^{(i)})\rb^{-1}\mathbf{u}_{k,n}^{(i)}\mathbf{u}_{k,n}^{(i)^\top}\lb \lcb \bfR_{n,k}^{(i)}(\theta_{j,n}^{(i)})\rcb^{\top}\rb^{-1},\nonumber
    \\
    \delta_{n,k}^{(i)}(\theta_{j,n}^{(i)}) &= \bfz_k^{(i)^\top}\bfP_{n,k}^{(i)}(\theta_{j,n}^{(i)})\bfz_k^{(i)}-\frac{1}{n}\tr\lb \bfP_{n,k}^{(i)}(\theta_{j,n}^{(i)})\rb, \nonumber
    \\
    \bar{\beta}_{n,k}^{(i)}(\theta_{j,n}^{(i)})&=\lcb 1-\frac{1}{n}\tr\lb \frac{1}{\theta_{j,n}^{(i)}}\bfSigma_{P,n}^{(i)}\lb \lcb \bfR_{n,k}^{(i)}(\theta_{j,n}^{(i)}) \rcb^\top \rb^{-1} \rb\rcb^{-1},\nonumber
    \\
    b_n^{(i)}(\theta_{j, n}^{(i)}) &= \lcb 1-\frac{1}{n}\E\lsb\tr\lb \frac{1}{\theta_{j,n}^{(i)}}\bfSigma_{P,n}^{(i)}\lb \lcb \bfR_{n,k}^{(i)}(\theta_{j,n}^{(i)}) \rcb^\top \rb^{-1} \rb \rsb\rcb^{-1},\nonumber
    \\
    \bfF_{n,k}^{(1)}&= \sum_{j=1}^{k}(\tilde{\bfx}_{j}^{(1)}\tilde{\bfx}_{j}^{(1)^\top}-\bfSigma_{n}^{(1)}), \quad \bfF_{n,k}^{(2)}=\sum_{j=1}^{k}(\tilde{\bfx}_j^{(2)}\tilde{\bfx}_j^{(2)^\top}-\bfSigma_{n}^{(2)}), \quad\bfG_{n,k}^{(i)} = \bfSigma_{n}^{(i)^{1/2}}\bfF_{n,k}^{(i)}\bfSigma_{n}^{(i)^{1/2}},\nonumber
    \\
    \bfG_{n,k}^{(i,j)} &= \bfSigma_{n}^{(j)^{1/2}}\bfF_{n,k}^{(i)}\bfSigma_{n}^{(j)^{1/2}}, \quad \bfH_{n}^{(i,j)} = \bfSigma_{n}^{(i)^{1/2}}\bfSigma_{n}^{(j)}\bfSigma_{n}^{(i)^{1/2}}\nonumber.
    \end{align}
    We note that the quantities such as $\bar{\beta}_{n, k}^{(i)}(\theta_{j,n}^{(i)})$ and $b_{n}^{(i)}(\theta_{j,n}^{(i)})$ are not always bounded and the matrix $\textbf{R}_{n, k}^{(i)}(\theta_{j,n}^{(i)})$ is not always invertible. Thus, for $i=1, 2$ and $1\leq k\leq n$ we introduce the events
    \begin{align*}
	\CB_{1, n, k}^{( i)}(\theta_{j,n}^{(i)})&=\lcb \lnorm \frac{1}{\theta_{j,n}^{(i)}}\frac{1}{n}\mathbf{\Sigma}_{P,n}^{(i)}\textbf{Z}_{n, k}^{(i)}\textbf{Z}_{n, k}^{(i)^\top}\rnorm_{2} \leq 1-\varepsilon\rcb,
	\allowdisplaybreaks
	\\
	\CB_{3, n, k}^{(i)}(\theta_{j,n}^{(i)})&=\lcb \bigg\vert\frac{1}{n}\tr\lb \frac{1}{\theta_{j,n}^{(i)}}{\mathbf{\Sigma}}_{P,n}^{(i)}\lb\textbf{R}_{n, k}^{(i)^\top}(\theta_{j,n}^{(i)})\rb^{-1}\rb-\lb 1+\frac{1}{\theta_{j,n}^{(i)}s_{\underline{F}_{y,H^{(i)}}} (\theta_{j,n}^{(i)})}\rb\bigg\vert<\varepsilon\rcb,
\end{align*}
where $\lnorm . \rnorm_2$ denotes the spectral norm of a matrix and $\varepsilon$ is a small positive constant and the Stieltjes transform ${s}_{\underline{F}_{y,H^{(i)}}}(z)$ denotes the Stieltjes transform of the limiting spectral distribution $\underline{F}_{y,H^{(i)}}$ of $\underline{\bfS}_n^{(i)}:= \bfZ_n^{(i)^T} \bfSigma_n^{(i)} \bfZ_n^{(i)}$. The Stieltjes transform ${s}_{\underline{F}_{y, H^{\lb i\rb}}}$ is the unique solution to the equation
\begin{align*}
    z=-\frac{1}{{s}_{\underline{F}_{y, H^{\lb i\rb}}}\lb z\rb}+y\int \frac{t}{1+t{s}_{\underline{F}_{y, H^{\lb i\rb}}}\lb z\rb}\, dH^{(i)}\lb t\rb, \quad z\in \mathbb{C}^{+}=\lcb z\mid \Im (z) >0\rcb.
\end{align*}
In addition, we define
\begin{align*}
	\CB_{n}^{(i)}:= \bigcap\limits_{k=1}^{n}\bigcap\limits_{j=1}^{K}\lb \CB_{1, n, k}^{(i)}(\theta_{j,n}^{(i)})\cap \CB_{3, n, k}^{(i)}(\theta_{j,n}^{(i)})\rb,\quad  i=1, 2.
\end{align*}

\begin{lemma}\label{lem_event_high_probability}
    Suppose the assumptions \ref{ass_p_n} - \ref{ass_population_lsd}, \ref{ass_supercritical_ev_generalized} are satisfied. For $i=1, 2$, the event $\CB_n^{(i)}$ holds with high probability, that is 
	\begin{align*}
		\PR \lb\CB_n^{(i)}\rb=1-O\lb n^{-\ell}\rb
	\end{align*}
	for any $\ell >0$.
\end{lemma}
The proof of Lemma \ref{lem_event_high_probability} is similar to proof of Lemma 4.2 in \cite{zhang2022asymptotic}, which can be found in the supplement to their paper, and therefore is omitted. 

Next, we introduce the random vectors 
\begin{align*}
	\bfy_k=\begin{cases}
		\bfz_k^{(1)} & \text{if } k\leq n
		\\
		\bfz_{k-n}^{(2)} & \text{if } k\geq n+1
	\end{cases}\,.
\end{align*}
In addition, $\CF_k=\sigma\lb \bfy_1, \ldots, \bfy_k\rb$ denotes the $\sigma$-algebra generated by $\bfy_1, \ldots, \bfy_k$ and $\E_k$ denotes the conditional expectation given $\CF_k$, where $\E_0=\E$. Finally, for $1\leq k\leq 2n$ we define
\begin{align}\label{DefintionListMnk}
	M_{n, k} = \lb \E_k-\E_{k-1}\rb \lsb \sigma_{n,1} \tilde{T}_{n,1}\rsb.
\end{align}
With these preparations, we are now able to prove Lemma \ref{lem_quad_form_frobenius_cen}.
\begin{proof}[Proof of Lemma \ref{lem_quad_form_frobenius_cen}]
   Due the the CLTs from Theorem $1$ in \cite{li_and_chen_2012} and from Theorem $2.1$ in \cite{zhang2022asymptotic}, it remains to proof the asymptotically independence of $\tilde{T}_{n,1}$ and $\widetilde{\mathbf{L}}_{K,K}$.
   \newline
   For arbitrary $c_0, c_1^{(i)}, \ldots, c_{K}^{(i)}\in \R$ $(i=1,2)$ our goal is to show
   \begin{align*}
       c_0\tilde{T}_{n,1}+\sum_{i=1}^{2}\sum_{j=1}^{K} c_j^{(i)} \tilde{Q}_{j,n}^{(i)} \cond \mathcal{N}(0, \sigma^2),
   \end{align*}
   where 
   \begin{align}\label{eq_definition_sigma_CWD}
       \sigma^2= c_{0}^2+ \sum_{i=1}^{2} \sum_{j,k=1}^{K} c_{j}^{(i)}c_{k}^{(i)}(\bfSigma_L^{(i)})_{j,k}.
   \end{align}
    Since the intersection of events $\CB_{n}^{\lb 1\rb}$ and $\CB_{n}^{\lb 2\rb}$ hold with high probability by Lemma \ref{lem_event_high_probability}, it suffices to prove that
   \begin{align*}
       c_{0}\tilde{T}_{n,1}+\sum_{i=1}^{2}\sum_{j=1}^{K} c_j^{(i)} \tilde{Q}_{j,n}^{(i)}\mathds{1}(\CB_n^{(i)}) \cond \mathcal{N}(0, \sigma^2).
   \end{align*}
   For $i=1, 2$ and $j=1, \ldots, K$, the convergence in equation $(4.42)$ of \cite{zhang2022asymptotic} yields
    \begin{align*}
        \sqrt{n}\lb \E\lsb \mathbf{u}_{j,n}^{(i)^\top} \bfZ_n^{(i)} \lcb \tilde{\bfA}_n^{(i)}(\theta_{j,n}^{(i)}) \rcb^{-1}\bfZ_n^{(i)^\top}\mathbf{u}_{j,n}^{(i)}\rsb-\frac{\theta_{j,n}^{(i)}}{\alpha_{j}^{(i)}}\rb\mathds{1}\lb \CB_{n}^{\lb i\rb}\rb\rightarrow 0, \quad n \to \infty ~.
    \end{align*}
    As a consequence, it suffices to analyse the asymptotic distribution of 
    \begin{align}\label{CramerWoldDeviceWithExpectedValueInsteadOfASExpectedValue}
    \begin{split}
        c_0\tilde{T}_{n,1} & +\sum_{i=1}^{2}\sum_{j=1}^{K}c_j^{(i)}\sqrt{n}\Big (  \mathbf{u}_{j,n}^{(i)^\top} \bfZ_n^{(i)}\lcb \tilde{\bfA}_{n}^{(i)}(\theta_{j,n}^{(i)})\rcb^{-1}\bfZ_n^{(i)^\top}\mathbf{u}_{j,n}^{(i)} \\
        & ~~~~~~~~~~~~~-\E\lsb \mathbf{u}_{j,n}^{(i)^\top} \bfZ_n^{(i)}\lcb \tilde{\bfA}_{n}^{(i)}(\theta_{j,n}^{(i)})\rcb^{-1}\bfZ_n^{(i)^\top}\mathbf{u}_{j,n}^{(i)}\rsb\Big ) \mathds{1}(\CB_{n}^{(i)}).
            \end{split}
    \end{align}
    Using the truncation argument in Section $\text{S}.7$ of the supplement to \cite{zhang2022asymptotic}, it is sufficient to prove the assertion where the entries of the matrix $\textbf{Z}_n^{\lb i\rb}$ in the supercritical eigenvalues terms are truncated at $t_n^{\lb i\rb}n^{1/4}$ $\lb i=1,2\rb$, where the sequence $\lcb t_{n}^{\lb i\rb}\rcb_{n\geq 1}$ satisfies 
	\begin{align*}
		t_{n}^{\lb i\rb} \to 0 \text{ and } \E\lsb t_{n}^{(i)^4}\lv z_{1,1}^{(i)}\rv^4\mathds{1}\lcb \lv z_{1, 1}^{\lb i\rb}\rv \rcb > t_{n}^{\lb i\rb}n^{1/4}\rsb \to 0, \quad n\to \infty,
	\end{align*}
	where $z_{1,1}^{\lb i\rb}$ is a standardized random variable with existing eighth moment defined in assumption \ref{ass_random_vectors}.
    Hence, it is sufficient to consider
    \begin{align}\label{TruncatedCramerWoldDeviceWithExpectedValueInsteadOfASExpectedValue}
        \begin{split}
        c_0\tilde{T}_{n,1} & +\sum_{j=1}^{K}c_j^{(1)}\sqrt{n}\Big (  \mathbf{u}_{j,n}^{(1)^\top} \dot{\bfZ}_n^{(1)}\lcb \dot{\tilde{\bfA}}_{n}^{(1)}(\theta_{j,n}^{(1)})\rcb^{-1}\dot{\bfZ}_n^{(1)^\top}\mathbf{u}_{j,n}^{(i)} \\
        & ~~~~~~~~~~~~~~~~~~~
        -\E\lsb \mathbf{u}_{j,n}^{(1)^\top} \dot{\bfZ}_n^{(1)}\lcb \dot{\tilde{\bfA}}_{n}^{(1)}(\theta_{j,n}^{(1)})\rcb^{-1}\dot{\bfZ}_n^{(1)^\top}\mathbf{u}_{j,n}^{(1)}\Big ) \rb\mathds{1}(\CB_{n}^{(1)})
        \\
        & +\sum_{j=1}^{K}c_j^{(2)}\sqrt{n}\Big ( \mathbf{u}_{j,n}^{(2)^\top} \ddot{\bfZ}_n^{(2)}\lcb \ddot{\tilde{\bfA}}_{n}^{(2)}(\theta_{j,n}^{(2)})\rcb^{-1}\ddot{\bfZ}_n^{(2)^\top}\mathbf{u}_{j,n}^{(2)}\\
        & ~~~~~~~~~~~~~~~~~~~
        -\E\lsb \mathbf{u}_{j,n}^{(2)^\top} \ddot{\bfZ}_n^{(2)}\lcb \ddot{\tilde{\bfA}}_{n}^{(2)}(\theta_{j,n}^{(2)})\rcb^{-1}\ddot{\bfZ}_n^{(2)^\top}\mathbf{u}_{j,n}^{(2)}\Big ) \rb\mathds{1}(\CB_{n}^{(2)})
            \end{split}
    \end{align}
    instead of the random variable \eqref{CramerWoldDeviceWithExpectedValueInsteadOfASExpectedValue}, where we use the notation $"."$ to indicate terms involved in the supercritical eigenvalue expression that any entries of $\textbf{Z}_n^{\lb 1\rb}$ are truncated at $t_n^{\lb 1\rb}n^{1/4}$ and while we use the notation $".."$ to indicate terms in the supercritical eigenvalue expression that any entries of $\textbf{Z}_n^{\lb 2\rb}$ are truncated at $t_n^{\lb 2\rb}n^{1/4}$.
    \newline
    The strategy of this proof is to rewrite each term in \eqref{TruncatedCramerWoldDeviceWithExpectedValueInsteadOfASExpectedValue} as a martingale difference sequence and to use the martingale central limit theorem.
	\newline
	For the first term in \eqref{TruncatedCramerWoldDeviceWithExpectedValueInsteadOfASExpectedValue}, we obtain by the identity $\lb 4.18\rb$ from \cite{zhang2022asymptotic}
    	\begin{align}\label{MDSSpikedEigenvaluesFirstSample}
		&c_j^{(1)}\sqrt{n}\lb \mathbf{u}_{j,n}^{\lb 1\rb^\top}\dot{\textbf{Z}}_n^{\lb 1\rb}\lcb \dot{\tilde{\textbf{A}}}_n^{\lb 1\rb}\lb\theta_{j, n}^{\lb 1\rb}\rb\rcb^{-1}\dot{\textbf{Z}}_n^{\lb 1\rb^\top}\mathbf{u}_{j,n}^{\lb 1\rb}-\E\lsb \mathbf{u}_{j,n}^{\lb 1\rb^\top}\dot{\textbf{Z}}_n^{\lb 1\rb}\lcb \dot{\tilde{\textbf{A}}}_n^{\lb 1\rb}\lb \theta_{j,n}^{\lb 1\rb}\rb\rcb^{-1}\dot{\textbf{Z}}_n^{\lb 1\rb^\top}\mathbf{u}_{j,n}^{\lb 1\rb}\rsb\rb\mathds{1}\lb \CB_{n}^{\lb 1\rb}\rb
		\\
		=&\sum_{k=1}^{n}\E_k\lb c_j^{(1)}\sqrt{n}\dot{\bar{\beta}}_{n,k}^{\lb 1\rb}(\theta_{j,n}^{(1)})\dot{\delta}_{n, k}^{\lb 1\rb}(\theta_{j,n}^{(1)})\mathds{1}\lb \CB_{1, n, k}^{\lb 1\rb}(\theta_{j,n}^{(1)})\cap \CB_{3, n, k}^{\lb 1\rb}(\theta_{j,n}^{(1)})\rb\rsb + o_{\PR}\lb 1\rb\nonumber
		\\
		=&\sum_{k=1}^{n}\E_{k}\lsb c_j^{(1)}\sqrt{n} \dot{b}_{n}^{\lb 1\rb}(\theta_{j,n}^{(1)}) \dot{\delta}_{n,k}^{\lb 1\rb}(\theta_{j,n}^{(1)})\mathds{1}\lb \CB_{1, n, k}^{\lb 1\rb}(\theta_{j,n}^{(1)})\cap \CB_{3, n, k}^{\lb 1\rb}(\theta_{j,n}^{(1)})\rb\rsb+ o_{\PR}\lb 1\rb\nonumber
		\\
		=&\sum_{k=1}^{n}\lb\E_{k}-\E_{k-1}\rb\lsb c_j^{(1)}\sqrt{n} \dot{b}_{n}^{\lb 1\rb}(\theta_{j,n}^{(1)}) \dot{\delta}_{n,k}^{\lb 1\rb}(\theta_{j,n}^{(1)})\mathds{1}\lb \CB_{1, n, k}^{\lb 1\rb}(\theta_{j,n}^{(1)})\cap \CB_{3, n, k}^{\lb 1\rb}(\theta_{j,n}^{(1)})\rb\rsb+ o_{\PR}\lb 1\rb\nonumber,
	\end{align}
    where the last line holds because 
	\begin{align}\label{ConditionalExpectationK1QuadraticFormVectorIndependentOfIndicator}
		\E_{k-1}\lsb \dot{b}_n^{\lb 1\rb}(\theta_{j,n}^{(1)})\dot{\delta}_{n,k}^{\lb 1\rb}(\theta_{j,n}^{(1)})\mathds{1}\lb \CB_{1, n, k}^{\lb 1\rb}\cap \CB_{3, n, k}^{\lb 1\rb}\rb\rsb =0
	\end{align}
	for $1\leq k\leq n$.
    Similar, we get for the supercritical eigenvalues part of the second sample 
	\begin{align}\label{MDSSpikedEigenvaluesSecondSample}
		&c_j^{(2)}\sqrt{n}\lb \mathbf{u}_{j,n}^{\lb 2\rb^\top}\ddot{\textbf{Z}}_n^{\lb 2\rb}\lcb\ddot{\tilde{\textbf{A}}}_n^{\lb 2\rb}\lb \theta_{j,n}^{\lb 2\rb}\rb\rcb^{-1}\ddot{\textbf{Z}}_n^{\lb 2\rb^\top}\mathbf{u}_{j,n}^{\lb 2\rb}-\E\lsb \mathbf{u}_{j,n}^{\lb 2\rb^\top}\ddot{\textbf{Z}}_n^{\lb 2\rb}\lcb\ddot{\tilde{\textbf{A}}}_n^{\lb 2\rb}\lb \theta_{j,n}^{\lb 2\rb}\rb\rcb^{-1}\ddot{\textbf{Z}}_n^{\lb 2\rb^\top}\mathbf{u}_{j,n}^{\lb 2\rb}\rsb\rb\mathds{1}\lb \CB_{n}^{\lb 2\rb}\rb 
		\\
		=&\sum_{k=1}^{n}\E_{k+n}\lsb c_j^{\lb 2\rb}\sqrt{n}\ddot{\bar{\beta}}_{n,k}^{\lb 2\rb}(\theta_{j,n}^{\lb 2\rb})\ddot{\delta}_{n,k}^{\lb 2\rb}(\theta_{j,n}^{\lb 2\rb})\mathds{1}\lb \CB_{1, n, k}^{\lb 2\rb}(\theta_{j,n}^{\lb 2\rb})\cap\CB_{3, n, k}^{\lb 2\rb}(\theta_{j,n}^{\lb 2\rb})\rb\rsb + o_{\PR}\lb 1\rb\nonumber
		\\
		=&\sum_{k=1}^{n}\E_{k+n}\lsb c_j^{\lb 2\rb}\sqrt{n} \ddot{b}_{n}^{\lb 2\rb}(\theta_{j,n}^{\lb 2\rb}) \ddot{\delta}_{n,k}^{\lb 2\rb}(\theta_{j,n}^{\lb 2\rb})\mathds{1}\lb \CB_{1, n, k}^{\lb 2\rb}(\theta_{j,n}^{\lb 2\rb})\cap\CB_{3, n, k}^{\lb 2\rb}(\theta_{j,n}^{\lb 2\rb})\rb\rsb+ o_{\PR}\lb 1\rb\nonumber
		\\
		=&\sum_{k=1}^{n}\lb\E_{k+n}-\E_{k-1+n}\rb\lsb c_j^{\lb 2 \rb}\sqrt{n} \ddot{b}_{n}^{\lb 2\rb}(\theta_{j,n}^{\lb 2\rb}) \ddot{\delta}_{n,k}^{\lb 2\rb}(\theta_{j,n}^{\lb 2\rb})\mathds{1}\lb \CB_{1, n, k}^{\lb 2\rb}(\theta_{j,n}^{\lb 2\rb})\cap\CB_{3, n, k}^{\lb 2\rb}(\theta_{j,n}^{\lb 2\rb})\rb\rsb+ o_{\PR}\lb 1\rb,\nonumber
	\end{align}
	where the last equality holds because
	\begin{align}\label{ConditionalExpectationK1QuadraticFormVectorIndependentOfIndicator2}
		\E_{k-1+n}\lsb \ddot{b}_n^{\lb 2\rb} (\theta_{j,n}^{\lb 2\rb})\ddot{\delta}_{n,k}^{\lb 2\rb}(\theta_{j,n}^{\lb 2\rb})\mathds{1}\lb \CB_{1, n, k}^{\lb 2\rb}(\theta_{j,n}^{\lb 2\rb})\cap\CB_{3, n, k}^{\lb 2\rb}(\theta_{j,n}^{\lb 2\rb})\rb\rsb=0,
	\end{align}
	for all $1\leq k\leq n$. The derivation of \eqref{ConditionalExpectationK1QuadraticFormVectorIndependentOfIndicator} and \eqref{ConditionalExpectationK1QuadraticFormVectorIndependentOfIndicator2} is straightforward. Therefore, we skip the details here for the sake of brevity.  For the rest of the discussion, we omit the indicators $\mathds{1}\lb \CB_{1, n, k}^{\lb i\rb}(\theta_{j,n}^{\lb i\rb})\cap \CB_{3, n, 1}^{\lb i\rb}(\theta_{j,n}^{\lb i\rb})\rb$ $\lb i=1, 2, j=1, \ldots, K\rb$ for the sake of a simple notation. But, we emphasize that a suitable indicator function of events is needed, whenever handling the inverses of random matrices. 
    Using the representation by a telescope sum yields
    \begin{align}\label{mds_frobenius_representation}
        \tilde{T}_{n,1}=\frac{1}{\sigma_{n,1}}\sum_{k=1}^{2n} M_{n,k},
    \end{align}
    where $M_{n,k}$ is defined in \eqref{DefintionListMnk} and can be represented as
    	\begin{align}\label{mds_frob_below_n}
		M_{n, k}&=\frac{2}{n\lb n-1\rb}\lcb \bfz_k^{\lb 1\rb^\top}\textbf{G}_{n, k-1}^{\lb 1\rb}\bfz_k^{\lb 1\rb}-\tr\lb \textbf{G}_{n, k-1}^{\lb 1\rb}\rb\rcb
		\\
		&\phantom{{}={}}+\frac{2}{n}\lcb \bfz_k^{\lb 1\rb^\top}\lb \bfSigma_{n}^{(1)^2}-\textbf{H}_n^{\lb 1, 2\rb}\rb \bfz_k^{\lb 1\rb}-\tr\lb \bfSigma_{n}^{(1)^2}-\textbf{H}_n^{\lb 1, 2\rb}\rb\rcb\nonumber
	\end{align}
	for $1\leq k\leq n$, while for $n+1\leq k\leq 2n$ the representation
	\allowdisplaybreaks
	\begin{align}\label{mds_frob_above_n}
		M_{n, k}&=\frac{2}{n\lb n-1\rb}\lcb \bfz_{k-n}^{\lb 2\rb^\top}\textbf{G}_{n, k-n-1}^{\lb 2\rb} \bfz_{k-n}^{\lb 2\rb}-\tr\lb \textbf{G}_{n, k-n-1}^{\lb 2\rb}\rb\rcb
		\\
		&\phantom{{}={}}+\frac{2}{n}\lcb \bfz_{k-n}^{\lb 2\rb^\top}\lb \bfSigma_{n}^{(2)^2}-\textbf{H}_n^{\lb 2, 1\rb}\rb \bfz_{k-n}^{\lb 2\rb}-\tr\lb \bfSigma_{n}^{(2)^2}-\textbf{H}_{n}^{\lb 2, 1\rb}\rb\rcb\nonumber
		\\
		&\phantom{{}={}}-\frac{2}{n^2}\lcb \bfz_{k-n}^{\lb 2\rb^\top}\textbf{G}_{n, n}^{\lb 1, 2\rb}\bfz_{k-n}^{\lb 2\rb}-\tr\lb \textbf{G}_{n, n}^{\lb 1, 2\rb}\rb\rcb\nonumber
	\end{align}
	holds. We emphasize that the matrices $\bfF_{n,k}^{(i)}$ defined in \eqref{eq_definition_list} are also defined for $k=0$. In this case, $\bfF_{n,0}^{(i)}=\mathbf{0}$. Thus, for $i=1, 2$ the matrices $\bfG_{n,0}^{(i)}=0$ as well.
    A derivation of \eqref{mds_frob_below_n} and \eqref{mds_frob_above_n} is tedious but straightforward. These formulas were already used in \cite{li_and_chen_2012}, and therefore we skip more details here for the sake of brevity.
    
    Summarizing the calculations \eqref{MDSSpikedEigenvaluesFirstSample}, \eqref{MDSSpikedEigenvaluesSecondSample} and \eqref{mds_frobenius_representation} implies that the random variable in \eqref{TruncatedCramerWoldDeviceWithExpectedValueInsteadOfASExpectedValue} can be represented as
     \begin{align*}
        \frac{c_0}{\sigma_{n,1}}\sum_{k=1}^{2n}M_{n,k} & +\sum_{j=1}^{K}\sum_{k=1}^{n}\E_{k}\lsb c_j^{(1)}\sqrt{n}\dot{b}_{n}^{(1)}(\theta_{j,n}^{(1)})\dot{\delta}_{n,k}^{(1)}(\theta_{j,n}^{(1)})\rsb \\
        & +\sum_{j=1}^{K}\sum_{k=1}^{n}\E_{k+n}\lsb c_j^{(2)}\sqrt{n}\ddot{b}_{n}^{(2)}(\theta_{j,n}^{(2)})\ddot{\delta}_{n,k}^{(2)}(\theta_{j,n}^{(2)})\rsb+ o_{\PR}(1)
        \\
        =&\frac{c_0}{\sigma_{n,1}}\sum_{k=1}^{2n} M_{n,k} + \sum_{k=1}^{2n} W_{n,k} + o_{\PR}(1),
    \end{align*}
    where 
    \begin{align*}
		W_{n, k}=\begin{cases}
		\sum\limits_{j=1}^{K}c_j^{(1)}\E_k\lsb \sqrt{n}\dot{b}_n^{\lb 1\rb}\lb \theta_{j,n}^{(1)}\rb\dot{\delta}_{n, k}^{\lb 1\rb}\lb \theta_{j,n}^{(1)}\rb\rsb  & \text{if } k\leq n
		\\
		\sum\limits_{j=1}^{K}c_j^{(2)}\E_k\lsb \sqrt{n}\ddot{b}_n^{\lb 2\rb}\lb \theta_{j,n}^{(2)}\rb\ddot{\delta}_{n, k-n}^{\lb 2\rb}\lb \theta_{j,n}^{(2)}\rb\rsb &\text{if } k\geq n+1
		\end{cases}\,.
	\end{align*}
    Moreover, by these calculations it follows that 
	\begin{align*}
		\frac{c_0}{\sigma_{n, 1}}M_{n, k}+ W_{n,k}
	\end{align*}
	is a martingale difference sequence. Therefore, by the CLT for martingales (Theorem \ref{thm_CLT_for_MDS}) we need to show 
    \begin{enumerate}[label=(\arabic*)]
        \item \begin{align*}
        \sum_{k=1}^{2n}\E_{k-1}\lsb\lcb \frac{c_0}{\sigma_{n,1}}M_{n,k}+ W_{n,k}\rcb^2\rsb= \sum_{k=1}^{2n}\E_{k-1}\lsb \frac{c_0^2}{\sigma_{n,1}^2}M_{n, k}^2+\frac{2c_0}{\sigma_{n, 1}}M_{n, k}W_{n, k} +W_{n, k}^{2}\rsb \conp \sigma^2,
        \end{align*}
        where $\sigma^2$ is defined in \eqref{eq_definition_sigma_CWD},
        \item for each $\varepsilon>0$ that
        \begin{align*}
        \sum_{k=1}^{2n}\E\lsb \lb \frac{c_0}{\sigma_{n,1}}M_{n,k}+ W_{n, k}\rb^2 \mathds{1}\lb \lv \frac{c_0}{\sigma_{n, 1}}M_{n, k}+W_{n, k}\rv>\varepsilon\rb\rsb\to 0, \quad n \to \infty ~.
        \end{align*}
    \end{enumerate}
    We note, \cite{li_and_chen_2012} and \cite{zhang2022asymptotic} use martingales to prove the CLT for the Frobenius norm and for the random quadratic forms, respectively. Thus, from their proofs we have
    \begin{align*}
        \sum_{k=1}^{2n}\E_{k-1}\lsb c_0^2\frac{M_{n,k}^{2}}{\sigma_{n,1}^2}\rsb \conp c_0^2 ,\quad \sum_{k=1}^{2n}\E_{k-1}\lsb W_{n,k}^2\rsb \conp \sum_{i=1}^{2} \sum_{j,k=1}^{K} c_{j}^{(i)}c_{k}^{(i)}(\bfSigma_L^{(i)})_{j,k}
    \end{align*}
    and
    \begin{align*}
    \sum_{k=1}^{2n}\E\lsb \lcb\frac{M_{n,k}}{\sigma_{n,1}}\rcb^4\rsb\to 0, \quad\sum_{k=1}^{2n}\E\lsb W_{n, k}^{4}\rsb \to 0.
    \end{align*}
    Because
	\begin{align*}
		&\sum_{k=1}^{2n}\E\lsb \lb \frac{c_0}{\sigma_{n,1}}M_{n, k}+W_{n, k}\rb^2 \mathds{1}\lb \lv \frac{c_0}{\sigma_{n,1}}M_{n, k}+W_{n, k}\rv>\varepsilon\rb\rsb	\leq\frac{C}{\varepsilon^2}\sum_{k=1}^{2n}\E\lsb \lcb\frac{M_{n,k}}{\sigma_{n,1}}\rcb^4+W_{n, k}^{4}\rsb \to 0,
	\end{align*}
    it remains to show that
    \begin{align}\label{cross_product_needs_op1}
		\sum_{k=1}^{2n}\E_{k-1}\lsb \frac{c_0}{\sigma_{n,1}}M_{n, k}W_{n,k}\rsb = o_{\PR}\lb 1\rb.
	\end{align}
    By the representation of $M_{n, k}$ in \eqref{mds_frob_below_n} and \eqref{mds_frob_above_n}, we obtain that the left-hand side of \eqref{cross_product_needs_op1} can be represented as
    \begin{align}\label{cross_product_representation_cross_terms}
    \begin{split}
 & \sum_{k=1}^{2n}\E_{k-1}\lsb \frac{c_0}{\sigma_{n,1}}M_{n,k}W_{n,k}\rsb \\
 & ~~~= c_0\sum_{j=1}^{K}\lsb c_j^{(1)}\lb\CC_{1,n}(\theta_{j,n}^{(1)})+ \CC_{2,n}(\theta_{j,n}^{(1)})\rb+c_j^{(2)}\lb\CC_{3,n}(\theta_{j,n}^{(2)})+\CC_{4,n}(\theta_{j,n}^{(2)})+\CC_{5,n}(\theta_{j,n}^{(2)})\rb\rsb,
 \end{split}
    \end{align}
    where
    \begin{align*}
    \CC_{1, n}(\theta_{j,n}^{(1)}) &= \sum_{k=1}^{n}\E_{k-1}\lsb \E_{k}\lsb \sqrt{n}\dot{b}_{n}^{\lb 1\rb}(\theta_{j,n}^{(1)})\dot{\delta}_{n, k}^{\lb 1\rb}(\theta_{j,n}^{(1)})\rsb\frac{1}{\sigma_{n,1}}\frac{2}{n\lb n-1\rb}\lcb \bfz_k^{\lb 1\rb^\top}\textbf{G}_{n, k-1}^{\lb 1\rb}\bfz_k^{\lb 1\rb}-\tr\lb \textbf{G}_{n, k-1}^{\lb 1\rb}\rb\rcb\rsb,
    \\
    \CC_{2, n}(\theta_{j,n}^{(1)}) &=\sum_{k=1}^{n}\E_{k-1}\Bigg[ \E_{k}\lsb \sqrt{n}\dot{b}_{n}^{\lb 1\rb}(\theta_{j,n}^{(1)})\dot{\delta}_{n, k}^{\lb 1\rb}(\theta_{j,n}^{(1)})\rsb\frac{1}{\sigma_{n,1}}\frac{2}{n}\Big\{\bfz_k^{\lb 1\rb^\top}\lb \bfSigma_{n}^{(1)^2}-\textbf{H}_n^{\lb 1, 2\rb}\rb \bfz_k^{\lb 1\rb}\nonumber
	\\
	&\phantom{=\sum_{k=1}^{n}\E_{k-1}\Bigg[ }-\tr\lb \bfSigma_{n}^{(1)^2}-\textbf{H}_n^{\lb 1, 2\rb}\rb\Big\}\Bigg],\nonumber
    \\
	\CC_{3, n}(\theta_{j,n}^{(2
    )}) &=\sum_{k=n+1}^{2n}\E_{k-1}\bigg[\E_{k}\lsb\sqrt{n}\ddot{b}_n^{\lb 2\rb}(\theta_{j,n}^{(2)})\ddot{\delta}_{n,k-n}^{\lb 2\rb}(\theta_{j,n}^{(2)})\rsb\frac{1}{\sigma_{n,1}}\frac{2}{n\lb n-1\rb}\Big\{ \bfz_{k-n}^{\lb 2\rb^\top}\textbf{G}_{n, k-n-1}^{\lb 2\rb} \bfz_{k-n}^{\lb 2\rb}\nonumber
	\\
	&\phantom{=\sum_{k=n+1}^{2n}\E_{k-1}\bigg[}-\tr\lb \textbf{G}_{n, k-n-1}^{\lb 2\rb}\rb\Big\}\bigg],\nonumber
    \\
    \CC_{4, n}(\theta_{j,n}^{(2)})&=\sum_{k=n+1}^{2n}\E_{k-1}\Bigg[\E_{k}\lsb\sqrt{n}\ddot{b}_n^{\lb 2\rb}(\theta_{j,n}^{(2)})\ddot{\delta}_{n,k-n}^{\lb 2\rb}(\theta_{j,n}^{(2)})\rsb \frac{1}{\sigma_{n,1}}\frac{2}{n}\Big\{\bfz_{k-n}^{\lb 2\rb^\top}\lb \bfSigma_{n}^{(2)^2}-\textbf{H}_n^{\lb 2, 1\rb}\rb \bfz_{k-n}^{\lb 2\rb}\nonumber
    \\
	&\phantom{=\sum_{n=k+1}^{2n}\E_{k-1}\Bigg[}-\tr\lb \bfSigma_{n}^{(2)^2}-\textbf{H}_{n}^{\lb 2, 1\rb}\rb\Big\}\Bigg],\nonumber
    \\
	\CC_{5, n}(\theta_{j,n}^{(2)})&=\sum_{k=n+1}^{2n}\E_{k-1}\lsb\E_{k}\lsb\sqrt{n}\ddot{b}_n^{\lb 2\rb}(\theta_{j,n}^{(2)})\ddot{\delta}_{n,k-n}^{\lb 2\rb}(\theta_{j,n}^{(2)})\rsb \frac{1}{\sigma_{n,1}}\frac{2}{n^2}\lcb \bfz_{k-n}^{\lb 2\rb^\top}\textbf{G}_{n, n}^{\lb 1, 2\rb}\bfz_{k-n}^{\lb 2\rb}-\tr\lb \textbf{G}_{n, n}^{\lb 1, 2\rb}\rb\rcb\rsb.\nonumber
    \end{align*}
    Because of the representation of the cross in product in \eqref{cross_product_representation_cross_terms} and Lemma \ref{lem_cross_terms_op1}, which is given in the next subsection, it follows by Slutsky's Theorem that \eqref{cross_product_needs_op1} holds.
\end{proof}
\subsubsection{Auxiliary results for the proof of Lemma \ref{lem_quad_form_frobenius_cen}}\label{subsec_proof_cross_terms}
\begin{lemma}\label{lem_cross_terms_op1}
    Suppose that assumptions \ref{ass_p_n} - \ref{ass_population_lsd}, \ref{ass_supercritical_ev_generalized} and \ref{ass_asympt_var_cov} are satisfied. 
    \newline 
    For $j=1, \ldots, K$, the cross terms 
    \begin{align*}
        \CC_{\ell,n} (\theta_{j,n}^{(1)})&=o_\PR(1), \quad \ell=1,2,
        \\
        \CC_{\ell, n}(\theta_{j,n}^{(2)})&=o_\PR(1), \quad \ell =3,4,5.
    \end{align*}
\end{lemma}
For $j=1,\ldots, K$ the proof that the associated cross terms vanish in probability is analogous. Therefore, we will provide only a proof of this lemma in the case that $j=1$.
\begin{proof}      
    We start with the first cross term $\CC_{1,n}(\theta_{1,n}^{(1)})$ and show that this quantity converges to $0$ in $L^2$. To be precise, we consider
\begin{align}\label{eq_first_cross_term_L2_start}
	\E\lsb \lv \sum_{k=1}^{n}\E_{k-1}\lsb \E_{k}\lsb \sqrt{n}\dot{b}_{n}^{\lb 1\rb}(\theta_{1,n}^{(1)})\dot{\delta}_{n, k}^{\lb 1\rb}(\theta_{1,n}^{(1)})\rsb\frac{1}{\sigma_{n,1}}\frac{2}{n\lb n-1\rb}\lcb \bfz_k^{\lb 1\rb^\top}\textbf{G}_{n, k-1}^{\lb 1\rb}\bfz_k^{\lb 1\rb}-\tr\lb \textbf{G}_{n, k-1}^{\lb 1\rb}\rb\rcb\rsb\rv^2\rsb
\end{align}
and show that the expectation is of order $o(1)$. 
\newline
Recall, we consider the above expectation and the following (conditional) expectations on the event $\CB_{1, n, k}^{\lb 1\rb}(\theta_{1,n}^{(1)})\cap \CB_{3, n, 1}^{\lb 1\rb}(\theta_{1,n}^{(1)})$ such that the expressions $\dot{b}_n^{\lb 1\rb}(\theta_{1,n}^{(1)})$ and $\dot{\delta}_{n, k}^{\lb 1\rb}(\theta_{1,n}^{(1)})$ are well defined. This also ensures that $\dot{b}_n^{\lb 1\rb}(\theta_{1,n}^{(1)}) \leq C < \infty. $

Moreover, by Lemma \ref{lem_lower_bounded_variances} it follows that
\begin{align*}
	\sup_{n\in\N}\sigma_{n,1}^{-1} \leq C <\infty.
\end{align*}
Thus, it is sufficient to prove that
\begin{align}\label{eq_first_cross_term_removing_terms}
	\E\lsb \lv \sum_{k=1}^{n}\E_{k-1}\lsb \E_{k}\lsb \sqrt{n}\dot{\delta}_{n, k}^{\lb 1\rb}(\theta_{1,n}^{(1)})\rsb\frac{2}{n\lb n-1\rb}\lcb \bfz_k^{\lb 1\rb^\top}\textbf{G}_{n, k-1}^{\lb 1\rb}\bfz_k^{\lb 1\rb}-\tr\lb \textbf{G}_{n, k-1}^{\lb 1\rb}\rb\rcb\rsb\rv^2\rsb = o(1).
\end{align}
 For this purpose, we give an estimate for each summand.  We note that for $k=1$ the summand is equal $0$ as the matrix $\bfG_{n,0}^{(1)}$ is the null matrix. For $2 \leq k \leq n$, there exists a constant $C$ by Lemma \ref{ConditionalExpectedValueOfProductOfNormalizedQuadraticForms}, such that
\begin{align}\label{eq_first_cross_term_summand_estimation_start}
	&\lv \E_{k-1}\lsb \E_{k}\lsb \dot{\delta}_{n, k}^{\lb 1\rb}(\theta_{1,n}^{(1)})\rsb \lcb \bfz_k^{\lb 1\rb^\top}\textbf{G}_{n, k-1}^{\lb 1\rb}\bfz_k^{\lb 1\rb}-\tr\lb \textbf{G}_{n, k-1}^{\lb 1\rb}\rb\rcb\rsb\rv
	\\
	\leq &\frac{C}{n}\lcb \lv\sum_{j=1}^{p}\lb \E_{k}\lsb\dot{\textbf{P}}_{n, k}^{\lb 1\rb}(\theta_{1,n}^{(1)})\rsb\rb_{j, j}\lb\textbf{G}_{n, k-1}^{\lb 1\rb}\rb_{j,j}\rv + \lv \tr\lb \E_{k}\lsb \dot{\textbf{P}}_{n, k}^{\lb 1\rb}(\theta_{1,n}^{(1)})\rsb\textbf{G}_{n, k-1}^{\lb 1\rb}\rb\rv\rcb\nonumber
\end{align}
holds. To give an estimate for the right-hand side of \eqref{eq_first_cross_term_summand_estimation_start}, we derive estimates for both summands.  Due to the definition of $\dot{\bfP}_{n,k}(\theta_{1,n}^{(1)})$ given in \eqref{eq_definition_list}, it follows that this random matrix is nonnegative definite, which implies that $\lb \dot{\textbf{P}}_{n, k}^{\lb 1\rb}(\theta_{1,n}^{(1)})\rb_{j, j} \geq 0$. Moreover, by the definition of the spectral norm we have $\lv (\textbf{G}_{n,k-1}^{(1)})_{j,j}\rv \leq \lnorm\textbf{G}_{n,k-1}^{(1)}\rnorm_2$. Thus, the first term on the right-hand side of \eqref{eq_first_cross_term_summand_estimation_start} can be estimated as follows
\begin{align}\label{eq_first_cross_term_first_summand_intermediate_estimate}
	\lv\sum_{j=1}^{p}\lb \E_{k}\lsb\dot{\textbf{P}}_{n, k}^{\lb 1\rb}(\theta_{1,n}^{(1)})\rsb\rb_{j, j}\lb\textbf{G}_{n, k-1}^{\lb 1\rb}\rb_{j,j}\rv&\leq 
	\lnorm \textbf{G}_{n, k-1}^{\lb 1\rb}\rnorm_2 \lv \sum_{j=1}^{p}\lb \E_{k}\lsb \dot{\textbf{P}}_{n, k}^{\lb 1\rb}(\theta_{1,n}^{(1)})\rsb\rb_{j, j}\rv 
    \\
    &= \lnorm \textbf{G}_{n, k-1}^{\lb 1\rb} \rnorm_2 \tr\lb \E_{k}\lsb \dot{\textbf{P}}_{n, k}^{\lb 1\rb}(\theta_{1,n}^{(1)})\rsb\rb\nonumber.
\end{align}
Next, we derive an estimate for each term in \eqref{eq_first_cross_term_first_summand_intermediate_estimate}. We start with $\tr\lb \E_{k}\lsb \dot{\textbf{P}}_{n, k}^{\lb 1\rb}(\theta_{1,n}^{(1)})\rsb\rb$ and recall that the vector $\mathbf{u}_{1,n}^{\lb 1\rb}$ is a nonrandom unit vector.
Thus, by definition of the matrix $\dot{\textbf{P}}_{n,k}^{\lb 1\rb}(\theta_{1,n}^{(1)})$ given in \eqref{eq_definition_list}, we obtain
\begin{align}\label{eq_trace_conditional_expectation_Pnk_1}
	\tr\lb \E_{k}\lsb \dot{\textbf{P}}_{n, k}^{\lb 1\rb}(\theta_{1,n}^{(1)})\rsb\rb 
	&=\E_k\lsb \mathbf{u}_{1,n}^{\lb 1\rb^\top}\lb\dot{\textbf{R}}_{n,k}^{\lb 1\rb^\top}(\theta_{1,n}^{(1)})\rb^{-1}\lb\dot{\textbf{R}}_{n,k}^{\lb 1\rb}(\theta_{1,n}^{(1)})\rb^{-1}\mathbf{u}_{1,n}^{\lb 1\rb}\rsb\leq \E_k \lsb \lnorm \lb\dot{\textbf{R}}_{n,k}^{\lb 1\rb}(\theta_{1,n}^{(1)})\rb^{-1}\rnorm_2^2\rsb
	\\
	&=\E_{k}\lsb \frac{1}{\lambda_{\text{min}}^2\lb \dot{\textbf{R}}_{n,k}^{\lb 1\rb}(\theta_{1,n}^{(1)})\rb}\rsb\leq C,\nonumber
\end{align}
where we have used the definition of the set $ \CB_{1, n, k}^{\lb 1\rb}(\theta_{1,n}^{(1)})\cap\CB_{3, n, 1}^{\lb 1\rb}(\theta_{1,n}^{(1)})$ in the last inequality. Due to assumption \ref{ass_leading_spiked_ev} the spectral norm of $\bfSigma_n^{(1)}$ is bounded, which implies that
\begin{align}\label{eq_spectral_norm_G_nk_1_intermediate_estimate}
	\lnorm \textbf{G}_{n, k-1}^{\lb 1\rb}\rnorm_2 \leq \lnorm \bfSigma_n^{(1)}\rnorm_2\lnorm\sum_{j=1}^{k-1} \lb\tilde{\bfx}_{j}^{\lb 1\rb}\tilde{\bfx}_{j}^{\lb 1\rb^\top}-\mathbf{\Sigma}_{n}^{\lb 1\rb}\rb\rnorm_2 \leq C\lb \lnorm \sum_{j=1}^{k-1}\tilde{\bfx}_j^{\lb 1\rb}\tilde{\bfx}_j^{\lb 1\rb^\top}\rnorm_2 +n\rb.
\end{align}
Furthermore, 
\begin{align}\label{eq_spectral_norm_sample_covariance_matrix_tilde_x}
	\lnorm \sum_{j=1}^{k-1} \tilde{\bfx}_j^{\lb 1\rb}\tilde{\bfx}_j^{\lb 1\rb^\top}\rnorm_2 &\leq \lb k-1\rb\lnorm \bfSigma_n^{(1)}\rnorm_2\lnorm \frac{1}{k-1}\sum_{j=1}^{k-1}\bfz_{j}^{\lb 1\rb}\bfz_{j}^{\lb 1\rb^\top}\rnorm_2\leq Cn,
\end{align}
where the last inequality follows from Theorem \ref{thm_largest_eigenvalue_high_dim_covariance_matrx}. Thus, the right-hand side of \eqref{eq_spectral_norm_G_nk_1_intermediate_estimate} can be further estimated such that
\begin{align}\label{eq_spectral_norm_G_nk_1_estimate}
    \lnorm \bfG_{n,k-1}^{(1)}\rnorm_2 \leq Cn \leq C(n+1).
\end{align}
Consequently, by the estimates \eqref{eq_trace_conditional_expectation_Pnk_1} and \eqref{eq_spectral_norm_G_nk_1_estimate}, we can further estimate the first term in \eqref{eq_first_cross_term_summand_estimation_start} by
\begin{align}\label{eq_first_cross_term_first_summand_estimate}
    \lv \sum_{j=1}^{p}\lb \E_{k}\lsb \dot{\textbf{P}}_{n, k}^{\lb 1\rb}(\theta_{1,n}^{(1)})\rsb\rb_{j, j}\lb \textbf{G}_{n, k-1}^{\lb 1\rb}\rb_{j, j}\rv\leq C(n+1).
\end{align}
For the second term on the right-hand side of \eqref{eq_first_cross_term_summand_estimation_start},
\begin{align}\label{eq_first_cross_term_second_summand_estimate}
\lv \tr\lb \E_{k}\lsb \dot{\bfP}_{n,k}^{(1)}(\theta_{1,n}^{(1)})\rsb \bfG_{n,k-1}^{(1)}\rb\rv &\leq \lv \tr\lb \E_{k}\lsb \dot{\textbf{P}}_{n, k}^{\lb 1\rb}(\theta_{1,n}^{(1)})\rsb \bfSigma_n^{(1)^{1/2}}\sum_{j=1}^{k-1} \tilde{\bfx}_j^{\lb 1\rb}\tilde{\bfx}_j^{\lb 1\rb^\top}\bfSigma_n^{(1)^{1/2}}\rb\rv
	\\
	&\phantom{{}={}}+\lv \tr\lb \E_{k}\lsb \dot{\textbf{P}}_{n,k}^{\lb 1\rb}(\theta_{1,n}^{(1)})\rsb\mathbf{\Sigma}_n^{\lb 1\rb^2}\rb\rv\nonumber
    \\
    &\leq \lnorm \bfSigma_n^{(1)^{1/2}}\sum_{j=1}^{k-1}\tilde{\bfx}_j^{\lb 1\rb}\tilde{\bfx}_j^{\lb 1\rb^\top}\bfSigma_n^{(1)^{1/2}}\rnorm_2 \tr\lb\E_{k}\lsb \dot{\textbf{P}}_{n, k}^{\lb 1\rb}(\theta_{1,n}^{(1)})\rsb\rb \nonumber
	\\
	&\phantom{{}={}}+ \tr\lb \E_{k}\lsb \dot{\textbf{P}}_{n, k}^{\lb 1\rb}(\theta_{1,n}^{(1)})\rsb\rb \lnorm \mathbf{\Sigma}_n^{\lb 1\rb^2}\rnorm_2\nonumber
	\\
    &\leq C\lcb \lnorm \bfSigma_n^{(1)}\rnorm_2\lnorm \sum_{j=1}^{k-1}\tilde{\bfx}_j^{\lb 1\rb}\tilde{\bfx}_j^{\lb 1\rb^\top}\rnorm_2 +\lnorm \bfSigma_n^{(1)}\rnorm_2^{2}\rcb\nonumber \leq C(n+1),
\end{align}
where we have used Von Neumann\textquotesingle s trace inequality (see Lemma \ref{lem_von_Neumann}) and the fact that $\dot{\bfP}_{n, k}^{\lb 1\rb}(\theta_{1,n}^{(i)})$ is positive semi-definite in the second inequality, the estimate \eqref{eq_trace_conditional_expectation_Pnk_1} in the third inequality, the estimate \eqref{eq_spectral_norm_sample_covariance_matrix_tilde_x} and the fact that the spectral norm of $\mathbf{\Sigma}_n^{\lb 1\rb}$ is bounded due to assumption \ref{ass_leading_spiked_ev} in the last inequality.
\newline
Plugging the estimates \eqref{eq_first_cross_term_first_summand_estimate} and \eqref{eq_first_cross_term_second_summand_estimate} into \eqref{eq_first_cross_term_summand_estimation_start} yields that
\begin{align*}
	&\lv \E_{k-1}\lsb \E_{k}\lsb \dot{\delta}_{n, k}^{\lb 1\rb}(\theta_{1,n}^{(1)})\rsb \lcb \bfz_k^{\lb 1\rb^\top}\textbf{G}_{n, k-1}^{\lb 1\rb}\bfz_k^{\lb 1\rb}-\tr\lb \textbf{G}_{n, k-1}^{\lb 1\rb}\rb\rcb\rsb\rv \leq \frac{C}{n}\lb n+1\rb
\end{align*}
for all $2\leq k\leq n$. By these estimates, it follows that the expectation in \eqref{eq_first_cross_term_removing_terms} can be bounded by
\begin{align*}
	&\frac{C}{n\lb n-1\rb^2}\E\lsb \lcb \sum_{k=1}^{n}\lv \E_{k-1}\lsb \E_{k}\lsb \dot{\delta}_{n, k}^{\lb 1\rb}(\theta_{1,n}^{(1)})\rsb\lcb \bfz_k^{\lb 1\rb^\top}\textbf{G}_{n, k-1}^{\lb 1\rb}\bfz_k^{\lb 1\rb}-\tr\lb \textbf{G}_{n, k-1}^{\lb 1\rb}\rb\rcb\rsb\rv\rcb^2\rsb
	\\
	\leq&\frac{C}{n\lb n-1\rb^2}\E\lsb \lcb \sum_{k=2}^{n} \frac{C}{n}\lb n+1\rb\rcb^2\rsb\leq\frac{C\lb n+1\rb^2}{n\lb n-1\rb^2}=o(1).\nonumber
\end{align*}
Hence, it follows that the expectation in \eqref{eq_first_cross_term_L2_start} is of order $o(1)$. Thus, the cross term $\CC_{1, n}(\theta_{1,n}^{(1)})$ converges in $L^2$ to $0$, which implies that
\begin{align*}
	\CC_{1,n}(\theta_{1,n}^{(1)}) = o_{\PR}\lb 1\rb
\end{align*}
holds.

\medskip

Using similar arguments as for the first cross term $\CC_{1,n}(\theta_{1,n}^{(1)})$, we show that $\CC_{2,n}(\theta_{1,n}^{(1)})$ converges to $0$ in $L^2$. To be precise, we consider
\begin{align}\label{eq_second_cross_term_L2_start}
	\E\lsb \lv \sum_{k=1}^{n}\E_{k-1}\lsb \E_{k}\lsb \sqrt{n}\dot{b}_{n}^{\lb 1\rb}(\theta_{1,n}^{(1)})\dot{\delta}_{n, k}^{\lb 1\rb}(\theta_{1,n}^{(1)})\rsb\frac{1}{\sigma_{n,1}}\frac{2}{n}\lcb \bfz_k^{\lb 1\rb^\top}\lb\bfSigma_{n}^{(1)^2}-\bfH_{n}^{(1,2)}\rb\bfz_k^{\lb 1\rb}-\tr\lb \bfSigma_{n}^{(1)^2}-\bfH_{n}^{(1,2)}\rb\rcb\rsb\rv^2\rsb
\end{align}
and show that the expectation is of order $o(1)$. As we already discussed in the previous proof for the first cross term, we consider the above expectation and the following (conditional) expectations on the event $\CB_{1, n, k}^{\lb 1\rb}(\theta_{1,n}^{(1)})\cap \CB_{3, n, 1}^{\lb 1\rb}(\theta_{1,n}^{(1)})$ such that the expressions $\dot{b}_n^{\lb 1\rb}(\theta_{1,n}^{(1)})$ and $\dot{\delta}_{n,k}^{\lb 1\rb}(\theta_{1,n}^{(1)})$ are well defined. Furthermore, we know $\dot{b}_n^{\lb 1\rb}(\theta_{1,n}^{(1)})$ and $\sigma_{n,1}^{-1}$ are uniformly bounded. Thus, it suffices to prove
\begin{align}\label{eq_second_cross_term_removing_terms}
	\E\lsb \lv \sum_{k=1}^{n}\E_{k-1}\lsb \E_{k}\lsb \sqrt{n}\dot{\delta}_{n, k}^{\lb 1\rb}(\theta_{1,n}^{(1)})\rsb\frac{2}{n}\lcb \bfz_k^{\lb 1\rb^\top}\lb\bfSigma_{n}^{(1)^2}-\bfH_{n}^{(1,2)}\rb\bfz_k^{\lb 1\rb}-\tr\lb\bfSigma_{n}^{(1)^2}-\bfH_{n}^{(1,2)}\rb\rcb\rsb\rv^2\rsb
\end{align}
is of order $o(1)$. For this purpose, we give an estimate for each summand of the
expression above. Again, for $1\leq k \leq n$, there exists there exists a constant $C>0$ by Lemma \ref{ConditionalExpectedValueOfProductOfNormalizedQuadraticForms}, such that 
\begin{align}\label{eq_second_cross_term_summand_estimation_start}
	&\lv \E_{k-1}\lsb \E_{k}\lsb \dot{\delta}_{n, k}^{\lb 1\rb}(\theta_{1,n}^{(1)})\rsb \lcb \bfz_k^{\lb 1\rb^\top}\lb\bfSigma_{n}^{(1)^2}-\bfH_{n}^{(1,2)}\rb\bfz_k^{\lb 1\rb}-\tr\lb\bfSigma_{n}^{(1)^2}-\bfH_{n}^{(1,2)}\rb\rcb\rsb\rv
	\\
	\leq &\frac{C}{n}\lcb \lv\sum_{j=1}^{p}\lb \E_{k}\lsb\dot{\textbf{P}}_{n, k}^{\lb 1\rb}(\theta_{1,n}^{(1)})\rsb\rb_{j, j}\lb\bfSigma_{n}^{(1)^2}-\bfH_{n}^{(1,2)}\rb_{j,j}\rv + \lv \tr\lb \E_{k}\lsb \dot{\textbf{P}}_{n, k}^{\lb 1\rb}(\theta_{1,n}^{(1)})\rsb\lb\bfSigma_{n}^{(1)^2}-\bfH_{n}^{(1,2)}\rb\rb\rv\rcb\nonumber
\end{align}
holds. To give an estimate for the right-hand side of \eqref{eq_second_cross_term_summand_estimation_start}, we derive estimates for both summands.
\newline
We begin with the first summand, and note that it is easy to see that
\begin{align}\label{eq_spectral_norm_covariance_matrix_H_12_estimate}
    \lv \lb\bfSigma_{n}^{(1)^2}-\bfH_{n}^{(1,2)}\rb_{j,j} \rv \leq \lnorm \bfSigma_{n}^{(1)^2}-\bfH_{n}^{(1,2)} \rnorm_2  \leq \lnorm \bfSigma_{n}^{(1)}\rnorm_2^2 + \lnorm \bfH_{n}^{(1,2)}\rnorm_2 \leq C,
\end{align}
where we used for the last inequality that $\lnorm \bfSigma_{n}^{(1)}\rnorm_2$ and $\lnorm \textbf{H}_n^{\lb 1, 2\rb}\rnorm_2$ are bounded due to assumption \ref{ass_leading_spiked_ev}.
Using this estimation and $\lb \dot{\textbf{P}}_{n, k}^{\lb 1\rb}(\theta_{1,n}^{(1)})\rb_{j, j}\geq 0$, it follows that the first sum on the right-hand side of \eqref{eq_second_cross_term_summand_estimation_start} is bounded by
\begin{align}\label{eq_second_cross_term_first_summand_estimate}
	\lv \sum_{j=1}^{p}\lb \E_{k}\lsb \dot{\textbf{P}}_{n, k}^{\lb 1\rb}(\theta_{1,n}^{(1)})\rsb\rb_{j, j}\lb \bfSigma_{n}^{(1)^2}-\textbf{H}_{n}^{\lb 1,2\rb}\rb_{j, j}\rv&\leq 
	C \lv\sum_{j=1}^{p}\lb \E_{k}\lsb \dot{\textbf{P}}_{n,k}^{\lb 1\rb}(\theta_{1,n}^{(1)})\rsb\rb_{j, j} \rv
    \\
    &= C \tr\lb \E_{k}\lsb \dot{\textbf{P}}_{n, k}^{\lb 1\rb}(\theta_{1,n}^{(1)})\rsb\rb \leq C,\nonumber
\end{align}
where we have used the estimate \eqref{eq_trace_conditional_expectation_Pnk_1} in the last inequality. For the second sum on the right-hand side of \eqref{eq_second_cross_term_summand_estimation_start}, Von Neumann\textquotesingle s trace inequality (see Lemma \ref{lem_von_Neumann}) implies
\begin{align*}
	\lv \tr\lb \E_{k}\lsb \dot{\textbf{P}}_{n, k}^{\lb 1\rb}(\theta_{1,n}^{(1)})\rsb\lb \bfSigma_{n}^{(1)^2}-\textbf{H}_{n}^{\lb 1,2\rb}\rb\rb\rv&\leq \lv \tr\lb \E_{k}\lsb \dot{\textbf{P}}_{n, k}^{\lb 1\rb}(\theta_{1,n}^{(1)})\rsb\bfSigma_{n}^{(1)^2}\rb\rv+\lv\tr\lb \E_{k}\lsb \dot{\textbf{P}}_{n,k}^{\lb 1\rb}(\theta_{1,n}^{(1)})\rsb \textbf{H}_{n}^{\lb 1, 2\rb}\rb\rv
	\\
	&\leq \lb \lnorm \bfSigma_{n}^{(1)}\rnorm_2^2 + \lnorm \textbf{H}_{n}^{\lb 1, 2\rb}\rnorm_2 \rb \tr\lb \E_{k}\lsb \dot{\textbf{P}}_{n, k}^{\lb 1\rb}(\theta_{1,n}^{(1)})\rsb\rb\leq C,
\end{align*}
where the last inequality follows from the estimations \eqref{eq_trace_conditional_expectation_Pnk_1} and \eqref{eq_spectral_norm_covariance_matrix_H_12_estimate}.
Plugging this bound and \eqref{eq_second_cross_term_first_summand_estimate} into the estimate \eqref{eq_second_cross_term_summand_estimation_start} yields that
\begin{align*}
	\lv \E_{k-1}\lsb\E_{k}\lsb \dot{\delta}_{n, k}^{\lb 1\rb}(\theta_{1,n}^{(1)})\rsb \lcb \bfz_k^{\lb 1\rb^\top}\lb \bfSigma_{n}^{(1)^2}-\textbf{H}_{n}^{\lb 1, 2\rb}\rb \bfz_{k}^{\lb 1\rb}-\tr\lb \bfSigma_{n}^{(1)^2}-\textbf{H}_n^{\lb 1, 2\rb}\rb\rcb\rsb\rv\leq \frac{C}{n}
\end{align*}
for all $1\leq k\leq n$. By these estimates, it follows that the expectation in \eqref{eq_second_cross_term_removing_terms} can be bounded by
\begin{align*}
	&\frac{C}{n}\E\lsb \lcb \sum_{k=1}^{n}\lv \E_{k-1}\lsb \E_{k}\lsb \dot{\delta}_{n, k}^{\lb 1\rb}(\theta_{1,n}^{(1)})\rsb\lcb \bfz_k^{\lb 1\rb^\top}\lb \bfSigma_{n}^{(1)^2}-\textbf{H}_n^{\lb 1, 2\rb}\rb \bfz_k^{\lb 1\rb}-\tr\lb \bfSigma_{n}^{(1)^2}-\textbf{H}_n^{\lb 1, 2\rb}\rb\rcb\rsb\rv\rcb^2\rsb
	\\
	\leq &\frac{C}{n}\E\lsb\lcb \sum_{k=1}^{n}\frac{C}{n}\rcb^2\rsb\leq\frac{C}{n}= o(1).\nonumber
\end{align*}
Hence, it follows that the expectation in \eqref{eq_second_cross_term_L2_start} is of order $o(1)$. Thus, the cross term $\CC_{2,n}(\theta_{1,n}^{(1)})$ converges in $L^2$ to 0, which implies that
\begin{align*}
	\CC_{2,n}(\theta_{1,n}^{(1)}) = o_{\PR}\lb 1\rb
\end{align*}
holds.
\newline
We emphasize that the proofs for the cross terms $\CC_{\ell,n}(\theta_{1,n}^{(2)}) = o_{\PR}\lb 1\rb$ $(\ell=3,4,5)$ are analogous to those for the previous cross terms $\CC_{\ell,n}(\theta_{1,n}^{(1)}) = o_{\PR}\lb 1\rb$ for $\ell=1,2$. Consequently, we omit these proofs. 
\end{proof}
\subsection{Proof of Corollary \ref{cor_generalized_fisher_combi_1_norm_performance}} 
Our goal is to show that
\begin{align} \label{eq_fisher_combi_power}
    \PR\lb T_{n, FC,m}> q_{1-\alpha} \mid \mathsf{H}_{1} \rb = \PR\lb \log(p_{n,1})+ \log(p_{n,2,m}) < -\frac{q_{1-\alpha}}{2}\mid \mathsf{H}_{1} \rb \to 1, \quad n\to \infty
\end{align}
holds if either \eqref{eq_frobenius_alternative} or \eqref{eq_generalized_eigenvalues_alternative} is satisfied.
\newline
First, we notice that due to the definition of the $p$-values $p_{n,1}$ and $p_{n,2,m} \, (m=1,\ldots, K)$ in \eqref{eq_def_p_value} and \eqref{eq_def_generalized_p_value_spi}, it follows immediately that
\begin{align*}
    \log(p_{n,1}), \log(p_{n,2,m}) \leq \log(2).
\end{align*}
Furthermore, we remember that because of Lemma \ref{lem_consistent_estimators} and \ref{lem_ind_quad_form_frobenius_non_cen}, we have as a consequence of Slutsky's Theorem
\begin{align}
    \breve{T}_{n,1}= \frac{\hat{\sigma}_{n,1}}{\sigma_{n,1}}T_{n,1}-\frac{\lnorm \bfSigma_n^{(1)}-\bfSigma_n^{(2)}\rnorm_{F}^2}{\sigma_{n,1}}&\cond \mathcal{N}(0,1), \label{eq_clt_breve_T}
    \\
    \mathbf{T}_{n,2,m}^{\star}&\cond \mathcal{N}_{m}(\mathbf{0}_m, \bfSigma_{E,m}) \quad n\to \infty, \label{eq_clt_tilde_T_m_critical_ev}
\end{align}
where 
\begin{align*}
    \mathbf{T}_{n,2,m}^{\star}=\sqrt{n}\lb \lambda_1(\bfS_n^{(1)}) - \theta_{1,n}^{(1)} - (\lambda_1(\bfS_n^{(2)}) - \theta_{1,n}^{(2)}), \ldots, \lambda_m(\bfS_n^{(1)}) - \theta_{m,n}^{(1)} - (\lambda_m(\bfS_n^{(2)}) - \theta_{m,n}^{(2)})\rb^\top.
\end{align*}
We start assuming that \eqref{eq_frobenius_alternative} is satisfied and let $c\in \R$ be an arbitrary constant. Because of the CLT in \eqref{eq_clt_breve_T} along with Slutsky's Theorem, we have
\begin{align}
    \PR\lb  T_{n,1} >c \mid\mathsf{H}_{1}\rb &= \PR\lb \breve{T}_{n,1} > c \frac
    {\hat{\sigma}_{n,1}}{\sigma_{n,1}}- \frac{\lnorm \bfSigma_{n}^{(1)}-\bfSigma_{n}^{(2)}\rnorm_{F}^2}{\sigma_{n,1}}\mid \mathsf{H}_{1}\rb 
    \nonumber \\
    &=1-\Phi\lb c \frac
    {\hat{\sigma}_{n,1}}{\sigma_{n,1}}- \frac{\lnorm \bfSigma_{n}^{(1)}-\bfSigma_{n}^{(2)}\rnorm_{F}^2}{\sigma_{n,1}}\rb + o(1)\nonumber
    \nonumber \\
    &=1-\Phi\lb c \frac{\hat{\sigma}_{n,1}}{\frac{2}{n}\lb \tr(\bfSigma_{n}^{(1)^2}+\bfSigma_n^{(2)^2})\rb}\mathcal{Z}_{n}(\bfSigma_n^{(1)},\bfSigma_{n}^{(2)})-\frac{\lnorm \bfSigma_{n}^{(1)}-\bfSigma_{n}^{(2)}\rnorm_{F}^2}{\sigma_{n,1}}\rb + o(1), \label{eq_frobenius_statistics_probability_unbounded_under_alternative}
\end{align}
where
\begin{align*}
    \mathcal{Z}_{n}(\bfSigma_{n}^{(1)},\bfSigma_n^{(2)})=\frac{\frac{2}{n}\lb \tr(\bfSigma_{n}^{(1)^2}+\bfSigma_n^{(2)^2})\rb}{\sigma_{n,1}}.
\end{align*}
Because of the fact that $\sigma_{n,1}^{-1}$ is bounded, it follows by the assumptions \ref{ass_p_n} and \ref{ass_leading_spiked_ev} that
\begin{align*}
    \frac{2}{n} \frac{\tr(\bfSigma_{n}^{(1)^2}+\bfSigma_n^{(2)^2})}{\sigma_{n,1}} \leq C \frac{p}{n}\leq C.
\end{align*}
Furthermore, we have by Lemma \ref{lem_consistent_estimators} that
\begin{align*}
    \frac{\hat{\sigma}_{n,1}}{\frac{2}{n}\lb \tr(\bfSigma_{n}^{(1)^2}+\bfSigma_n^{(2)^2})\rb}\conp 1.
\end{align*}
Thus, \eqref{eq_frobenius_statistics_probability_unbounded_under_alternative} is dominated by 
\begin{align*}
    1 - \Phi \lb -\lnorm \bfSigma_n^{(1)}-\bfSigma_n^{(2)}\rnorm_F^2 \rb .
\end{align*}
As we assume that \eqref{eq_frobenius_alternative} holds, we conclude 
\begin{align*}
    \PR\lb  T_{n,1} >c \mid \mathsf{H}_{1}\rb \to 1, \quad n\to \infty ~.
\end{align*}
Consequently, for every $\varepsilon>0$, we have
\begin{align*}
    \PR(p_{n,1}>\varepsilon\mid \mathsf{H}_{1})\to 0, \quad n \to \infty ~.
\end{align*}
Finally, for any arbitrary constant $c\in \R$ it follows that
\begin{align*}
    \PR(\log(p_{n,1})<c \mid \mathsf{H}_{1}) \to 1, \quad n\to \infty ~.
\end{align*}
This leads to 
\begin{align} \label{a1}
    \PR\lb \log(p_{n,1})+\log(p_{n,2,m})< -\frac{q_{1-\alpha}}{2}\mid \mathsf{H}_1\rb \to 1, \quad n \to \infty,
\end{align}
where we used the fact that $\log(p_{n,2,m})\leq \log(1)$.
\newline
Now, we assume \eqref{eq_generalized_eigenvalues_alternative} is satisfied. Then, we can show with similar arguments that \eqref{eq_fisher_combi_power} holds as well. Let $c\in \R$ be an arbitrary constant. By the CLT for leading supercritical sample eigenvalue in \eqref{eq_clt_tilde_T_m_critical_ev} and Lemma \ref{lem_ind_quad_form_frobenius_non_cen} along with Slutsky's Theorem, we get 
\begin{align}\label{eq_generalized_ev_statistics_probability_unbounded_under_alternative}
    \PR( T_{n,2,m} > c\mid \mathsf{H}_{1}) &\geq \PR\lb \lnorm\mathbf{T}_{n,2,m}^\star\rnorm_1 > c-\sqrt{n}\sum_{j=1}^{m}\lv\psi_{n,K}^{(1)}(\alpha_j^{(1)})-\psi_{n,K}^{(2)}(\alpha_{j}^{(2)})\rv\mid \mathsf{H}_{1}\rb
    \\
    &= 1-F\lb c-\sqrt{n}\sum_{j=1}^{m}\lv\psi_{n,K}^{(1)}(\alpha_j^{(1)})-\psi_{n,K}^{(2)}(\alpha_{j}^{(2)})\rv\rb + o(1)\nonumber,
\end{align}
where $F$ is the cumulative distribution function of $\lnorm \mathbf{W}\rnorm_1$ with $\mathbf{W}\sim \mathcal{N}_m (\mathbf{0}_m, \bfSigma_{E,m})$.
Thus, the right-hand side of \eqref{eq_generalized_ev_statistics_probability_unbounded_under_alternative} depends primarily on $-\sqrt{n}\sum\limits_{j=1}^{m}\lv\psi_{n,K}^{(1)}(\alpha_j^{(1)})-\psi_{n,K}^{(2)}(\alpha_{j}^{(2)})\rv$. As we assume that one summand converges to infinity, we conclude
\begin{align*}
    \PR \lb T_{n,2,m} >c \mid \mathsf{H}_{1} \rb \to 1, \quad n\to \infty ~.
\end{align*}
Similar arguments leading to \eqref{a1} show that 
\begin{align*}
    \PR\lb \log(p_{n,1})+\log(p_{n,2,m})< -\frac{q_{1-\alpha}}{2}\mid \mathsf{H}_1\rb \to 1, \quad n \to \infty ~.
\end{align*}

\bigskip

{\bf Acknowledgments}
The authors would like to thank Guangming Pan and Zhixiang Zhang for some helpful discussions about their work \cite{zhang2022asymptotic} and for sharing their code. 
\bigskip

{\bf  Funding}
The work of  H. Dette and T. Lam  was partially supported by the  
Deutsche Forschungsgemein-
schaft  (DFG),   {\it Research unit 5381: Mathematical Statistics in the Information Age}, project number 460867398,  
 and  partially funded by the Deutsche Forschungsgemeinschaft (DFG, German Research Foundation) under \textit{Germany's Excellence Strategy - EXC 2092 CASA - 390781972.}
 The work of N. Dörnemann was partially supported by  the Aarhus University Research Foundation (AUFF), project numbers 47221 and 47388.

\setlength{\bibsep}{1pt}
\begin{small}
\bibliography{references}
\end{small}
\newpage 
\appendix
\section{Appendix} \label{sec_appendix}
\subsection{Background results}

\begin{theorem}[Corollary 3.1 in \cite{hallheyde1980martingales}]\label{thm_CLT_for_MDS}
    Suppose that for each $n$, the sequence of random variables $Y_{n, 1}, Y_{n_, 2}, \ldots, Y_{n, r_n}$ is a real martingale difference sequence with respect to an increasing sequence of $\sigma$-algebras $\lcb \CF_{n, j}\rcb_{j=1}^{r_n}$ and with existing second moments. If
    \begin{enumerate}[label=(\arabic*)]
    \item $\sum\limits_{j=1}^{r_n}\E\lsb Y_{n,j}^{2}\mid \CF_{n, j-1}\rsb \conp \sigma^2$, where $\sigma^2$ is a positive constant, and
	\\
	\item for each $\varepsilon>0$, $\sum\limits_{j=1}^{r_n}\E\lsb Y_{n,j}^2\mathds{1}\lb\lv Y_{n,j}\rv >\varepsilon\rb\rsb\rightarrow 0$
	\end{enumerate}
	as $n\to\infty$ then
	\begin{align*}
		\sum_{j=1}^{r_n}Y_{n,j}\cond\mathcal{N}( 0, \sigma^2), \quad n\to \infty ~.
	\end{align*}
\end{theorem}

\begin{lemma}\label{ConditionalExpectedValueOfProductOfNormalizedQuadraticForms}
    Let $\bfx = (x_1, \ldots, x_p)^\top$ be a $p$-dimensional random vector with unit variance and i.i.d. entries such that $\E\lsb x_{i}\rsb =0$, and $\E\lsb x_{i}^{4}\rsb =\gamma < \infty$, $i=1, \ldots, p$. Suppose $\bfA = (a_{i,j})$ and $\bfB=(b_{i,j})$ are matrices in $\mathbb{C}^{p \times p}$. Then, we have the identity
    \begin{align*}
        \E \lsb \lb \bfx^\top \bfA \bfx - \tr\lb \bfA \rb\rb \lb \bfx^\top \bfB \bfx - \tr\lb \bfB \rb\rb\rsb = \lb \E\lsb x_1^{4}\rsb-3\rb\sum_{i=1}^{p} a_{i,i}b_{i,i} + \tr\lb \bfA\bfB^\top\rb + \tr\lb \bfA\bfB\rb.
    \end{align*}
\end{lemma}

\begin{theorem}[Theorem $3.1$ in \cite{yinbaikrishnaiah1988}]\label{thm_largest_eigenvalue_high_dim_covariance_matrx}
	Suppose that $\mathbf{X}_n$ is a random $p \times n$ matrix consisting of i.i.d. random variables with $\E\lsb x_{1, 1}\rsb =0$, $\Var\lb x_{1, 1}\rb =\sigma^2$ and finite fourth moment and let $\mathbf{S}_n=\frac{1}{n}\mathbf{X}_n\mathbf{X}_n^\top$. Then 
	\begin{align*}
		\lambda_{\max}\lb \mathbf{S}_n\rb \conas \sigma^2\lb 1+\sqrt{y}\rb^2
	\end{align*}
	as $n\rightarrow \infty$, $p\rightarrow\infty$ such that $p/n\rightarrow y>0$.
\end{theorem}

\begin{lemma}(Von Neumann\textquotesingle s trace inequality) \label{lem_von_Neumann}
	Suppose that $\mathbf{A}$ and $\mathbf{B}$ are $n\times n$ complex matrices with singular values $\alpha_1\leq \ldots \leq \alpha_n$ and $\beta_1\leq \ldots \leq \beta_n$, respectively. Then
	\begin{align*}
		\lv \tr\lb \mathbf{A}\mathbf{B}\rb\rv \leq \sum_{i=1}^{n}\alpha_i\beta_i.
	\end{align*}
	In particular, if $\mathbf{A}, \mathbf{B}$ are Hermitian matrices, then
	\begin{align*}
		\sum_{j=1}^{n}\lambda_j(\bfA)\lambda_{n-j+1}(\bfB) \leq \tr\lb \mathbf{A}\mathbf{B}\rb \leq \sum_{j=1}^{n}\lambda_j(\bfA)\lambda_j(\bfB),
	\end{align*}
	where $\lambda_1(\bfA)\leq \ldots \leq \lambda_n(\bfA)$ and $\lambda_1(\bfB)\leq \ldots \leq \lambda_n(\bfB)$ are the eigenvalues of $\mathbf{A}$ and $\mathbf{B}$, respectively.
\end{lemma}
A proof of this inequality can be found in \cite{mirsky1975}.
    \begin{lemma}\label{lem_lower_bounded_variances}
        Suppose assumption \ref{ass_p_n} - \ref{ass_population_lsd} and \ref{ass_supercritical_ev_generalized} are satisfied, then it holds that
        \begin{align*}
        0 < C \leq \inf_{n\in \mathbb{N}} \sigma_{n,1}^{2} \leq \sup_{n\in \mathbb{N}} \sigma_{n,1}^2 < \infty
        \end{align*}
        and for $j=1, \ldots, K$
        \begin{align*}
        0 < \inf_{n\in \mathbb{N}} \sigma_{\spi, j, n}^{(i)^2} \leq \sup_{n\in \mathbb{N}} \sigma_{\spi, j, n}^{(i)^2} < \infty.
        \end{align*}
    \end{lemma}
    \begin{proof}
        First, we note that for any symmetric matrix $\bfA$, the inequality $\tr\lb \bfA \circ \bfA\rb \leq \tr\lb \bfA^2 \rb$ holds. Consequently, this implies that
        \begin{align*}
				\tr\lb\lcb \bfSigma_n^{(i)^{1/2}}\lb \mathbf{\Sigma}_{n}^{\lb 1\rb}-\mathbf{\Sigma}_{n}^{\lb 2\rb}\rb\bfSigma_n^{(i)^{1/2}}\rcb \circ \lcb\bfSigma_n^{(i)^{1/2}}\lb \mathbf{\Sigma}_{n}^{\lb 1\rb}-\mathbf{\Sigma}_{n}^{\lb 2\rb}\rb\bfSigma_n^{(i)^{1/2}}\rcb\rb \leq \tr\lb \lcb \mathbf{\Sigma}_{n}^{\lb i\rb^2}-\mathbf{\Sigma}_{n}^{\lb 1\rb}\mathbf{\Sigma}_{n}^{\lb 2\rb}\rcb^2\rb
		\end{align*}
		holds for $i=1, 2$. Due to assumption \ref{ass_random_vectors} it yields that $-2\leq \gamma_4^{(i)}-3$, which leads to
        \begin{align*}
				&\frac{4(\gamma_4^{(i)}-3)}{n}\tr\lb\lcb \bfSigma_n^{(i)^{1/2}}\lb \mathbf{\Sigma}_{n}^{\lb 1\rb}-\mathbf{\Sigma}_{n}^{\lb 2\rb}\rb\bfSigma_n^{(i)^{1/2}}\rcb \circ \lcb\bfSigma_n^{(i)^{1/2}}\lb \mathbf{\Sigma}_{n}^{\lb 1\rb}-\mathbf{\Sigma}_{n}^{\lb 2\rb}\rb\bfSigma_n^{(i)^{1/2}}\rcb\rb
				\\
				&\phantom{{}={}}+\frac{8}{n}\tr\lb \lcb \mathbf{\Sigma}_{n}^{\lb i\rb^2}-\mathbf{\Sigma}_{n}^{\lb 1\rb}\mathbf{\Sigma}_{n}^{\lb 2\rb}\rcb^2\rb >0, \quad i=1,2.
			\end{align*}
        Thus, we obtain
            \begin{align*}
               \sigma_{n,1}^{2}\geq \frac{4}{n^2}\tr^2\lb \mathbf{\Sigma}_{n}^{\lb 1\rb^2}\rb\geq \frac{4}{n^2}\lb p\lambda_{\min}\lb \mathbf{\Sigma}_{n}^{\lb 1\rb}\rb\rb^2 \geq \frac{4}{n^2}\lb p L^{\lb 1\rb}\rb^2\geq C > 0, 
            \end{align*}
            where we used the assumptions \ref{ass_p_n} and \ref{ass_population_lsd} and $C$ is a positive constant. On the other hand, due to assumption \ref{ass_p_n} it is easy to verify that $\sup_{n\in \mathbb{N}} \sigma_{n,1}^{2} < \infty$.
            \newline
            First, we recall that $\mathbf{u}_{j,n}^{(i)}$ is the $j$-th column of the orthogonal matrix $\bfU_n^{(i)}$, which leads to $\sum\limits_{k=1}^{p} u_{k,j,n}^4 \leq 1$. By assumption \ref{ass_supercritical_ev_generalized} and the definition of the function $\psi^{(i)^\prime}$, we have $0<\psi^{(i)^\prime}(\alpha_j^{(i)})<1$. With the fact that $-2\leq \gamma_4^{(i)}-3$, we conclude that 
        \begin{align*}
            \sigma_{\spi, j, n}^{(i)^2} \geq \begin{cases}
                2\alpha_j^{(i)^2}\psi^{(i)^\prime}(\alpha_j^{(i)}) & \text{if } \gamma_4^{(i)}-3 \geq 0
                \\
                (\gamma_4^{(i)}-3) \alpha_j^{(i)^2} (\psi^{(i)^\prime}(\alpha_j^{(i)}))^2 + 2 \alpha_j^{(i)^2} \psi^{(i)^\prime}(\alpha_j^{(i)}) & \text{if } \gamma_4^{(i)}-3 < 0
            \end{cases}
            \, > 0
        \end{align*}
        and
        \begin{align*}
        \sigma_{\spi,j,n}^{(i)^2} \leq \begin{cases}
         (\gamma_4^{(i)}-3) \alpha_j^{(i)^2} (\psi^{(i)^\prime}(\alpha_j^{(i)}))^2 + 2 \alpha_j^{(i)^2} \psi^{(i)^\prime}(\alpha_j^{(i)}) & \text{if } \gamma_4^{(i)}-3 \geq 0
         \\
         2\alpha_j^{(i)^2}\psi^{(i)^\prime}(\alpha_j^{(i)}) & \text{if } \gamma_4^{(i)}-3 < 0
        \end{cases}
        < \infty.
        \end{align*}
    \end{proof}
\subsection{Proof of Lemma \ref{lem_consistent_estimators} (consistency of estimators)} \label{sec_proof_consistent_estimators}
In this section, we prove that the estimators 
\begin{align*}
\hat\sigma_{n,1}^2 &= \frac{4}{n^2} \lb B_n^{(1)}+B_n^{(2)}\rb^2,
\\
\hat\sigma_{n,2}^2 &=\hat{\sigma}_{n,2}^{(1)^2}+\hat{\sigma}_{n,2}^{(2)^2}
\end{align*}

and $(\hat{\bfSigma}_{E,m,n})_{k, \ell}$ given in \eqref{eq_def_estimator_covariance_matrix_entries} are consistent. 

First, we give a complete description of $\hat\sigma_{n,2}^{(i)^2}$ and $(\hat{\bfSigma}_{E,m,n})_{k, \ell}$. Recalling the definition its population version $\sigma_{n,2}^{(i)^2}$ in \eqref{def_var_spike}, we introduce for $i=1,2$ and $1\leq k\leq K$ the estimators
\begin{align*}
\hat{\alpha}_{k,n}^{(i)}=\lb\frac{1-y_n}{\lambda_k(\bfS_n^{(i)})}+\frac{1}{n}\sum_{j\neq k}^{p}\frac{1}{\lambda_k(\bfS_n^{(i)})-\lambda_j(\bfS_n^{(i)})}\rb^{-1}
\end{align*}
for the supercritical eigenvalues $\alpha_{1}^{(i)}, \ldots, \alpha_{K}^{(i)},$
\begin{align*}
\hat{\xi}_{k,n}^{(i)} = \lcb\hat{\alpha}_{k,n}^{(i)^2}\lb \frac{1-y_n}{\lambda_k^2(\bfS_n^{(i)})}+\frac{1}{n}\sum_{j\neq k}^{p}\frac{1}{\lb\lambda_k(\bfS_n^{(i)})-\lambda_j(\bfS_n^{(i)})\rb^2}\rb\rcb^{-1}          
\end{align*}
for the quantities $\psi^{(i)^\prime}(\alpha_{1}^{(i)}), \ldots, \psi^{(i)^\prime}(\alpha_{K}^{(i)})$, and 
\begin{align*}
    \hat{\gamma}_{4,n}^{(i)} = \max\lcb 3+ \frac{\hat{\nu}^{(i)}_n-2\hat{\tau}_n^{(i)}}{\hat{\omega}_n^{(i)}},1 \rcb.
\end{align*}
for the kurtosis $\gamma_4^{(i)}$, where
\begin{align*}
    \hat{\tau}_n^{(i)} &= \tr\lb \bfS_n^{(i)^2}\rb -\frac{1}{n} \tr^2 \lb \bfS_n^{(i)}\rb,
    \\
    \hat{\nu}_n^{(i)} &= \frac{1}{n-1} \sum_{j=1}^{n}\lb \lnorm \bfx_j^{(i)}-\bar{\bfx}^{(i)} \rnorm^2- \frac{1}{n}\sum_{\ell=1}^{n} \lnorm \bfx_{\ell}^{(i)}-\bar{\bfx}^{(i)}\rnorm^2 \rb ^2,
    \\
    \hat{\omega}_n^{(i)}&= \sum_{j=1}^{p} \lcb \frac{1}{n}\sum_{k=1}^{n}\lb x_{j, k}^{(i)}-\frac{1}{n}\sum_{\ell=1}^{n}x_{j,\ell}^{(i)}\rb^2 \rcb^2.
\end{align*}
Next, we describe an estimator for the function of eigenvectors appearing in \eqref{def_var_spike}. To this end, let  $\mathbf{g}_{j,n}^{(i)}$ denote the eigenvector of the sample covariance matrix $\bfS_n^{(i)}$ associated with the eigenvalue $\lambda_j(\bfS_n^{(i)})$ and $g_{k, j,n}^{(i)}$ be the $k$-th coordinate of $\mathbf{g}_{j,n}^{(i)}$. In addition, we define
\begin{align*}
		\rho_{k}^{(i)}(m) &=
		\begin{cases}
			1+\eta_{k}^{(i)}(m) & \text{if } m=k
			\\
			-\zeta_{k}^{(i)}(m) & \text{if } m\neq k
		\end{cases},
        \\
		\eta_k^{(i)}(m)&=\sum_{\substack{\ell=1\\ \ell\neq k}}^{p}\lb \frac{\lambda_\ell(\mathbf{S}_n^{(i)})}{\lambda_m( \mathbf{S}_n^{(i)})-\lambda_\ell(\mathbf{S}_n^{(i)})}-\frac{\vartheta_\ell^{(i)}}{\lambda_m( \mathbf{S}_n^{(i)})-\vartheta_\ell^{(i)}}\rb,
		\\
		\zeta_k^{(i)}(m) &= \frac{\lambda_k( \bfS_n^{( i)})}{\lambda_m( \bfS_n^{(i)})-\lambda_k( \bfS_n^{(i)})}-\frac{\vartheta_k^{(i)}}{\lambda_m( \bfS_n^{(i)})-\vartheta_k^{( i)}},
	\end{align*}
	where $\vartheta_1^{(i)}\geq \vartheta_2^{(i)}\geq \ldots \geq \vartheta_p^{(i)}$ are the real valued solutions to the equation in $x$
	\begin{align*}
		\frac{1}{p}\sum_{j=1}^{p}\frac{\lambda_j(\mathbf{S}_n^{(i)})}{\lambda_j(\mathbf{S}_n^{(i)})-x}=\frac{1}{y}.
	\end{align*}
	When $y>1$, take $\vartheta_{n}^{(i)} =\ldots =\vartheta_p^{(i)}=0.$ In the expressions of $\zeta_k^{(i)}(m)$ and $\eta_k^{(i)}(m)$, we use the convention that any term of the form $\frac{0}{0}$ is $0$.
Then, for $k, \ell =1, \ldots, K$ we set 
\begin{align*}
    \hat{\kappa}_{k,\ell,n}^{(i)}=\sum_{\ell^\prime=1}^{p} \lcb \sum_{m=1}^{p}\rho_{k}^{(i)}(m) g_{\ell^\prime, m,n}^{(i)^2}\rcb \lcb \sum_{m^\prime=1}^{p}\rho_{\ell}^{(i)}(m^\prime) g_{\ell^\prime, m^\prime,n}^{(i)^2}\rcb
\end{align*}
as an estimator for $\sum\limits_{j=1}^p ( u_{j,k,n}^{(i)} ) ^2( u_{j,\ell,n}^{(i)} )^2$.
\newline
Finally, the complete description of the estimator $\sigma_{n,2}^2$ is given by 
\begin{align*}
    \hat{\sigma}_{n,2}^{2}=\hat{\sigma}_{\spi,1,n}^{(1)^2}+\hat{\sigma}_{\spi,1,n}^{(2)^2},
\end{align*}
and the entries $(\hat{\bfSigma}_{E,m,n})_{k,\ell}$  of the estimator $(\hat{\bfSigma}_{E,m,n})$ for $\bfSigma_{E,m}$ is given by 
\begin{align*}
    (\hat{\bfSigma}_{E,m,n})_{k, \ell}=\begin{cases}
        \hat{\sigma}_{\spi, k,n}^{(1)^2} + \hat{\sigma}_{\spi, k,n}^{(2)^2}  & \text{if }k= \ell
        \\
        \hat{\sigma}_{\spi, k, \ell,n}^{(1)} + \hat{\sigma}_{\spi, k, \ell,n}^{(2)} & \text{if }k\neq \ell
    \end{cases},
\end{align*}
where
\begin{align*}
    \hat{\sigma}_{\spi,k,n}^{(i)^2} & = (\hat{\gamma}_{4,n}^{(i)}-3)\hat{\alpha}_{k, n}^{(i)^2}\hat{\xi}_{k,n}^{(i)^2}\hat{\upsilon}_{k,n}^{(i)}+2\hat{\alpha}_{k, n}^{(i)^2}\hat{\xi}_{k,n}^{(i)},
    \\
    \hat{\sigma}_{\spi, k, \ell,n}^{(i)} &= (\hat{\gamma}_{4,n}^{(i)} -3) \hat{\alpha}_{k,n}^{(i)}\hat{\alpha}_{\ell,n}^{(i)} \hat{\xi}_{k,n}^{(i)}\hat{\xi}_{\ell, n}^{(i)} \hat{\kappa}_{k, \ell, n}^{(i)}.
\end{align*}
\begin{lemma}\label{lem_kurtosis_supercrit_estimators}
    Suppose that assumptions \ref{ass_p_n} - \ref{ass_population_lsd} and \ref{ass_supercritical_ev_generalized} holds. For $i=1,2$ and $k, \ell =1,\ldots, K$, then 
    \begin{enumerate}[label=(\alph*)]
        \item $\hat{\alpha}_{k,n}^{(i)}\conp \alpha_k^{(i)}$,
        \item $\hat{\xi}_{k,n}^{(i)} \conp \psi^{(i)^\prime} (\alpha_k^{(i)})$,
        \item $\hat{\kappa}_{k,\ell,n}^{(i)} - \sum\limits_{j=1}^p ( u_{j,k,n}^{(i)} ) ^2( u_{j,\ell,n}^{(i)} )^2\conp 0$,
        \item $\hat{\gamma}_{4,n}^{(i)} \conp \gamma_4^{(i)}$.
    \end{enumerate}
\end{lemma}
The consistency of the estimators for the supercritical eigenvalues in \textit{(a)} is provided in Theorem 3.1 in \cite{baiding2012}. As a consequence of the continuous mapping theorem, \textit{(b)} follows from \textit{(a)}.
The assertions \textit{(c)} and \textit{(d)} are given in Theorem 2.6 in
\cite{zhang2022asymptotic}. The estimator for the kurtosis was proposed by \cite{dornemanndette2024} for the case $p/n \in (0,1)$. A closer examination of the proof reveals that the consistency of this estimator remains valid in the case $p/n \in (0,\infty)$.

Now we are in the position to prove Lemma \ref{lem_consistent_estimators}.
\begin{proof}[Proof of Lemma \ref{lem_consistent_estimators}]
By the continuous mapping theorem, Theorem $2$ in \cite{li_and_chen_2012} implies that
\begin{align*}
    \frac{\sigma_{n,1}}{\hat{\sigma}_{n,1}} \conp 1, \quad n\to \infty
\end{align*}
holds under $\mathsf{H}_{0}$ while
\begin{align*}
    \frac{2}{n}\frac{\tr(\bfSigma_n^{(1)^2}+\bfSigma_n^{(2)^2})}{\hat{\sigma}_{n,1}}\conp 1, \quad n\to \infty
\end{align*}
holds under $\mathsf{H}_{1}$. 
\newline
We note, by utilizing the results Lemma \ref{lem_kurtosis_supercrit_estimators} along with the continuous mapping theorem, it follows that
\begin{align*}
    &\hat{\sigma}_{n,2}^2 - \sigma_{n,2}^2 \conp 0,
    \\
    &(\hat{\bfSigma}_{E,m,n})_{k, \ell} - (\bfSigma_{E,m})_{k, \ell} \conp 0.
\end{align*}
Due to Lemma \ref{lem_lower_bounded_variances}, we emphasize that $\hat{\sigma}_{n,2}^2$ is ratio-consistent to the true variance $\sigma_{n,2}^2$. In particular,
\begin{align*}
    \frac{\sigma_{n,2}^2}{\hat{\sigma}_{n,2}^2}\conp 1, \quad n \to \infty ~.
\end{align*}
\end{proof}

\subsection{Simulation}\label{sec_appendix_simulation}
In this section, we provide numerical results on the performance of the test \eqref{eq_generalized_rejection_region} for covariance matrices that have supercritical eigenvalues with multiplicity greater than $1$. Specifically, we consider
\begin{align}\label{eq_model_5}
\begin{split}
    \bfSigma_n^{(1)} &= \diag(10, \underbrace{7, \ldots, 7}_{10}, \underbrace{1, \ldots, 1}_{p-11}), 
    \\
    \bfSigma_n^{(2)}(\delta) &= \diag(10+\delta, 7+\delta, 7+\delta, \underbrace{7, \ldots, 7}_{8},\underbrace{1, \ldots, 1}_{p-11}) 
\end{split}
    \end{align}
    as covariance matrices, where $\delta=0$ corresponds to the null hypothesis in \eqref{eq_hypothesis}. Analogous to Model \eqref{eq_model_1}, $\alpha_1=10+\delta$, $\alpha_2=\alpha_3=7+\delta$ and $\alpha_4=\ldots = \alpha_{10}=7$ are for $\delta \in \N_0$ supercritical eigenvalues. The comparison of the empirical rejection probabilities is displayed in Figure \ref{fig_empirical_rej_model_5}. We observe that in this model, all tests maintain the nominal level $\alpha=0.05$ under the null hypothesis, although the assumption \ref{ass_supercritical_ev_generalized} is not satisfied. Nonetheless, we can see from Figure \ref{fig_empirical_rej_model_5} that the test \eqref{eq_generalized_rejection_region} (dot dashed line) is the most powerful test among the four competitors. Beyond that, we can observe once again, that the proposed test in \cite{zhang2022asymptotic} loses considerable power, if the data is not normal distributed.
     \begin{figure}[p]
		\begin{center}	            
        \includegraphics[scale=0.36]{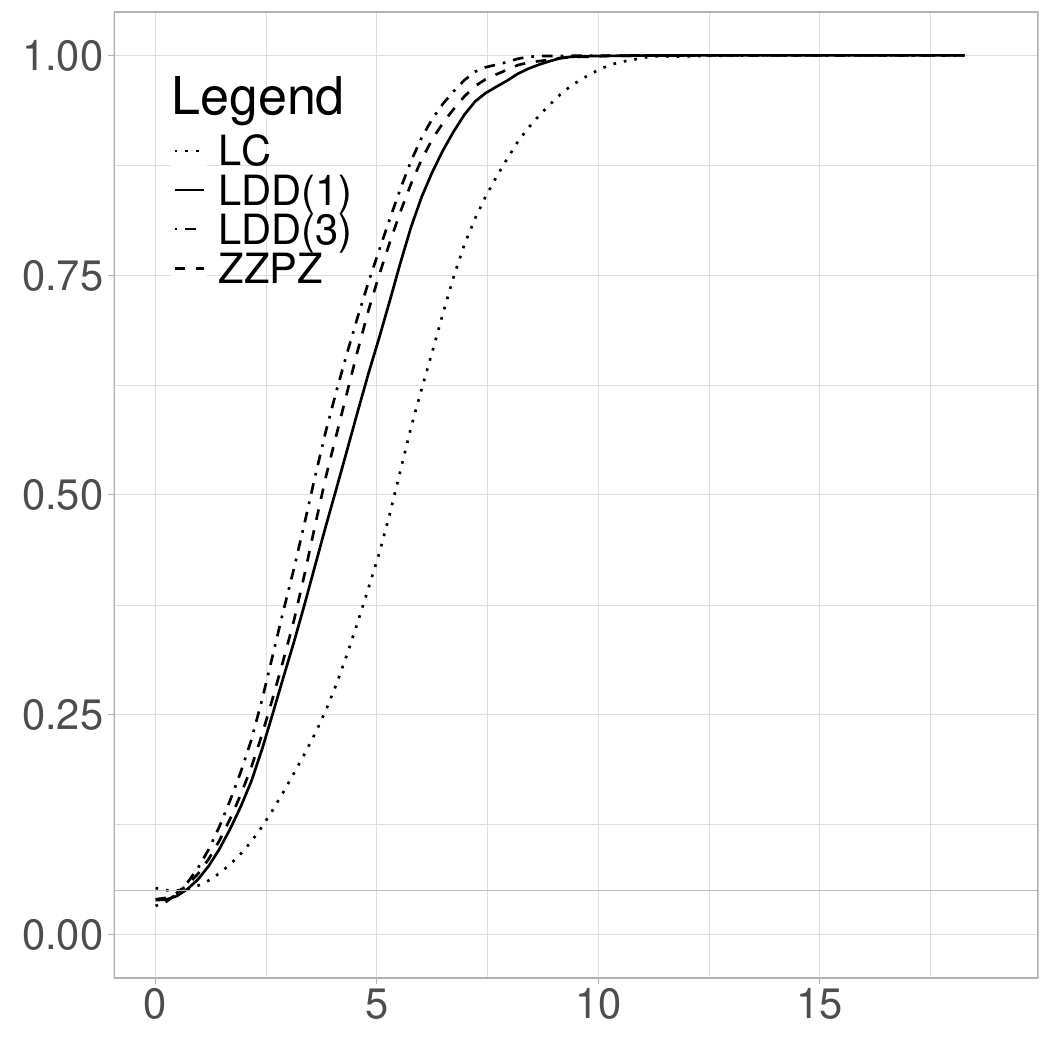}
         ~~~~ ~~~~
        \includegraphics[scale=0.36]{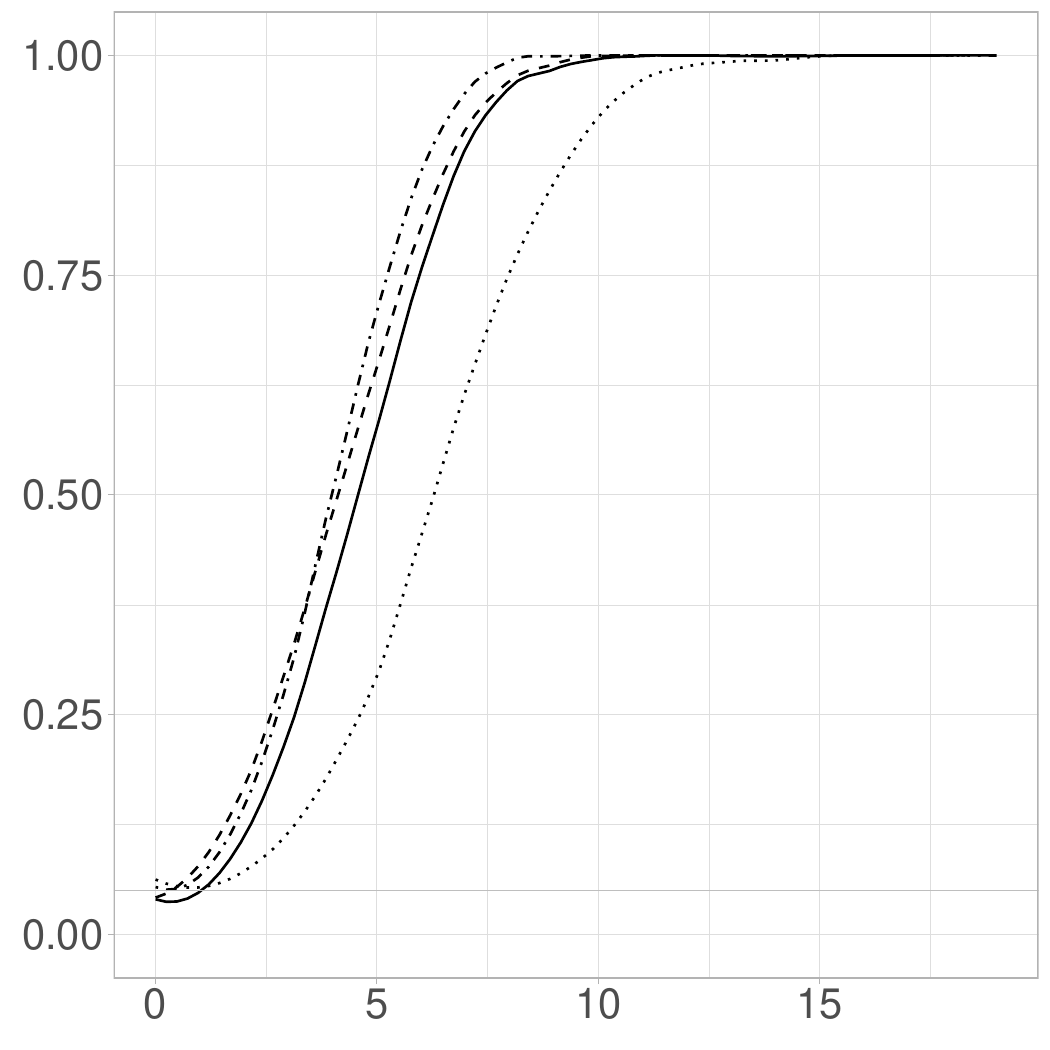}
        \includegraphics[scale=0.36]{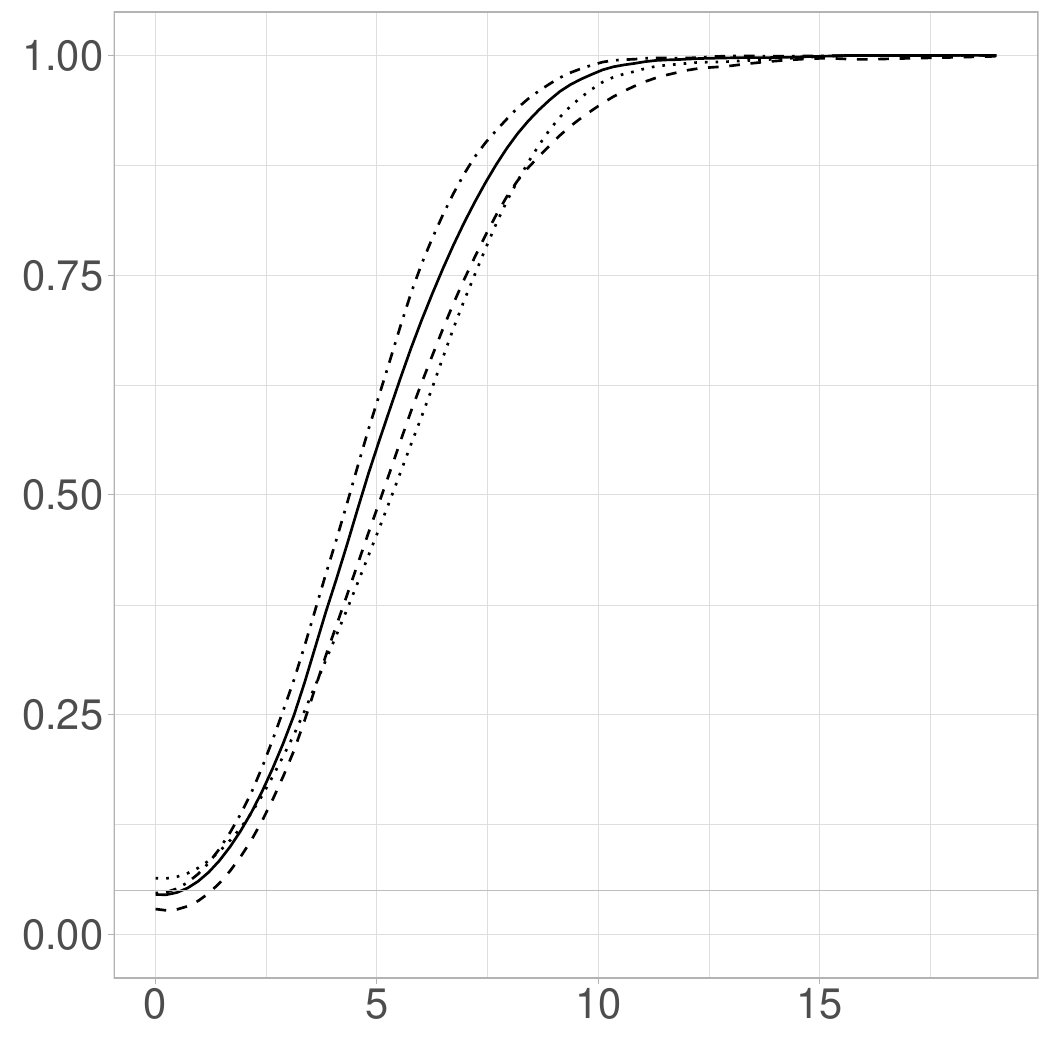}
         ~~~~ ~~~~
        \includegraphics[scale=0.36]{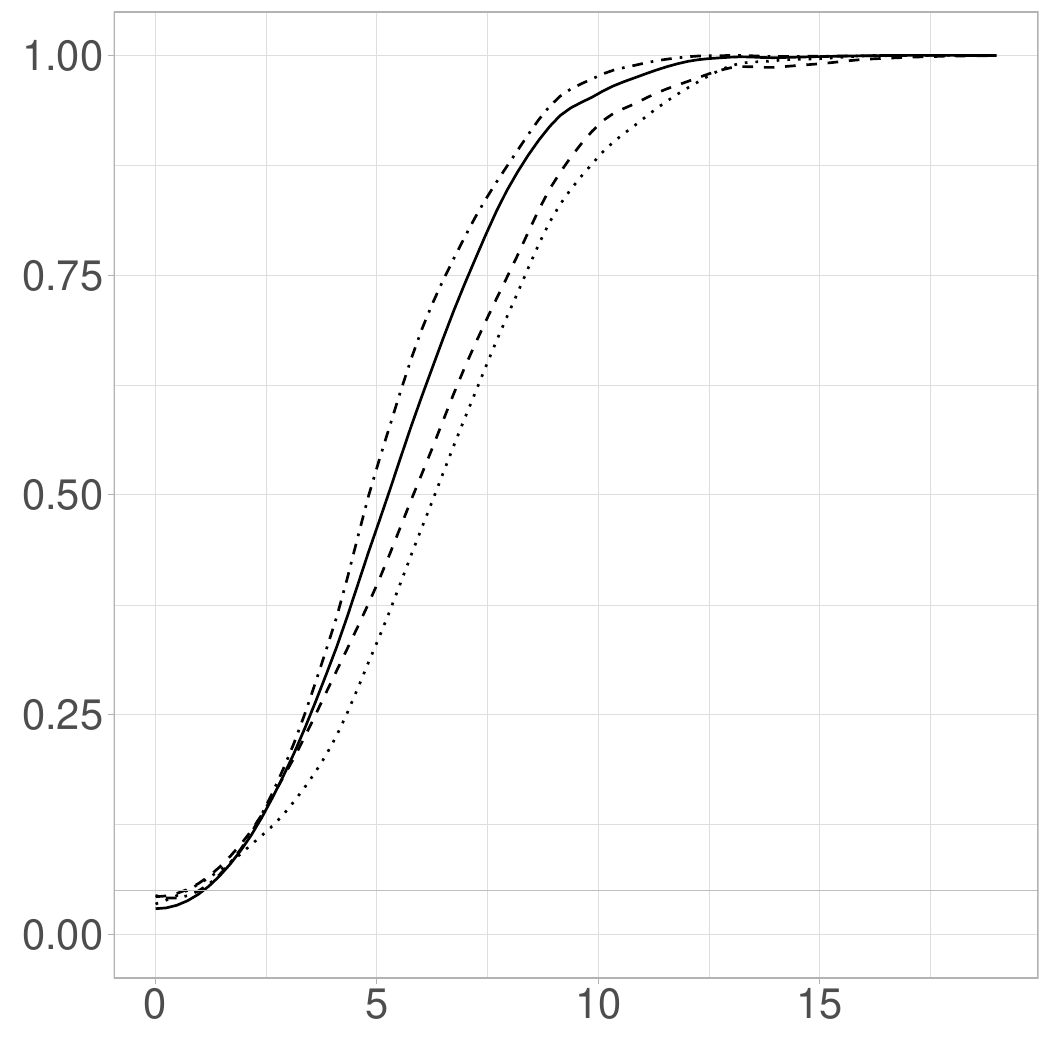}
        \includegraphics[scale=0.36]{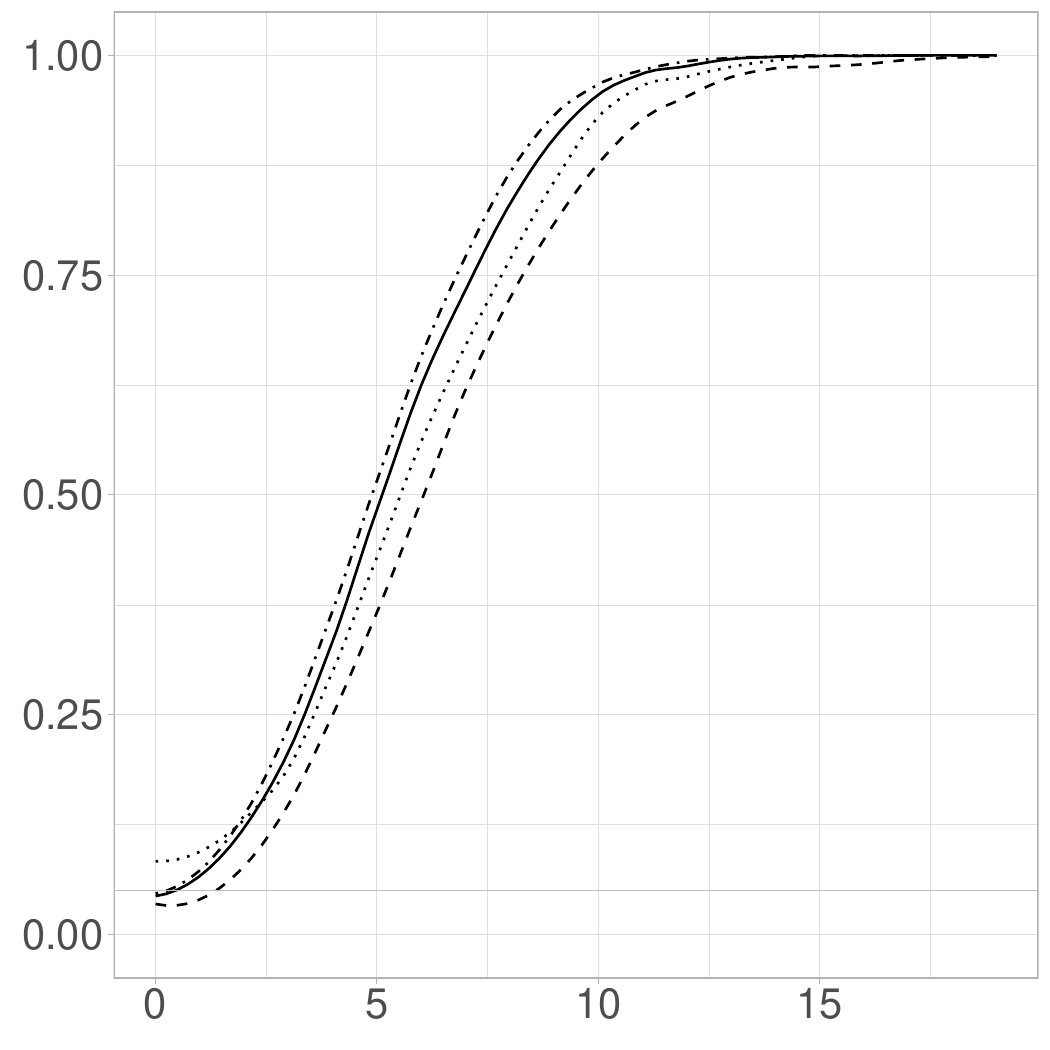}
         ~~~~ ~~~~
        \includegraphics[scale=0.36]{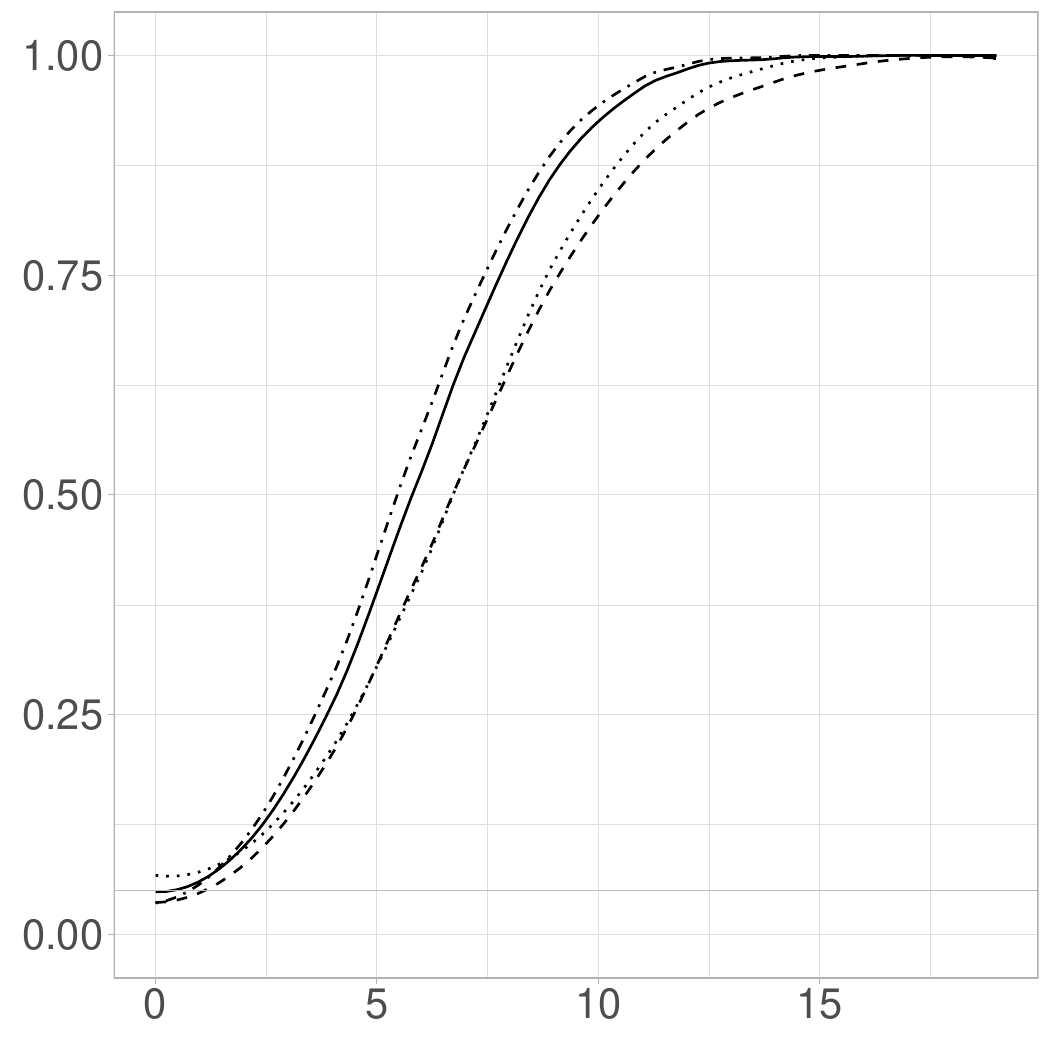}
		\end{center}
		\caption
        {\it Simulated rejection probabilities of the test \eqref{eq_rejection_region} (solid line, LDD(1)), the test of \eqref{eq_generalized_rejection_region} (dot dashed line, LDD(3)), the test of \cite{zhang2022asymptotic} (dashed line, ZZPZ) and the test of \cite{li_and_chen_2012} (dotted line, LC)  in model \eqref{eq_model_5}. Left panels:   $(p,n)=(500, 100)$. Right panels: $(p,n)=(1000, 100)$. First row: $\bfz_{1,1}^{(i)} \sim \mathcal{N}(0,1)$, second row: $\bfz_{1,1}^{(i)}  \sim t_{7}/\sqrt{7/5}$, third row: $\bfz_{1,1}^{(i)}  \sim \Laplace(0, 1/\sqrt{2})$.}\label{fig_empirical_rej_model_5}
	\end{figure}
\end{document}